\documentclass[11pt,oneside,english]{amsart}
\usepackage[a4paper]{geometry}
\usepackage[active]{srcltx}
\usepackage{babel}
\usepackage{enumitem}
\usepackage{slashed}
\usepackage{amstext}
\usepackage{amsthm}
\usepackage{amssymb}
\usepackage{setspace}
\usepackage[unicode=true,pdfusetitle,
 bookmarks=true,bookmarksnumbered=false,bookmarksopen=false,
 breaklinks=false,pdfborder={0 0 1},backref=false,colorlinks=false]
 {hyperref}

\makeatletter
\numberwithin{equation}{section}
\numberwithin{figure}{section}
\theoremstyle{plain}
\newtheorem{thm}{\protect\theoremname}[section]
\theoremstyle{plain}
\newtheorem{question}[thm]{\protect\questionname}
\theoremstyle{remark}
\newtheorem{rem}[thm]{\protect\remarkname}
\theoremstyle{definition}
\newtheorem{defn}[thm]{\protect\definitionname}
\theoremstyle{plain}
\newtheorem{prop}[thm]{\protect\propositionname}
\theoremstyle{definition}
\newtheorem{example}[thm]{\protect\examplename}
\theoremstyle{plain}
\newtheorem{cor}[thm]{\protect\corollaryname}
\theoremstyle{definition}
\newtheorem{problem}[thm]{\protect\problemname}
\theoremstyle{remark}
\newtheorem*{acknowledgement*}{\protect\acknowledgementname}
\theoremstyle{plain}
\newtheorem{lem}[thm]{\protect\lemmaname}

\usepackage{babel}
\usepackage{tikz}
\usetikzlibrary{shapes,arrows,positioning}

\newtheorem{theorem}{Theorem}
\providecommand{\theoremname}{Theorem}

\makeatother

\providecommand{\acknowledgementname}{Acknowledgement}
\providecommand{\corollaryname}{Corollary}
\providecommand{\definitionname}{Definition}
\providecommand{\examplename}{Example}
\providecommand{\lemmaname}{Lemma}
\providecommand{\problemname}{Problem}
\providecommand{\propositionname}{Proposition}
\providecommand{\questionname}{Question}
\providecommand{\remarkname}{Remark}
\providecommand{\theoremname}{Theorem}

\begin{document}
\title{Integrability and singularities of Harish-Chandra characters}
\author{Itay Glazer, Julia Gordon and Yotam I. Hendel}
\begin{abstract}
Let $G$ be a reductive group over a local field $F$ of characteristic
$0$. By Harish-Chandra's regularity theorem, the character $\Theta_{\pi}$
of an irreducible, admissible representation $\pi$ of $G$ is given
by a locally integrable function $\theta_{\pi}$ on $G$. It is a
natural question whether $\theta_{\pi}$ has better integrability
properties, namely, whether it is locally $L^{1+\epsilon}$-integrable
for some $\epsilon>0$. It turns out that the answer is positive,
and this gives rise to a new singularity invariant of representations
$\epsilon_{\star}(\pi):=\sup\left\{ \epsilon:\theta_{\pi}\in L_{\mathrm{Loc}}^{1+\epsilon}(G)\right\} $,
which we explore in this paper.

We provide a lower bound on $\epsilon_{\star}(\pi)$ which depends
only on the absolute root system of $G$, and explicitly determine
$\epsilon_{\star}(\pi)$ in the case of a $p$-adic $\mathrm{GL}_{n}$.
This is done by studying integrability properties of the Fourier transforms
$\widehat{\xi}_{\mathcal{O}}$ of stable Richardson nilpotent orbital
integrals $\xi_{\mathcal{O}}$. We express $\epsilon_{\star}(\widehat{\xi}_{\mathcal{O}})$
as the log-canonical threshold of a suitable relative Weyl discriminant,
and use a resolution of singularities algorithm coming from the theory
of hyperplane arrangements, to compute it in terms of the partition
associated with the orbit.

We obtain several applications; firstly, we provide bounds on the
multiplicities of $K$-types in irreducible representations of $G$
in the $p$-adic case, where $K$ is an open compact subgroup. We
further obtain bounds on the multiplicities of the irreducible representations
appearing in the space $L^{2}(K/L)$, where $K$ is a compact simple
Lie group, and $L\leq K$ is a Levi subgroup. Finally, we discover
surprising applications in random matrix theory, namely to the study
of the eigenvalue distribution of powers of random unitary matrices. 
\end{abstract}

\maketitle
\global\long\def\N{\mathbb{\mathbb{N}}}%
\global\long\def\R{\mathbb{\mathbb{R}}}%
\global\long\def\Z{\mathbb{\mathbb{Z}}}%
\global\long\def\val{\mathbb{\mathrm{val}}}%
\global\long\def\Qp{\mathbb{Q}_{p}}%
\global\long\def\Zp{\mathbb{Z}_{p}}%
\global\long\def\ac{\mathrm{ac}}%
\global\long\def\C{\mathbb{C}}%
\global\long\def\Q{\mathbb{Q}}%
\global\long\def\supp{\mathrm{supp}}%
\global\long\def\VF{\mathrm{VF}}%
\global\long\def\RF{\mathrm{RF}}%
\global\long\def\VG{\mathrm{VG}}%
\global\long\def\spec{\mathrm{Spec}}%
\global\long\def\orb{\mathrm{orb}}%
\global\long\def\rss{\mathrm{rss}}%
\global\long\def\mexp{\mathbf{e}}%
\global\long\def\GL{\mathrm{GL}}%
\global\long\def\U{\mathrm{U}}%
\global\long\def\SU{\mathrm{SU}}%
\global\long\def\Wg{\mathrm{Wg}}%
\global\long\def\Ldp{\mathcal{L}_{\mathrm{DP}}}%
\global\long\def\Irr{\mathrm{Irr}}%
\global\long\def\g{\mathfrak{g}}%
\global\long\def\l{\mathfrak{l}}%
\global\long\def\p{\mathfrak{p}}%
\global\long\def\n{\mathfrak{n}}%
\global\long\def\m{\mathfrak{m}}%
\global\long\def\t{\mathfrak{t}}%
\global\long\def\q{\mathfrak{q}}%
\global\long\def\gl{\mathfrak{gl}}%
\global\long\def\sl{\mathfrak{sl}}%
\global\long\def\ggm{\underline{\mathfrak{g}}}%
\global\long\def\llm{\underline{\mathfrak{l}}}%
\global\long\def\ppm{\underline{\mathfrak{p}}}%
\global\long\def\nnm{\underline{\mathfrak{n}}}%
\global\long\def\ttm{\underline{\mathfrak{t}}}%
\global\long\def\qqm{\underline{\mathfrak{q}}}%
\global\long\def\Lie{\mathrm{Lie}}%
\global\long\def\ad{\operatorname{ad}}%
\global\long\def\lct{\operatorname{lct}}%
\global\long\def\Ad{\operatorname{Ad}}%
\global\long\def\Int{\operatorname{Int}}%
\global\long\def\GG{\underline{G}}%
\global\long\def\PP{\underline{P}}%
\global\long\def\QQ{\underline{Q}}%
\global\long\def\LL{\underline{L}}%
\global\long\def\NN{\underline{N}}%
\global\long\def\sX{\mathsf{X}}%
\global\long\def\sY{\mathsf{Y}}%
\global\long\def\tr{\operatorname{tr}}%

\raggedbottom

\section{Introduction}

Let $F$ be a local field of characteristic zero, let $\GG$ be a
connected reductive algebraic group defined over $F$, where $G:=\GG(F)$
is the corresponding $F$-analytic group. Let $\mathrm{Rep}(G)$ be
the category of smooth, admissible representations of $G$ of finite
length, when $F$ is non-Archimedean, and the category of smooth,
admissible, finite length, Fr\'echet representations of $G$ of moderate
growth, when $F$ is Archimedean. Let $\Irr(G)$ be the set of equivalence
classes of irreducible representations in $\mathrm{Rep}(G)$.

Harish-Chandra's regularity theorem says that the character $\Theta_{\pi}$
of an irreducible representation $\pi\in\Irr(G)$, which is a priori
a distribution on $G$, is given by a locally integrable function
$\theta_{\pi}\in L_{\mathrm{Loc}}^{1}(G)$. It follows from the proof
of the theorem that $\theta_{\pi}$ is in fact in $L_{\mathrm{Loc}}^{1+\epsilon}(G)$
for some $\epsilon>0$ (this follows e.g.~from \cite[Theorem 16.3]{HC99}
and \cite[Theorem 15]{HC70}). It is natural to study a quantitative
version of the regularity theorem, namely:
\begin{question}
~\label{que:Main question}Let $\pi$ be an admissible representation
of $G$. 
\begin{enumerate}
\item What is the largest value $\epsilon>0$ such that $\theta_{\pi}\in L_{\mathrm{Loc}}^{1+\epsilon}(G)$?
\item Set $\epsilon_{\star}(\pi):=\sup\left\{ \epsilon:\theta_{\pi}\in L_{\mathrm{Loc}}^{1+\epsilon}(G)\right\} $.
What does the invariant $\epsilon_{\star}(\pi)$ tell us about the
representation $\pi$?
\end{enumerate}
\end{question}

In this paper we explore $\epsilon_{\star}(\pi)$. More generally,
for every distribution $\xi$ on $G$ represented by a locally $L^{1}$-function
$f_{\xi}$, and every $x\in G$, we define the \emph{integrability
exponent} 
\begin{equation}
\epsilon_{\star}(\xi;x):=\sup\left\{ s\geq0:\exists\text{ a ball }x\in B\subseteq G\text{ such that }f_{\xi}|_{B}\in L^{1+s}(G)\right\} ,\label{eq:invariant of representations}
\end{equation}
and set $\epsilon_{\star}(\xi):=\underset{x\in G}{\inf}\epsilon_{\star}(\xi;x)$.
Similarly, for $\pi\in\mathrm{Rep}(G)$, we define $\epsilon_{\star}(\pi,x):=\epsilon_{\star}(\theta_{\pi},x)$.

One of the motivations for defining $\epsilon_{\star}(\pi)$ comes
from motivic integration and tame geometry. In \cite[Theorems E, F]{GH21},
the first and third authors proved that any absolutely integrable
constructible function $f:F^{n}\rightarrow\R$, in the sense of Cluckers\textendash Loeser
\cite{CL08,CL10}, is in fact in $L^{1+\epsilon}(F)$ for some $\epsilon>0$.
A similar statement can be deduced for real constructible functions
from \cite{CM13}. This phenomenon is expected to hold for ``tame
enough'' functions in various other settings. Based on the above,
we defined in \cite{GHS}, an invariant similar to (\ref{eq:invariant of representations})
which measures integrability of pushforwards of smooth measures by
analytic maps. Since Harish-Chandra characters $\Theta_{\pi}$ are
of ``motivic nature'' (see \cite[Theorem 5.8]{CGH14a}), we expect
$\epsilon_{\star}(\pi)$ to be a meaningful invariant.

In contrast to the Gelfand\textendash Kirillov dimension $\mathrm{GK}(\pi)$
(Definition \ref{def:Gelfand-Kirillov dimension}), which measures
how \textbf{large} a given representation $\pi$ is, the quantity
$\epsilon_{\star}(\pi)$ measures how \textbf{singular }$\pi$ is,
where large values of $\epsilon_{\star}(\pi)$ correspond to representations
which are far from being generic. The invariant $\epsilon_{\star}(\pi)$
is controlled by the geometry and singularities of the nilpotent orbits
appearing in the local character expansion of $\Theta_{\pi}$, while
$\mathrm{GK}(\pi)$ only takes into account their dimensions. In particular,
there are examples of representations of the same GK dimension, but
very different values of $\epsilon_{\star}(\pi)$ and vice versa (see
Remark \ref{rem:on GK dimension and epsilon}(2)).

In order to state the main results of this paper, we first introduce
some notation, as well as a local description of $\Theta_{\pi}$.

\subsection{\label{subsec:Fourier-transform-of}Fourier transforms of orbital
integrals }

Let $\g$ be the Lie algebra of $G$. Denote by $\mathcal{S}(\g)$
(resp.~$\mathcal{S}(G)$) the space of Schwartz functions on $\g$
(resp.~$G$). Similarly, let $\mathcal{S}^{*}(\g)$ (resp.~$\mathcal{S}^{*}(G)$)
be the space of tempered distributions on $\g$ (resp.~$G$). Let
$\langle\,,\,\rangle$ be a non-degenerate $G$-invariant symmetric
bilinear form on $\g$. Then $\langle\,,\,\rangle$ gives an identification
$\g\simeq\g^{*}$, by $Y\longmapsto\langle\,,Y\rangle$. Fixing a
non-trivial additive character $\chi$ of $F$, we denote the Fourier
transform $\mathcal{F}:\mathcal{S}(\g)\rightarrow\mathcal{S}(\g)$
by $f\mapsto\widehat{f}$, where:
\[
\widehat{f}(Y):=\int_{\g}f(X)\chi\left(\left\langle X,Y\right\rangle \right)dX.
\]
Similarly, for a tempered distribution $\xi\in\mathcal{S}^{*}(G)$,
we set $\widehat{\xi}(f):=\xi(\widehat{f})$.

Let $\mathcal{N}_{\g}\subseteq\g$ be the set of nilpotent elements
in $\g$. Let $\mathcal{O}(\mathcal{N}_{\g})$ be the set of $\Ad(G)$-orbits
in $\mathcal{N}_{\g}$ (it is a finite set, see e.g.~\cite{Mor88}).
For any $X\in\mathcal{N}_{\g}$, the stabilizer $G_{X}$ is unimodular
\cite[Remark 3.27]{SS70}, and hence the corresponding $G$-orbit
$\mathcal{O}$ carries a $G$-invariant measure $\mu_{\mathcal{O}}$.
Since a non-zero nilpotent orbit $\mathcal{O}\subseteq\g$ is not
closed, the restriction $f|_{\mathcal{O}}$ of $f\in\mathcal{S}(\g)$
is not necessarily a Schwartz function. Nonetheless, Deligne and Ranga
Rao \cite{Rao72} showed that the integral $\int_{\mathcal{O}}f\mu_{\mathcal{O}}$
converges, and hence $\mu_{\mathcal{O}}$ gives rise to a tempered
distribution $\xi_{\mathcal{O}}\in\mathcal{S}^{*}(\g)$. We write
$\widehat{\xi}_{\mathcal{O}}\in\mathcal{S}^{*}(\g)$ for its Fourier
transform. It was shown by Harish-Chandra (\cite{HC65a,HC78}, see
also \cite[Theorem 27.8]{Kot05}) that the distribution $\widehat{\xi}_{\mathcal{O}}$
is given by a locally integrable function on $\g$, which is analytic
on $\g_{\rss}$, the set of regular semisimple elements of $\g$.
The local behavior of $\theta_{\pi}$ can be understood using the
Harish-Chandra\textendash Howe local character expansion in the non-Archimedean
case, and the Barbasch\textendash Vogan asymptotic expansion in the
Archimedean case: 
\begin{thm}[{\cite{How74}, \cite[Theorem 16.2]{HC99} and \cite[Theorem 4.1]{BV80}}]
\label{thm:local character expansion}Let $\pi\in\Irr(G)$. 
\begin{enumerate}
\item If $F$ is non-Archimedean (of characteristic $0$), then there exist
constants $\left\{ c_{\mathcal{O}}(\pi)\right\} _{\mathcal{O}\in\mathcal{O}(\mathcal{N}_{\g})}$
and an $F$-analytic neighborhood $U$ of zero in $\g$, such that
for all $X\in\g_{\rss}\cap U$: 
\begin{equation}
\theta_{\pi}(\exp(X))=\sum_{\mathcal{O}\in\mathcal{O}(\mathcal{N}_{\g})}c_{\mathcal{O}}(\pi)\cdot\widehat{\xi}_{\mathcal{O}}(X).\label{eq:local character expansion non-Archimedean}
\end{equation}
\item If $F$ is Archimedean, there exist constants $\left\{ a_{\mathcal{O}}(\pi)\right\} _{\mathcal{O}\in\mathcal{O}(\mathcal{N}_{\g})}$,
such that $t^{-\dim\g}\theta_{\pi}(t^{-1}X)$ admits an asymptotic
expansion\footnote{See \cite[Section 2]{BV80} for the precise meaning of asymptotic
expansion.}, as $t\rightarrow0$, whose leading term is 
\[
\sum_{\mathcal{O}\in\mathcal{O}(\mathcal{N}_{\g}):\dim\mathcal{O}=\mathrm{GK}(\pi)}a_{\mathcal{O}}(\pi)\cdot\widehat{\xi}_{\mathcal{O}}.
\]
\end{enumerate}
\end{thm}

\begin{rem}
A similar expansion of $\theta_{\pi}$ can be obtained near any semisimple
element of $\g$ (for a precise formulation, see e.g.~\cite[Theorem 16.2]{HC99}). 
\end{rem}

\subsection{\label{subsec:Main-results}Main results}

In view of Theorem \ref{thm:local character expansion}, the study
of $\epsilon_{\star}(\pi)$ goes through the integrability of $\widehat{\xi}_{\mathcal{O}}$.
The description of $\epsilon_{\star}(\widehat{\xi}_{\mathcal{O}})$
will be in terms of a singularity invariant called the \emph{log-canonical
threshold }(see e.g.~\cite{Kol,Mus12} for more details about this
invariant). 
\begin{defn}
\label{def:log canonical threshold }The\emph{ $F$-log-canonical
threshold} of an analytic function $f:F^{n}\to F$,\emph{ }at $x\in F^{n}$
is:\footnote{$\left|\,\cdot\,\right|_{F}$ is the absolute value on $F$, normalized
so that $\mu_{F}(aS)=\left|a\right|_{F}\cdot\mu_{F}(S)$, for all
$a\in F^{\times}$, $S\subseteq F$, where $\mu_{F}$ is a Haar measure
on $F$.} 
\[
\lct_{F}(f;x):=\sup\left\{ s>0:\exists\text{ a ball }x\in B\subseteq F^{n}\text{ s.t.}\int_{B}\left|f(x)\right|_{F}^{-s}dx<\infty\right\} .
\]
\end{defn}

Let $T\subseteq G$ be a maximal $F$-torus, with $\t=\Lie(T)$. Let
$\Sigma(\g,\t)$ be the corresponding root system, with $\Sigma^{+}(\g,\t)\subseteq\Sigma(\g,\t)$
the set of positive roots. Let $\p\subseteq\g$ be an $F$-parabolic
subalgebra, with Levi $\l\supseteq\t$. We denote by $D_{\g}:\g\rightarrow F$
and $D_{\l}:\l\rightarrow F$ the Weyl discriminants of $\g$ and
$\l$, respectively (see $\mathsection$\ref{subsec:Weyl-discriminants}).
Define the relative Weyl discriminant $D_{\g,\l}:\l\rightarrow F$
on $\l$ by: 
\begin{equation}
D_{\g,\l}(X):=D_{\g}(X)/D_{\l}(X),\,\,\,X\in\l.\label{eq:relative Weyl discriminant}
\end{equation}
If $\g$ is $F$-split, and $\t$ is an $F$-split Cartan, we have
\begin{equation}
D_{\g}|_{\t}=(-1)^{\left|\Sigma^{+}(\g,\t)\right|}(\triangle_{\g}^{\t})^{2}\text{\,\,\,\, and \,\,\,\,}D_{\l}|_{\t}=(-1)^{\left|\Sigma^{+}(\l,\t)\right|}(\triangle_{\l}^{\t})^{2},\label{eq:reduction to hyperplane arrangement}
\end{equation}
where $\triangle_{\g}^{\t},\triangle_{\l}^{\t}:\t\rightarrow F$ are
given by: 
\begin{equation}
\triangle_{\g}^{\t}(X):=\prod_{\alpha\in\Sigma^{+}(\g,\t)}\alpha(X)\text{ and }\triangle_{\l}^{\t}(X):=\prod_{\alpha\in\Sigma^{+}(\l,\t)}\alpha(X).\label{eq:Weyl discriminant on Cartan}
\end{equation}
Similarly, we set $\triangle_{\g,\l}^{\t}(X):=\triangle_{\g}^{\t}(X)/\triangle_{\l}^{\t}(X)$.
We can now state our first main theorem, which focuses on stable Richardson
nilpotent orbits (see Definition \ref{def:Richardson orbit}).

\begin{theorem}\label{thm A: epsilon of Fourier of orbital integral is lct(D(G/M))}Let
$\underline{G}$ be a connected reductive algebraic $F$-group. Let
$\mathcal{O}_{\mathrm{st}}$ be a stable Richardson orbit in $\g$,
with a polarization $\p=\l\oplus\n$. Then 
\begin{equation}
\epsilon_{\star}(\widehat{\xi}_{\mathcal{\mathcal{O}_{\mathrm{st}}}})=\epsilon_{\star}(\widehat{\xi}_{\mathcal{\mathcal{O}_{\mathrm{st}}}};0)=2\lct_{F}(D_{\g,\l}(X);0).\label{eq:formula for epsilon of Richardson}
\end{equation}
\end{theorem}

When analyzing (\ref{eq:formula for epsilon of Richardson}), we first
consider the case that $\g$ is $F$-split, and $\l=\t$ is an $F$-split
Cartan, and compute the $F$-log-canonical threshold of the Weyl discriminant
$D_{\g,\t}=D_{\g}|_{\t}$. We set $h_{\g}:=\frac{\left|\Sigma(\g,\t)\right|}{\mathrm{rk}(\g)}$
to be the Coxeter number of $\Sigma(\g,\t)$.
\begin{prop}[Proposition \ref{prop:lct for simple}]
\label{prop:Lct of Weyl discriminant}Suppose that $\underline{G}$
is an $F$-split, simple algebraic group. Then: 
\[
\lct_{F}(\triangle_{\g}^{\t};0)=2\lct_{F}(D_{\g,\t};0)=\frac{2}{h_{\g}}=\frac{\mathrm{rk}(\g)}{\left|\Sigma^{+}(\g,\t)\right|}.
\]
\end{prop}

The notion of $\epsilon_{\star}(\widehat{\xi};X)$ can be extended
to a larger family of invariant distributions $\xi$ as follows. Denote
by $\mathcal{S}^{*}(\g)^{G}$ the space of $G$-invariant distributions
on $\g$, and by $\mathcal{S}_{\mathrm{HC}}^{*}(\g)^{G}\subseteq\mathcal{S}^{*}(\g)^{G}$
the subspace consisting of: 
\begin{itemize}
\item Distributions whose Fourier transform has bounded support modulo conjugation
by $G$, if $F$ is non-Archimedean. 
\item $\mathcal{Z}(\g_{\C})$-eigendistributions, where $\mathcal{Z}(\g_{\C})$
is the center of the universal enveloping algebra $\mathcal{U}(\g_{\C})$
of $\g_{\C}$, if $F$ is Archimedean. 
\end{itemize}
By results of Harish-Chandra \cite[Theorem 4.4]{HC99}, \cite{HC65a},
it follows that any $\xi\in\mathcal{S}_{\mathrm{HC}}^{*}(\g)^{G}$
is given by a locally integrable function $f_{\xi}\in L_{\mathrm{Loc}}^{1}(\g)$
and $\left|D_{\g}\right|_{F}^{\frac{1}{2}}\cdot f_{\xi}$ is locally
bounded. Furthermore, by \cite[Theorem 16.3]{HC99} and \cite[Theorem 2]{HC65b}
the function $\left|D_{G}\right|_{F}^{\frac{1}{2}}\cdot\theta_{\pi}$
is locally bounded for every $\pi\in\Irr(G)$, where $D_{G}$ is the
Weyl discriminant of $G$. Using Proposition \ref{prop:Lct of Weyl discriminant}
and the Weyl integration formula, we prove (in $\mathsection$\ref{sec:Integrability-of-global characters}):

\begin{theorem}\label{thmB:lower bound on epsilon of reps}Let $\underline{G}$
be a connected reductive algebraic $F$-group, with $G:=\GG(F)$ and
$\g:=\Lie(G)$. Let $\g_{1},\dots,\g_{M}$ be the simple constituents
of $[\g_{\overline{F}},\g_{\overline{F}}]$. 
\begin{enumerate}
\item Let $\xi$ be a distribution in $\mathcal{S}_{\mathrm{HC}}^{*}(\g)^{G}$.
Then 
\[
\epsilon_{\star}(\xi)\geq\underset{1\leq i\leq M}{\min}\frac{2}{h_{\mathfrak{g_{i}}}}.
\]
\item For every $\pi\in\mathrm{Rep}(G)$, we have: 
\[
\epsilon_{\star}(\pi)\geq\underset{1\leq i\leq M}{\min}\frac{2}{h_{\mathfrak{g_{i}}}}.
\]
\end{enumerate}
\end{theorem} 
\begin{rem}
Note that the lower bounds in Theorem \ref{thmB:lower bound on epsilon of reps}
depend \textbf{only on the absolute root system of} $\underline{G}$. 
\end{rem}

The function $\triangle_{\g}^{\t}:\t\rightarrow F$ is the defining
polynomial of a \textbf{hyperplane arrangement of finite Coxeter type}
(see $\mathsection$\ref{subsec:Hyperplanes-arrangements--basic}).
The fact that $\lct_{F}(\triangle_{\g}^{\t};0)=\frac{2}{h_{\g}}$
for $F$ Archimedean follows from conjectures by Macdonald \cite[Conjectures 2.1'' and 6.1]{Mac82}
and Sekiguchi\textendash Yano \cite[Conjecture 5.II]{SY79}, which
were proved by Opdam \cite{Opd89,Opd93} (see \cite{DL95} for a more
general setting). A similar result for non-Archimedean $F$ is probably
known by the experts. In Proposition \ref{prop:Lct of Weyl discriminant}
we give an alternative, elegant proof of this result (see Proposition
\ref{prop:lct for simple} and the preceding discussion in $\mathsection$\ref{def:The-Coxeter-number}
and $\mathsection$\ref{subsec:Evaluation-of-some}). We refer to
a survey by Budur \cite[2.5-2.6]{Bud12} for further discussion about
singularities of hyperplane arrangements, and relation to other conjectures,
such as the strong monodromy conjecture (see e.g.~\cite{BMT11}).\medskip{}

We now turn our attention to $\GG=\GL_{n}$. In this case we have:
\begin{enumerate}
\item Every nilpotent $\GL_{n}(F)$-orbit in $\gl_{n}(F)$ is Richardson
(see e.g.~\cite[Theorem 7.2.3]{CM93}) and stable (since both $\mathcal{O}(\mathcal{N}_{\gl_{n}(\overline{F})})$
and $\mathcal{O}(\mathcal{N}_{\gl_{n}(F)})$ are classified by the
Jordan normal form, and hence in bijection with the set of partitions
of $n$). 
\item The combinatorial description of the nilpotent orbits in $\gl_{n}(F)$
is simpler than in other classical Lie algebras. 
\end{enumerate}
Items (1) and (2) allow us to give a full description of $\epsilon_{\star}(\pi)$
and $\epsilon_{\star}(\widehat{\xi}_{\mathcal{O}})$ for $\GL_{n}$
(see $\mathsection$\ref{subsec:Proof-of-Theorem A} and Remark \ref{rem:Rango Rao}
for a discussion on the difficulties in the non-Richardson and non-stable
cases). Since the geometry of the nilpotent orbits in $\GL_{n}$ is
already very rich, it is an excellent model case for studying Question
\ref{que:Main question}

By analyzing $\lct_{F}(D_{\gl_{n},\l}(X);0)$ (see Propositions \ref{Prop:lct of relative Weyl discriminant}
and \ref{prop:relating lct to relative lct in split groups}), and
using Theorem \ref{thm A: epsilon of Fourier of orbital integral is lct(D(G/M))},
we provide an explicit formula for $\epsilon_{\star}(\widehat{\xi}_{\mathcal{O}})$,
as well as the behavior of $\epsilon_{\star}(\widehat{\xi}_{\mathcal{O}})$
with respect to orbit closures.

\begin{theorem}\label{thmC: epsilon of Fourier of orbital integral for GL_n}Let
$\mathcal{O}=\mathcal{O}_{\nu}$ be a nilpotent orbit in $\gl_{n}(F)$
whose Jordan normal form has blocks of size $\nu_{1}\geq\dots\geq\nu_{M}$
and set $n_{k}:=\sum_{i=1}^{k}\nu_{i}$. Let $\p=\l\oplus\n$ be a
polarization of $\mathcal{O}$. 
\begin{enumerate}
\item If $\mathcal{O}$ is the zero orbit, then $\epsilon_{\star}(\widehat{\xi}_{\mathcal{O}})=\infty$.
Otherwise, we have 
\begin{equation}
\epsilon_{\star}(\widehat{\xi}_{\mathcal{O}_{\nu}})=2\lct_{F}(D_{\g,\l};0)=\min_{1\le k\le M}\frac{\left(\sum_{j=1}^{k}j\nu_{j}\right)-1}{\binom{n_{k}}{2}-\sum_{j=1}^{k}(j-1)\nu_{j}}.\label{eq:formula for epsilon of orbit}
\end{equation}
\item If $\mathcal{O}'\subseteq\overline{\mathcal{O}}$ then $\epsilon_{\star}(\widehat{\xi_{\mathcal{O}'}})\geq\epsilon_{\star}(\widehat{\xi}_{\mathcal{O}})$. 
\end{enumerate}
\end{theorem}
\begin{rem}
\label{rem:Geometric formula}The formula in Equation \ref{eq:formula for epsilon of orbit}
above can be given a geometric interpretation as follows. Keep the
notation of Theorem \ref{thmC: epsilon of Fourier of orbital integral for GL_n},
let $[\nu]_{k}:=(\nu_{1},\ldots,\nu_{k})$ denote the partition obtained
by considering the first $k$ parts of $\nu$, let $\mathcal{O}_{[\nu]_{k}}$
denote the nilpotent orbit in $\gl_{n_{k}}$ whose Jordan normal form
is $[\nu]_{k}$, let $\p_{k}=\l_{k}\oplus\n_{k}$ be a polarization
of $\mathcal{O}_{[\nu]_{k}}$, and let $\n_{\l_{k}}$ be a maximal
nilpotent subalgebra of $\l_{k}$. Then,
\[
\epsilon_{\star}(\widehat{\xi}_{\mathcal{O}_{\nu}})=\min_{1\le k\le M}\frac{\mathrm{ss.rk}(\gl_{n_{k}})+\dim\n_{\l_{k}}}{\dim\n_{k}}=\min_{1\le k\le M}\frac{2\cdot\mathrm{ss.rk}(\gl_{n_{k}})+\left(\dim\mathcal{N}_{\gl_{n_{k}}}-\dim\mathcal{O}_{[\nu]_{k}}\right)}{\dim\mathcal{O}_{[\nu]_{k}}}.
\]
\end{rem}

\begin{example}
\label{exa:some examples of epsilon}It is worth analyzing two extreme
cases: 
\begin{enumerate}
\item If $\mathcal{O}$ is the regular orbit $\mathcal{O}_{\mathrm{reg}}$
(resp.~subregular orbit $\mathcal{O}_{\mathrm{sreg}}$), then its
Jordan normal form corresponds to the partition $\nu=(n)$ (resp.~$\nu=(n-1,1)$).
By (\ref{eq:formula for epsilon of orbit}) we have 
\begin{equation}
\epsilon_{\star}(\widehat{\xi}_{\mathcal{\mathcal{O}_{\mathrm{reg}}}})=\frac{2}{n}\text{ \,\,\,\,\,\, }\epsilon_{\star}(\widehat{\xi}_{\mathcal{\mathcal{O}_{\mathrm{sreg}}}})=\frac{2}{n-1}.\label{eq:regular+subreg orbit on GL_n}
\end{equation}
Since $\mathcal{\mathcal{O}_{\mathrm{reg}}}$ has a polarization with
Levi $\l=\t$, Eq.~(\ref{eq:regular+subreg orbit on GL_n}) agrees
with Theorem \ref{thm A: epsilon of Fourier of orbital integral is lct(D(G/M))}
and Proposition \ref{prop:Lct of Weyl discriminant}. 
\item Let $\mathfrak{p}_{1},\dots,\p_{n-1}\subseteq\gl_{n}$ be the maximal
standard parabolic subalgebras, such that $\p_{j}$ has Levi $\gl_{n-j}\times\gl_{j}$.
Using (\ref{eq:formula for epsilon of orbit}), it can be shown (see
Example \ref{exam:analyzing maximal parabolics}) that for the orbit
$\mathcal{O}_{j}$ with polarization $\p_{j}$, 
\begin{equation}
\epsilon_{\star}(\widehat{\xi}_{\mathcal{O}_{j}})=1.\label{eq:formula for maximal parabolics}
\end{equation}
In other words, for all $j$, $\widehat{\xi}_{\mathcal{O}_{j}}$ just
fails to be in $L_{\mathrm{loc}}^{2}(\gl_{n})$. In particular, since
$\mathcal{O}_{1}$ is the minimal orbit, it follows from Theorem \ref{thmC: epsilon of Fourier of orbital integral for GL_n}(2)
that $\epsilon_{\star}(\widehat{\xi}_{\mathcal{O}})\leq1$ for every
$\{0\}\neq\mathcal{O}\in\mathcal{O}(\mathcal{N}_{\gl_{n}})$.
\end{enumerate}
\end{example}

\begin{rem}
\label{rem: no simple formula for epsilon}As evident from (\ref{eq:formula for epsilon of orbit}),
given a partition corresponding to a nilpotent orbit $\mathcal{O}$
in $\gl_{n}$, the integrability exponent $\epsilon_{\star}(\widehat{\xi}_{\mathcal{O}})$
can be explicitly calculated. In Figure \ref{fig: epsilon calculation for gl10}
we include the values of $\epsilon_{\star}(\widehat{\xi}_{\mathcal{O}})$
for all nilpotent orbits in $\gl_{10}$, together with their orbit
closure relations. Observing (\ref{eq:regular+subreg orbit on GL_n})
and (\ref{eq:formula for maximal parabolics}), one could naively
guess that $\epsilon_{\star}(\widehat{\xi}_{\mathcal{O}})$ is equal
to $2/\nu_{1}$, where $\nu$ is the partition of the Jordan normal
form of $\mathcal{O}$. While this is true for $\gl_{n}$ for $n\le7$,
this fails for $n\geq8$. Taking Figure \ref{fig: epsilon calculation for gl10}
(and additional calculations) into account, it seems that the formula
given in Theorem \ref{thmC: epsilon of Fourier of orbital integral for GL_n}
can probably not be simplified much further. 
\end{rem}

\smallskip{}

We now move from Fourier transforms of orbital integrals to characters
of representations. By \cite{MW87}, when $\GG=\GL_{n}$, there is
a unique orbit $\mathcal{O}_{\mathrm{max}}$ in $\mathcal{N}_{\gl_{n}}$
appearing in the local character expansion of $\Theta_{\pi}$ (i.e.~we
have $c_{\pi}(\mathcal{O}_{\mathrm{max}})\neq0$ in (\ref{eq:local character expansion non-Archimedean})),
such that all other orbits $\mathcal{O}$ appearing in the character
expansion are contained in $\overline{\mathcal{O}}_{\mathrm{max}}$.
We further show the following:

\begin{theorem}\label{thm D:epsilon of a representation}Let $F$
be a non-Archimedean local field of characteristic $0$. Let $\pi\in\Irr(\GL_{n}(F))$
and let $\mathcal{O}_{\mathrm{max}}$ be the maximal orbit as above.
Then, 
\[
\epsilon_{\star}(\pi;e)=\epsilon_{\star}(\widehat{\xi}_{\mathcal{O}_{\mathrm{max}}}).
\]

\end{theorem}

Finally, we discuss several applications of the above results. 

\subsubsection{\label{subsec:Upper-bounds-on}Upper bounds on the multiplicities
of $K$-types in $p$-adic groups }

Let $\underline{G}$ be a reductive algebraic $F$-group, and $G=\underline{G}(F)$.
Let $\pi\in\mathrm{Rep}(G)$, let $K\leq G$ be an open compact subgroup,
and let $\pi|_{K}=\sum_{\tau\in\Irr(K)}m_{\tau,\pi}\tau$ be the decomposition
of $\pi$ into irreducible representations of $K$, with $m_{\tau,\pi}:=\dim\mathrm{Hom}_{K}\left(\tau,\pi|_{K}\right)$.
Then we have the following equality of distributions (see Lemma \ref{lem:restiction to maximal compact}):
\[
\Theta_{\pi}|_{K}=\sum_{\tau\in\Irr(K)}m_{\tau,\pi}\Theta_{\tau}.
\]
By a generalization of the classic Hausdorff\textendash Young Theorem
to the setting of compact groups, the $L^{1+\epsilon}$-integrability
of $\Theta_{\pi}|_{K}$ can be translated in the (non-commutative)
Fourier picture, into upper bounds on the multiplicities $m_{\tau,\pi}$.
More precisely, we have the following upper bound on multiplicities
of $K$-types:

\begin{theorem}\label{thmE:upper bounds on multiplicities of K-types}Let
$\pi\in\mathrm{Rep}(G)$ and let $K\leq G$ be an open compact subgroup.
Then for every $\epsilon<\epsilon_{\star}(\pi)$, there exists $C=C(\epsilon,K,\pi)>0$
such that for every $\tau\in\Irr(K)$: 
\[
\dim\mathrm{Hom}_{K}\left(\tau,\pi|_{K}\right)<C\left(\dim\tau\right)^{1-\frac{2\epsilon}{1+\epsilon}}.
\]

\end{theorem}

Applying Theorems \ref{thmC: epsilon of Fourier of orbital integral for GL_n}
and \ref{thmE:upper bounds on multiplicities of K-types}, in the
special case that $\pi=\mathrm{Ind}_{\PP(F)}^{\GL_{n}(F)}1$, we prove
the following:
\begin{cor}
\label{cor:upper bound for multiplicities in principal series}Let
$\PP\leq\GL_{n}$ be a standard parabolic subgroup, with Levi $\LL\simeq\LL_{\lambda}=\GL_{\lambda_{1}}\times\dots\times\GL_{\lambda_{N}}$
for a partition $\lambda=(\lambda_{1},\dots,\lambda_{N})$. Let $\nu=(\nu_{1},\dots,\nu_{M})$
be the conjugate partition to $\lambda$, defined by $\nu_{j}:=\#\left\{ i:\lambda_{i}\geq j\right\} $.
Then the following hold, for every $\epsilon>0$ satisfying: 
\[
\epsilon<\epsilon_{\star}(\widehat{\xi}_{\mathcal{O}_{\nu}})=\underset{1\leq k\leq M}{\min}\frac{\left(\sum_{j=1}^{k}j\nu_{j}\right)-1}{\binom{\sum_{j=1}^{k}\nu_{j}}{2}-\sum_{j=1}^{k}(j-1)\nu_{j}},
\]
\begin{enumerate}
\item The character $\Theta_{\mathrm{Ind}_{\PP(\mathcal{O}_{F})}^{\GL_{n}(\mathcal{O}_{F})}1}:\GL_{n}(\mathcal{O}_{F})\rightarrow\C$,
is in $L^{1+\epsilon}(\GL_{n}(\mathcal{O}_{F}))$, where $\mathcal{O}_{F}$
is the ring of integers of $F$. 
\item There exists $C=C(F,\epsilon,n)>0$ such that for every $\tau\in\Irr(\GL_{n}(\mathcal{O}_{F}))$,
we have: 
\[
\dim\mathrm{Hom}_{\GL_{n}(\mathcal{O}_{F})}\left(\tau,\mathrm{Ind}_{\PP(\mathcal{O}_{F})}^{\GL_{n}(\mathcal{O}_{F})}1\right)<C\left(\dim\tau\right)^{1-\frac{2\epsilon}{1+\epsilon}}.
\]
\end{enumerate}
\end{cor}

We end this subsection with a discussion on the relation between $\mathrm{GK}(\pi)$
and $\epsilon_{\star}(\pi)$. 
\begin{defn}
\label{def:Gelfand-Kirillov dimension}Let $K=\underline{G}(\mathcal{O}_{F})$.
For each $r\in\Z_{\geq1}$, let $K_{r}$ be the kernel of the map
$\underline{G}(\mathcal{O}_{F})\rightarrow\underline{G}(\mathcal{O}_{F}/\mathfrak{m}_{F}^{r})$,
where $\mathfrak{m}_{F}$ is the maximal ideal in $\mathcal{O}_{F}$,
with $q_{F}:=\#\mathcal{O}_{F}/\mathfrak{m}_{F}$. Given $\pi\in\Irr(\underline{G}(F))$,
let $f(r):=\dim\pi^{K_{r}}$. Then $f(2r)$ is a polynomial in $q_{F}^{r}$,
for $r$ large enough (see e.g.~\cite[Section 16]{Vin} and \cite[p.2]{Suz}).
The \emph{Gelfand\textendash Kirillov dimension} $\mathrm{GK}(\pi)$
is half the degree of this polynomial. 
\end{defn}

\begin{rem}
\label{rem:on GK dimension and epsilon}~
\begin{enumerate}
\item In view of Definition \ref{def:Gelfand-Kirillov dimension}, the Gelfand\textendash Kirillov
dimension $\mathrm{GK}(\pi)$ of $\pi$ measures the rate of growth
of the multiplicity of the trivial representation of $K_{r}$ in $\pi|_{K_{r}}$,
as one takes smaller and smaller open compacts $K_{r}$. On the other
hand, as we observe from Theorem \ref{thmE:upper bounds on multiplicities of K-types},
the  invariant $\epsilon_{\star}(\pi)$ gives information on the multiplicities
of all representations of $K_{r}$ appearing in $\Theta_{\pi}|_{K_{r}}$
for a given open compact $K_{r}$, where now the asymptotics is with
respect to the dimension of the representation instead of the depth
of the open compact. 
\item Back to Example \ref{exa:some examples of epsilon}(2), if $\PP_{j}<\GL_{2n}$
is the parabolic subgroup with $\p_{j}=\mathrm{Lie}(\PP_{j}(\Qp))$,
then one can observe that the two representations $\pi_{1}:=\mathrm{Ind}_{\PP_{1}(\Qp)}^{\GL_{2n}(\Qp)}1$
and $\pi_{2}:=\mathrm{Ind}_{\PP_{n}(\Qp)}^{\GL_{2n}(\Qp)}1$ satisfy
$\epsilon_{\star}(\pi_{1})=\epsilon_{\star}(\pi_{2})=1$, while on
the other hand $\mathrm{GK}(\pi_{1})=2n-1$ is very small while $\mathrm{GK}(\pi_{2})=n^{2}$
is much larger. 
\end{enumerate}
\end{rem}

\subsubsection{\label{subsec:Branching-multiplicities-in compact homogeneous}Branching
multiplicities in compact homogeneous spaces}

Let $\underline{G}$ be a complex, connected reductive algebraic group,
let $\underline{T}\subseteq\underline{G}$ be a maximal torus. Let
$\underline{L}$ be a Levi subgroup of $\underline{G}$, containing
$\underline{T}$. Let $G_{\C}=\underline{G}(\C)$, $T_{\C}=\underline{T}(\C)$
and $L_{\C}=\underline{L}(\C)$. Let $K\leq G_{\C}$ be a maximal
compact subgroup in $G_{\C}$, containing the maximal compact torus
$T$ of $T_{\C}$. We call the subgroup $L:=L_{\C}\cap K$ a \emph{Levi
subgroup} of $K$. Conversely, every compact connected Lie group $K$
is a maximal compact subgroup of a complex, connected reductive group
$G_{\C}$ (see e.g.~\cite[Section 5.2.5, Theorem 12]{OV90}).

In $\mathsection$(\ref{subsec:Bounds-on-multiplicities in compact homogeneous}),
by analyzing the integrability of $\Theta_{\mathrm{Ind}_{L}^{K}1}=\Theta_{L^{2}(K/L)}$,
we provide upper bounds on the multiplicities of the irreducible subrepresentations
in $L^{2}(K/L)$:

\begin{theorem}\label{thm F:bound on multilicites in compact homogeneous}Let
$K$ be a compact, connected Lie group, and let $L\leq K$ be a Levi
subgroup. Then:
\begin{enumerate}
\item We have
\begin{equation}
\epsilon_{\star}(\Theta_{L^{2}(K/L)})=\frac{1}{2}\underset{\t_{\C}\varsubsetneq\l'\subseteq\g_{\C}}{\min}\frac{\mathrm{ss.rk}(\l')+2\left|\Sigma^{+}(\l_{\C},\t_{\C})\cap\Sigma^{+}(\l',\t_{\C})\right|}{\left|\Sigma^{+}(\l',\t_{\C})\backslash\Sigma^{+}(\l_{\C},\t_{\C})\right|},\label{eq:epsilon of character}
\end{equation}
where $\l'$ runs over all pseudo-Levi subalgebras\footnote{A pseudo-Levi subalgebra of $\g_{\C}$ is the Lie algebra of the centralizer
of any semisimple element in $G_{\C}$.} of $\g_{\C}$ containing $\t_{\C}$.
\item For every $\epsilon<\epsilon_{\star}(\Theta_{L^{2}(K/L)})$, there
exists $C=C(\epsilon,K)>0$ such that for every $\tau\in\Irr(K)$:
\[
\dim\mathrm{Hom}_{K}\left(\tau,L^{2}(K/L)\right)=\dim\tau^{L}<C\cdot\left(\dim\tau\right)^{1-\frac{2\epsilon}{1+\epsilon}}.
\]
\end{enumerate}
\end{theorem}

In the case of $\g=\gl_{n}$, every pseudo-Levi subalgebra is a Levi
subalgebra. With the help of Proposition \ref{prop:formula for log canonical threshold},
Eq.~(\ref{eq:epsilon of character}) becomes $\epsilon_{\star}(\Theta_{L^{2}(\U_{n}/\mathrm{\U_{\lambda}})})=\frac{1}{2}\lct_{\C}(\triangle_{\g,\l_{\lambda}}^{\t};(\triangle_{\l_{\lambda}}^{\t})^{2},0)$,
where $\mathrm{lct_{F}}(f;g,x)$ is the log-canonical threshold of
$f$ relative to $g$ (see Definition \ref{def:relative log canonical threshold}).
We analyze this log-canonical threshold in Proposition \ref{Prop:lct of relative Weyl discriminant},
which yields a simplified formula for $\epsilon_{\star}(\Theta_{L^{2}(\U_{n}/\mathrm{\U_{\lambda}})})$,
as follows.
\begin{cor}
\label{cor:epsilon for Levis in SU_n}Let $\lambda=(\lambda_{1},\dots,\lambda_{N})\vdash n$
be a partition with $N$ non-zero parts. Let $\LL=\LL_{\lambda}:=\prod_{i=1}^{N}\GL_{\lambda_{i}}$
and $\U_{\lambda}=\prod_{i=1}^{N}\mathrm{U}_{\lambda_{i}}$ be the
corresponding Levi subgroups of $\GL_{n}$ and $\U_{n}$. Then, 
\begin{equation}
\epsilon_{\star}(\Theta_{L^{2}(\U_{n}/\mathrm{\U_{\lambda}})})=\frac{1}{2}\lct_{\C}(\triangle_{\g,\l_{\lambda}}^{\t};(\triangle_{\l_{\lambda}}^{\t})^{2},0)=\frac{1}{N}.\label{eq:concrete formula for epsilon}
\end{equation}
\end{cor}

\begin{rem}
\label{rem:two remarks}~
\begin{enumerate}
\item We note that pseudo-Levi subalgebras in semisimple Lie algebras are
fully classified in terms of the affine Dynkin diagram. More precisely,
in the same way that standard Levis are determined by subsets of vertices
in the Dynkin diagram, then the standard pseudo-Levi subalgebras are
determined by subsets of vertices of the affine Dynkin diagram (see
e.g. \cite{BdS49} and also \cite[Secion 1.2 and Proposition 2]{Som97}).
We therefore expect the methods of this paper to give a simplified
formula as in (\ref{eq:concrete formula for epsilon}) for any compact
connected semisimple Lie group $K$ and any Levi subgroup $L$ in
$K$, with some combinatorial challenges coming from the more complicated
structure of type $B,C,D$ (affine) root systems.
\item Similarly to the discussion in Remark \ref{rem:on GK dimension and epsilon},
the Gelfand-Kirillov dimension is also related to the growth of multiplicities
of $K$-representation in $L^{2}(K/L)$, see \cite[Theorem 1.2]{Vog78}.
However, this relation is not refined enough to give bounds on the
multiplicity of each individual representation. 
\end{enumerate}
\end{rem}

\subsubsection{\label{subsec:Applications-in-Random}Applications in random matrix
theory}

The field of random matrix theory focuses on understanding eigenvalue
statistics of matrices whose entries are drawn randomly according
to some probability distribution. One of the popular models is to
consider Haar-random $n\times n$-unitary matrices $\sX$ in $\U_{n}$,
also known as the Circular Unitary Ensemble (CUE) (see e.g.~\cite{AGZ10,Mec19}
and the references within). Given $\sX$, denote by $\{\lambda_{1},...,\lambda_{n}\}$
the set of its eigenvalues, where $\lambda_{i}\in\mathbb{S}^{1}\subseteq\C$.
A formula for the joint density of $\{\lambda_{1},...,\lambda_{n}\}$
is given by the Weyl integration formula \cite{Wey39}. In their seminal
work \cite{DS94}, Diaconis and Shahshahani have shown that the sequence
$(\tr(\sX),\tr(\sX^{2}),\ldots,\tr(\sX^{m}))$, i.e.~the sequence
$(\sum_{i=1}^{n}\lambda_{i}^{1},...,\sum_{i=1}^{n}\lambda_{i}^{m})$,
converges, as $n\rightarrow\infty$, to $m$ independent complex normal
random variables. The rate of convergence is known to be super-exponential
by Johansson \cite{Joh97}. 

Let $\mu_{\U_{n}}$ be the Haar probability measure on $\U_{n}$.
For each $\ell\in\{1,\ldots,m\}$, let $w_{\ell}:\mathrm{U}_{n}\rightarrow\mathrm{U}_{n}$
be the power map $X\mapsto X^{\ell}$. Then the pushforward measure
$\tau_{\ell,n}:=(w_{\ell})_{*}(\mu_{\mathrm{U}_{n}})$ describes the
distribution of the random matrix $\sX^{\ell}$. Since $\tau_{\ell,n}$
is a conjugation invariant measure, it is equal, as a distribution,
to a linear combination of characters of $\U_{n}$: 
\[
\tau_{\ell,n}=\sum_{\rho\in\Irr(\U_{n})}\overline{a_{n,\ell,\rho}}\cdot\Theta_{\rho},\text{\,\,\, where \,\,\,\,}a_{n,\ell,\rho}:=\int_{\U_{n}}\Theta_{\rho}(\sX^{\ell})\mu_{\U_{n}}.
\]
The result of \cite{DS94} is a statement about the low-dimensional
Fourier coefficients $a_{n,\ell,\rho}$ of $\tau_{\ell,n}$, i.e.~when
$\rho\in\Irr(\U_{n})$ has dimension which is at most polynomial in
$n$ (see e.g.~the discussion in \cite[Section 1.2.2]{AGL}). However,
in order to fully understand $\tau_{\ell,n}$, one has to analyze
the higher dimensional Fourier coefficients. The key component in
analyzing $\tau_{\ell,n}$ are the works of Rains \cite{Rai97,Rai03},
which showed (\cite[Theorem 1.3]{Rai03}) that the eigenvalue distribution
of $\sX^{\ell}$ is the same as the eigenvalue distribution of a Haar-random
matrix in $H_{n,\ell}:=\mathrm{U}_{\left\lfloor n/\ell\right\rfloor +1}^{j}\times\mathrm{U}_{\left\lfloor n/\ell\right\rfloor }^{\ell-j}$,
where $j:=n\mod\ell$. From Rain's result one can deduce (see Proposition
\ref{prop:Fourier coefficients of a power}) that 
\begin{equation}
a_{n,\ell,\rho}=\dim\rho^{H_{n,\ell}}=\dim\mathrm{Hom}_{\U_{n}}(\rho,L^{2}(\U_{n}/H_{n,\ell})),\label{eq:Fourier coefficients are dimension of fixed vectors}
\end{equation}
which implies the following equality of distributions: 
\begin{equation}
\tau_{\ell,n}=\Theta_{L^{2}(\U_{n}/H_{n,\ell})}.\label{eq:pushforward is a character}
\end{equation}

In joint works of the first author with Avni and Larsen \cite{AG,AGL},
we used (\ref{eq:Fourier coefficients are dimension of fixed vectors})
to obtain asymptotically sharp estimates on all Fourier coefficients
(see \cite[Proposition 7.4]{AGL}): 
\begin{equation}
a_{n,\ell,\rho}=\dim\rho^{H_{n,\ell}}\leq(\dim\rho)^{1-\frac{1}{\ell-1}+\delta},\text{ with }\delta=\delta(n)\underset{n\rightarrow\infty}{\longrightarrow}0.\label{eq:sharp estimates on Fourier coefficients}
\end{equation}
In this paper, we utilize (\ref{eq:pushforward is a character}),
to provide another high energy result regarding the random matrix
$\sX^{\ell}$, by determining the integrability of its distribution
$\tau_{\ell,n}$. More precisely, the measure $\tau_{\ell,n}$ is
absolutely continuous with respect to $\mu_{\U_{n}}$, and thus it
can be written as $\tau_{\ell,n}=f_{\ell,n}\mu_{\U_{n}}$, with $f_{\ell,n}\in L^{1}(\U_{n})$.
However, we show that $\tau_{\ell,n}$ further belongs to $L^{1+\frac{1}{\ell}-\delta}(\U_{n})$
for every $\delta>0$: 

\begin{theorem}\label{thmG:intergability of power measure}Let $n,\ell\geq2$
be integers, and let $\tau_{\ell,n}$ as above. Then, 
\[
\epsilon_{\star}(\tau_{\ell,n})=\frac{1}{\min(n,\ell)}.
\]

\end{theorem}

Theorems \ref{thmG:intergability of power measure} and \ref{thm F:bound on multilicites in compact homogeneous}(2)
imply the following bounds on $a_{n,\ell,\rho}$\footnote{Note that the case $\ell=2$ follows from the Gelfand property of
$(\U_{n},\U_{\left\lfloor n/2\right\rfloor }\times\U_{\lceil n/2\rceil})$.}, which on the one hand are dimension-dependent, but on the other
hand have improved exponent compared to (\ref{eq:sharp estimates on Fourier coefficients}).
\begin{cor}
\label{cor:Fourier coefficients of power word measure}For every $\ell,n\geq2$
and every $\delta>0$, there exists $C(n,\delta)>0$, such that for
every $\rho\in\Irr(\U_{n})$:
\[
a_{n,\ell,\rho}=\dim\rho^{H_{n,\ell}}\leq\begin{cases}
C(n,\delta)\cdot\left(\dim\rho\right)^{1-\frac{2}{\min(n,\ell)+1}+\delta} & \text{if }\ell\geq3\\
1 & \text{if }\ell=2.
\end{cases}
\]
\end{cor}

\begin{rem}
\label{rem:A few remarks}~
\begin{enumerate}
\item We expect a similar statement as Theorem \ref{thmG:intergability of power measure}
to hold for $\ell$-th powers of Haar-random matrices in any compact,
simple Lie group. The Fourier estimates of (\ref{eq:sharp estimates on Fourier coefficients})
were recently generalized from $\U_{n}$ to any compact simple Lie
group by Saar Bader \cite{Bad}.
\item In \cite[Question 1.15]{GHS}, we conjectured more generally, that
if $\sX_{1},\ldots,\sX_{r}$ are independent Haar-random matrices
in a compact simple Lie group $K$, and $w(\sX_{1},\ldots,\sX_{r})$
is a word on $r$ letters (e.g.~$w(\sX,\sY)=\sX\sY\sX^{-1}\sY^{-1}$
is the commutator word), the integrability exponent $\epsilon_{\star}(\tau_{w,K})$
of $\tau_{w,K}:=w_{*}\mu_{K^{r}}$ is bounded from below by some constant
$\epsilon(w)>0$ which \textbf{depends only on }$w$. Theorem \ref{thmG:intergability of power measure}
provides evidence to this conjecture. Various other dimension-independent
results in this spirit were made in the setting of compact $p$-adic
groups (e.g.~for $\mathrm{SL}_{n}(\Zp)$) in \cite{GHb}.
\end{enumerate}
\end{rem}

\subsubsection{Transfer of the results to large positive characteristic}

As mentioned above, the distributions $\widehat{\xi}_{\mathcal{O}}$
and consequently the Harish-Chandra characters $\Theta_{\pi}$ are
tame objects in a certain model\textendash theoretic sense, and their
study is amenable to techniques from model theory, and specifically
the theory of motivic integration. Based on this, we believe that
Theorems \ref{thm A: epsilon of Fourier of orbital integral is lct(D(G/M))},
\ref{thmB:lower bound on epsilon of reps}, \ref{thmC: epsilon of Fourier of orbital integral for GL_n},
\ref{thm D:epsilon of a representation} and \ref{thmE:upper bounds on multiplicities of K-types}
as well as Corollary \ref{cor:upper bound for multiplicities in principal series}
also hold for every local field $F$ of characteristic larger than
$M$, for some constant $M$ depending only on the absolute root datum
of $\underline{G}$.

Some of the key components in the results above are based on resolution
of singularities of hyperplane arrangements, so one can use good reduction
in algebraic geometry to transfer those components to large enough
positive characteristic. However, to transfer those results to large
positive characteristic in full generality, one need to generalize
the transfer results of the second author jointly with Cluckers and
Halupczok in \cite{CGH14a,CGH14b,CGH18} in the theory of motivic
integration. This is investigated in a sequel to this paper. 

\subsection{\label{subsec:Methods of proof and further discussion}Main ideas
of proofs and some open questions}

\subsubsection{\label{subsec:Proof-of-Theorem A}Sketch of the proof of Theorem
\ref{thm A: epsilon of Fourier of orbital integral is lct(D(G/M))}}

We prove Theorem \ref{thm A: epsilon of Fourier of orbital integral is lct(D(G/M))}
in $\mathsection$\ref{sec:Geometric construction of Fourier transform}.
Let $\mathcal{O}$ be a Richardson nilpotent orbit in $\ggm:=\Lie(\GG)$,
with polarization $\ppm=\llm\oplus\nnm$ (see Definition \ref{def:Richardson orbit}),
and let $\nnm^{-}$ be the opposite of $\nnm$. We consider the polynomial
map $\Phi:\ggm\rightarrow\ggm$, given by $(X,Z)\mapsto\Ad(\exp(X))(Z)$
for $(X,Z)\in\nnm^{-}\times\ppm\simeq\ggm$. We then use a geometric
construction which is based on the generalized Springer resolution
(see $\mathsection$\ref{subsec:geometric construction}), to show
that for each local field $F$ of characteristic zero, $\widehat{\xi}_{\mathcal{\mathcal{O}_{\mathrm{st}}}}$
is given as the pushforward of the Haar measure $\mu_{\g}$ on $\g=\ggm(F)$
by $\Phi$ (see Corollary \ref{cor:Fourier transform of orbital integrals as pushforward measures}
and the proof of Proposition \ref{prop:explicit description of the Fourier transform of obital integral }). 

A direct calculation shows that the Jacobian of $\Phi$ at a point
$(X,Z)\in\n^{-}\times\p$ satisfies $\left|\mathrm{Jac}(\Phi)|_{(X,Z)}\right|_{F}=\left|D_{\g,\l}(\overline{Z})\right|_{F}^{\frac{1}{2}}$,
where $\overline{Z}$ is the projection of $Z\in\p$ to $\l$. Since
$\Phi$ is a dominant analytic map between spaces of equal dimension,
it fits the setting of \cite[Proposition 4.1]{GHS}, which implies,
\begin{equation}
\epsilon_{\star}(\widehat{\xi}_{\mathcal{\mathcal{O}_{\mathrm{st}}}};0)=\epsilon_{\star}(\Phi_{*}\mu_{\g};0)=\lct_{F}(\mathrm{Jac}(\Phi);0),\label{eq:epsilon is the same as lct of the Jacobian}
\end{equation}
where $\mathrm{Jac}(\Phi)$ denotes the Jacobian of $\Phi$. The theorem
now follows from (\ref{eq:epsilon is the same as lct of the Jacobian}),
along with the homogeneity of $\widehat{\xi}_{\mathcal{\mathcal{O}_{\mathrm{st}}}}$
which guarantees that $\epsilon_{\star}(\widehat{\xi}_{\mathcal{\mathcal{O}_{\mathrm{st}}}})=\epsilon_{\star}(\widehat{\xi}_{\mathcal{\mathcal{O}_{\mathrm{st}}}};0)$. 

We see from the proof above that $\widehat{\xi}_{\mathcal{\mathcal{O}_{\mathrm{st}}}}$
is a positive function. On the other hand, its non-stable constituents
$\widehat{\xi_{\mathcal{\mathcal{O}}}}$ might not be positive functions,
so potential cancellations might occur among them. Nonetheless, we
expect the following to have an affirmative answer: 
\begin{question}
\label{que:non-stable orbits}Let $\mathcal{O}$ be a nilpotent orbit
in $\text{\ensuremath{\ggm}}$ and let $\mathcal{O}_{1}\cup\dots\cup\mathcal{O}_{N}=\mathcal{O}(F)=\mathcal{O}_{\mathrm{st}}$
be the corresponding stable orbit. Is it true that $\epsilon_{\star}(\widehat{\xi}_{\mathcal{O}_{i}})=\epsilon_{\star}(\widehat{\xi}_{\mathcal{O}_{st}})$
for each $i$?
\end{question}

A potentially hard problem (see Remark \ref{rem:Rango Rao} for a
discussion on the difficulties) would be: 
\begin{problem}
Find a formula similarly to Theorem \ref{thmC: epsilon of Fourier of orbital integral for GL_n}
for $\epsilon_{\star}(\widehat{\xi}_{\mathcal{O}})$, for a non-Richardson
orbit $\mathcal{O}$.
\end{problem}

\subsubsection{\label{subsec:Proof-of-Theorem C}Sketch of the proof of Theorem
\ref{thmC: epsilon of Fourier of orbital integral for GL_n}}

The proof of Theorem \ref{thmC: epsilon of Fourier of orbital integral for GL_n}
is the heart of the paper and it is done in Sections \ref{sec:Some-integrals-associated to coxeter arrangements}
and \ref{subsec:An-explicit-formula for epsilon(O)}. While the theorem
is stated for $\GL_{n}$, some of the components of the proof work
in a larger generality. Here are the main steps of the proof: 
\begin{itemize}
\item First note that Theorem \ref{thm A: epsilon of Fourier of orbital integral is lct(D(G/M))}
reduces Theorem \ref{thmC: epsilon of Fourier of orbital integral for GL_n}
to a \textbf{geometric problem} of analyzing the singularities (precisely,
the $F$-log-canonical threshold) of the relative Weyl discriminant
$D_{\g,\l}$, where $\ppm=\llm\oplus\nnm$ is a polarization of the
orbit $\mathcal{O}$. 
\item \textbf{Step 1: reduction to the setting of hyperplane arrangements
of Coxeter type}: here we assume that $\GG$ is $F$-split. We use
the Weyl integration formula (Proposition \ref{prop:Weyl integration formula})
to express the integral corresponding to $2\lct_{F}(D_{\g,\l};0)$
as a finite sum of integrals over representatives of $L=\LL(F)$-conjugacy
classes of maximal $F$-tori in $L$:
\begin{equation}
\int_{B}\left|D_{\g,\l}(Y)\right|_{F}^{-\frac{s}{2}}dY=\sum_{T\subseteq L}\frac{1}{\#W(L,T)}\int_{\t}\left|D_{\g,\l}(X)\right|_{F}^{-\frac{s}{2}}\left|D_{\l}(X)\right|_{F}\left(\int_{L/T}1_{B}(\Ad(l).X)d\overline{l}\right)dX,\label{eq:Weyl integration}
\end{equation}
where $B$ is a small ball around $0$ in $\l$, $W(L,T)$ is the
Weyl group of $L$, and $d\overline{l}$ is a suitably normalized
$L$-invariant measure on $L/T$. We use the fact that $\GG$ is $F$-split
to show:
\begin{itemize}
\item The contribution to (\ref{eq:Weyl integration}) is maximal for the
$F$-split torus $T_{0}$, with $\Lie(T_{0})=\t_{0}$.
\item The term $\left|D_{\l}(X)\right|_{F}^{1/2}\left(\int_{L/T_{0}}1_{B}(\Ad(l).X)d\overline{l}\right)dX$
is bounded and positive near $0$. 
\item We have $\left|D_{\l}|_{\t_{0}}\right|_{F}^{\frac{1}{2}}=\left|\triangle_{\l}^{\t_{0}}\right|_{F}$
and $\left|D_{\g}|_{\t_{0}}\right|_{F}^{\frac{1}{2}}=\left|\triangle_{\g}^{\t_{0}}\right|_{F}$
as in (\ref{eq:relative Weyl discriminant}) and (\ref{eq:reduction to hyperplane arrangement}),
and the vanishing set of $\triangle_{\g}^{\t_{0}}$ (and $\triangle_{\l}^{\t_{0}}$)
is a finite collection $\left\{ W_{\alpha}\right\} _{\alpha}$ of
hyperplanes $W_{\alpha}:=\{\alpha=0\}\subseteq\t_{0}$, or a \emph{hyperplane
arrangement} (see $\mathsection$\ref{subsec:Hyperplanes-arrangements--basic}),
which is of \emph{Coxeter type}. 

It is left to estimate, for each ball $B'$ around $0$ in $\t_{0}$,
the integral
\begin{equation}
\int_{B'}\left|\triangle_{\g}^{\t_{0}}\right|_{F}^{-s}\left|\triangle_{\l}^{\t_{0}}\right|_{F}^{1+s}dX.\label{eq:reduced integral}
\end{equation}

\end{itemize}
\item \textbf{Step 2: reduction to a complicated combinatorial formula}:
we use a known algorithm to resolve singularities of hyperplane arrangements
(Theorem \ref{Thm:resolution of hyperplane arrangements}), to turn
(\ref{eq:reduced integral}) into finitely many integrals of monomial
functions, which can be easily analyzed, and deduce that:
\begin{equation}
2\lct_{F}(D_{\g,\l};0)=\underset{W\in\mathcal{L}}{\mathrm{min}}\beta(W),\label{eq:difficult combinatorial formula}
\end{equation}
where $W$ runs over all possible intersections $\mathcal{L}$ of
hyperplanes from $\left\{ W_{\alpha}\right\} _{\alpha\in\Sigma^{+}(\g,\t_{0})}$,
and each $\beta(W)$ is a rational number explicitly determined by
$W$, $\g$ and $\l$. Note that to compute (\ref{eq:difficult combinatorial formula}),
naively one has to consider intersections of $2^{\left|\Sigma^{+}(\g,\t_{0})\right|}$
possible subsets of hyperplanes in $\left\{ W_{\alpha}\right\} _{\alpha\in\Sigma^{+}(\g,\t_{0})}$. 
\item \textbf{Step 3: simplifying the combinatorial formula}: we first identify
each $W\in\mathcal{L}$ with a Levi subalgebra $\t_{0}\subseteq\l_{W}\subseteq\g$
(Lemma \ref{lem:bijection betwweb dense elements and standard Levis}).
We then show that it is enough to consider only $W\in\mathcal{L}$
with $[\l_{W},\l_{W}]$ simple, and express $\beta(W)$ in terms of
$\l_{W}$, obtaining the following (Proposition \ref{prop:formula for log canonical threshold}):
\begin{equation}
2\lct_{F}(D_{\g,\l};0)=\underset{\substack{\t\varsubsetneq\l'\subseteq\g\\{}
[\l',\l']\text{ is simple}
}
}{\min}\frac{\mathrm{rk}[\l',\l']+\left|\Sigma^{+}(\l,\t)\cap\Sigma^{+}(\l',\t)\right|}{\left|\Sigma^{+}(\l',\t)\backslash\Sigma^{+}(\l,\t)\right|}.\label{eq:simplified combinatorial formula}
\end{equation}
\end{itemize}
The results up until now are valid for any $F$-split $\GG$ . To
further simplify (\ref{eq:simplified combinatorial formula}), we
assume $\GG=\GL_{n}$ for steps 4-6.
\begin{itemize}
\item \textbf{Step 4}: for $\g=\gl_{n}$, we describe all possible embeddings
of Levis $\t_{0}\varsubsetneq\l'\subseteq\gl_{n}$ with $[\l',\l']$
simple (Lemma \ref{lem:structure of Levis}). We then show that for
each $2\leq k\leq n$ and $\l$, there is a \textbf{deterministic
Levi} $\l'$ with $\mathrm{rk}[\l',\l']=k-1$, which minimizes the
expression in (\ref{eq:simplified combinatorial formula}), obtaining
the following:
\begin{equation}
2\lct_{F}(D_{\g,\l};0)=\underset{2\leq k\leq n}{\min}\frac{k-1+\psi_{\l,k}}{\binom{k}{2}-\psi_{\l,k}},\label{eq:even more simplified formula}
\end{equation}
where $\psi_{\l,k}\in\N$ is explicitly determined by $\l$ and $k$
(see the proof of Proposition \ref{Prop:lct of relative Weyl discriminant}). 
\item \textbf{Step 5}: If $\nu_{1}\geq\dots\geq\nu_{M}$ are the sizes of
blocks in the Jordan normal form of $\mathcal{O}$, we can partition
the interval $[2,n]$ into $M$ subintervals, on which the inner term
in (\ref{eq:even more simplified formula}) is decreasing in $k$.
Hence the minimum in (\ref{eq:even more simplified formula}) must
be obtained at the end points of each subinterval. This proves Item
(1) of Theorem \ref{thmC: epsilon of Fourier of orbital integral for GL_n}. 
\item \textbf{Step 6}: given $\mathcal{O},\mathcal{O}'\in\mathcal{O}(\mathcal{N}_{\gl_{n}(F)})$,
with Jordan blocks corresponding to the partitions $\nu_{1}\geq\dots\geq\nu_{M}$
and $\nu'_{1}\geq\dots\geq\nu'_{M'}$, respectively, we have $\mathcal{O}'\subseteq\overline{\mathcal{O}}$
if and only if the partition $\nu$ \emph{dominates} $\nu'$ (see
Definition \ref{def:partitions} and Theorem \ref{thm:orbit closure and dominance}).
Then Item (2) of Theorem \ref{thmC: epsilon of Fourier of orbital integral for GL_n}
follows by observing that (\ref{eq:formula for epsilon of orbit})
becomes smaller as $\nu$ becomes more dominant. 
\end{itemize}
We further stress that Steps 1-3 above are valid for $F$-split groups,
while Steps 4-6 are valid only for $\GL_{n}$. However, we expect
that a similar analysis can be done for stable Richardson orbits in
$F$-split simple algebraic groups, with a few difficulties, due to
the more complicated combinatorics of the nilpotent orbit structure.

It is natural to ask whether a similar phenomenon as in Item (2) of
Theorem \ref{thmC: epsilon of Fourier of orbital integral for GL_n}
holds in general:
\begin{question}
\label{que:behavior of epsilon with respect to orbit closure}Let
$\text{\ensuremath{\GG}}$ be a connected reductive $F$-group. Given
two orbits $\mathcal{O},\mathcal{O}'\in\mathcal{O}(\mathcal{N}_{\g})$
such that $\mathcal{O}'\subseteq\overline{\mathcal{O}}$, is it true
that $\epsilon_{\star}(\widehat{\xi}_{\mathcal{O}'})\geq\epsilon_{\star}(\widehat{\xi}_{\mathcal{O}})$?
\end{question}

\subsubsection{\label{subsec:Proof-of-Theorem D}Sketch of the proof of Theorem
\ref{thm D:epsilon of a representation}}

In order to compute $\epsilon_{\star}(\pi;e)$ of $\pi\in\Irr(\GG(F))$
for a connected, reductive $F$-group $\text{\ensuremath{\GG}}$,
we consider the local character expansion of $\Theta_{\pi}$ near
$0\in\g$:
\begin{equation}
\theta_{\pi}(\exp(X))=\sum_{\mathcal{O}\in\mathcal{O}(\mathcal{N}_{\g})}c_{\mathcal{O}}(\pi)\cdot\widehat{\xi}_{\mathcal{O}}(X).\label{eq:local character}
\end{equation}
Naively, one would hope that $\epsilon_{\star}(\pi;e)=\underset{\mathcal{O}:c_{\mathcal{O}}(\pi)\neq0}{\min}\epsilon_{\star}(\widehat{\xi}_{\mathcal{O}})$.
However, there might be cancellations between the terms in (\ref{eq:local character}).
When $\underline{G}=\GL_{n}$, two pleasant phenomena occur: 
\begin{itemize}
\item There is a unique orbit $\mathcal{O}_{\mathrm{max}}$ such that $c_{\mathcal{O}_{\mathrm{max}}}(\pi)\neq0$
and any other orbit $\mathcal{O}'$ with $c_{\mathcal{O}'}(\pi)\neq0$
must be contained in $\overline{\mathcal{O}}_{\mathrm{max}}$ (see
Proposition \ref{prop:local character for GLn}).
\item By Theorem \ref{thmC: epsilon of Fourier of orbital integral for GL_n},
$\epsilon_{\star}(\widehat{\xi}_{\mathcal{O}_{\mathrm{max}}})=\underset{\mathcal{O}:c_{\mathcal{O}}(\pi)\neq0}{\min}\epsilon_{\star}(\widehat{\xi}_{\mathcal{O}})$. 
\end{itemize}
Thus, in order to verify the naive hope above, one has to check whether
$\epsilon_{\star}(\pi;e)$ is equal to $\epsilon_{\star}(\widehat{\xi}_{\mathcal{O}_{\mathrm{max}}})$.
The fact that $\mathcal{O}_{\mathrm{max}}$ is the unique orbit in
(\ref{eq:local character}) of largest dimension allows us to use
the \textbf{homogeneity} of $\widehat{\xi}_{\mathcal{O}}$ to show
that an effective cancellation between $\widehat{\xi}_{\mathcal{O}_{\mathrm{max}}}$
and the other orbits cannot occur, so that indeed $\epsilon_{\star}(\pi;e)=\epsilon_{\star}(\widehat{\xi}_{\mathcal{O}_{\mathrm{max}}})$.
The proof can be found in $\mathsection$\ref{subsec:A-formula-for epsilon(pi)}.

Two natural questions arise following the above discussion. The first
question is about potential cancellations between the $\widehat{\xi}_{\mathcal{O}}$,
generalizing Question \ref{que:non-stable orbits}:
\begin{question}
\label{que:cancelations}Let $\text{\ensuremath{\GG}}$ be a connected
reductive $F$-group.\textup{ }Given a collection of complex numbers
$\left\{ a_{\mathcal{O}}\right\} _{\mathcal{O}\in\mathcal{O}(\mathcal{N}_{\g})}$,
is it true that
\[
\epsilon_{\star}\left(\sum_{\mathcal{O}\in\mathcal{O}(\mathcal{N}_{\g})}a_{\mathcal{O}}\cdot\widehat{\xi}_{\mathcal{O}}(X)\right)=\underset{\mathcal{O}:a_{\mathcal{O}}\neq0}{\min}\epsilon_{\star}(\widehat{\xi}_{\mathcal{O}})?
\]
\end{question}

Secondly, one can consider the local character expansion of $\pi$
near an arbitrary semisimple element $x$ (see e.g.~\cite[Theorem 16.2]{HC99}),
and analyze $\epsilon_{\star}(\pi;x)$. As $x$ becomes more regular,
the potential nilpotent orbits appearing in the character expansion
are of smaller dimension. One may wonder: 
\begin{question}
Let $\pi\in\Irr(\GG(F))$ for $\GG$ connected, reductive $F$-group.
Is it true that $\epsilon_{\star}(\pi;x)$ minimal for a central element
$x$ of $\GG(F)$?
\end{question}

Finally, we expect Theorem \ref{thm D:epsilon of a representation}
to hold in the Archimedean case as well, so that $\epsilon_{\star}(\pi;e)=\epsilon_{\star}(\widehat{\xi}_{\mathcal{O}_{\mathrm{max}}})$,
where $\mathcal{O}_{\mathrm{max}}$ is the unique maximal orbit appearing
in the Barbasch\textendash Vogan asymptotic expansion. 

\subsection{Conventions}

Throughout the paper we use the following conventions: 
\begin{itemize}
\item We use underlined letters to indicate algebraic group (e.g.~$\GG,\LL,\PP$)
and underlined Gothic letters to denote Lie algebras (e.g.~$\ggm,\llm,\ppm$).
The corresponding non-underlined letters $G,L,P,\g,\l,\p$ denote
the $F$-points of these objects, for a suitable local field $F$.
\item We write $Z(H)$ (resp.~$Z(\mathfrak{h})$) for the center of the
group $H$ (resp.~Lie algebra $\mathfrak{h}$).
\item We write $\mathrm{Cent}_{G}(H)$ (resp.~$\mathrm{Cent}_{\g}(\mathfrak{h})$)
for the centralizer of a subgroup $H$ in $G$ (resp.~the centralizer
of a Lie subalgebra $\mathfrak{h}$ in $\g$). 
\end{itemize}
\begin{acknowledgement*}
We thank Nir Avni, Joseph Bernstein, Nero Budur, Dan Ciubotaru, Dmitry
Gourevitch, Max Gurevich, Erez Lapid and Eitan Sayag. We especially
thank Rami Aizenbud for many useful discussions on the project. We
further thank Weizmann Institute of Science for their hospitality. 

I.G.~was supported by ISF grant 3422/24. J.G.~was supported by NSERC.
Y.H.~was partially supported by the Labex CEMPI (ANR-11-LABX-0007-01)
and by an FWO fellowship of the Research Foundation - Flanders, file
number 12B4X24N. 
\end{acknowledgement*}

\section{\label{sec:Preliminaries}Preliminaries}

\subsection{Embedded resolution of singularities }

Let $F$ be a local field of characteristic zero. 
\begin{defn}
\label{def:blowup, exceptional divisors, proper transform }Let $\mathbb{P}_{F}^{n}$
be the $n$-dimensional projective space. 
\begin{enumerate}
\item Let $U\subseteq F^{n}$ be an open subset, and let $L\subseteq F^{n}$
be the subspace of dimension $n-r$, given by $x_{1}=\ldots=x_{r}=0$.
We set 
\[
\mathrm{Bl}_{L\cap U}U:=\left\{ \left((x_{1},\dots,x_{n}),(X_{1}:\dots:X_{n})\right)\in U\times\mathbb{P}_{F}^{n-1}:x_{i}X_{j}=x_{j}X_{i},\text{ for all }1\leq i,j\leq r\right\} .
\]
The projection $\pi_{U,L\cap U}:\mathrm{Bl}_{L\cap U}U\rightarrow U$
is called \emph{the blowup of $U$ along $L\cap U$}. 
\item More generally, if $X$ is an $n$-dimensional $F$-analytic manifold,
and $Z\subseteq X$ is a closed submanifold of dimension $n-r$, then
by choosing any atlas $\left\{ (V_{i},\phi_{V_{i}})\right\} _{i}$,
the above construction can be glued into a space $\mathrm{Bl}_{Z}X$
and a map $\pi_{X,Z}:\mathrm{Bl}_{Z}X\rightarrow X$, which is the
\emph{blowup of $X$ along $Z$}. Note that $\pi_{X,Z}:\pi^{-1}(X\backslash Z)\rightarrow X\backslash Z$
is an $F$-bi-analytic map. 
\item The inverse image $\pi_{X,Z}^{-1}(Z)$ is called the \emph{exceptional
divisor} of the blowup. If $W\subseteq X$ is any closed submanifold,
the \emph{proper transform} of $W$ is the submanifold $\overline{\pi^{-1}(W\backslash Z)}$. 
\end{enumerate}
\end{defn}

We will use the following analytic version of Hironaka's theorem \cite{Hir64}
on embedded resolution of singularities. 
\begin{thm}[{\cite[Theorem 2.3]{VZG08}, \cite[Theorem 2.2]{DvdD88}, \cite{BM89}
and \cite{Wlo09}}]
\label{thm:analytic resolution of sing}Let $U\subseteq F^{n}$ be
an open subset, and let $f_{1},\dots,f_{r}:U\rightarrow F$ be $F$-analytic
maps, generating a non-zero ideal $J$ in the algebra of $F$-analytic
functions on $U$. Then there exist an $F$-analytic manifold $Y$,
a proper $F$-analytic map $\Lambda:Y\rightarrow U$ and a collection
of closed submanifolds $\left\{ E_{i}\right\} _{i\in T}$ of $Y$
of codimension $1$, equipped with pairs of non-negative integers
$\left\{ (a_{i},b_{i})\right\} _{i\in T}$, such that the following
hold: 
\begin{enumerate}
\item $\Lambda$ is locally a composition of a finite number of blow-ups
at closed submanifolds, and is an isomorphism over the complement
of the common zero set $\mathrm{V}(J)$ of $J$ in $U$. 
\item For every $c\in Y$, there are local coordinates $(y_{1},\dots,y_{n})$
in a neighborhood $V\ni c$, such that each $E_{i}$ containing $c$
is given by the equation $y_{i}=0$. Moreover, if without loss of
generality $E_{1},\dots,E_{m}$ contain $c$, then there exists an
$F$-analytic unit $v:V\rightarrow F$, such that the pullback of
$J$ is the principal ideal 
\begin{equation}
\Lambda^{*}J=\langle y_{1}^{a_{1}}\cdots y_{m}^{a_{m}}\rangle\label{eq:prinipalization}
\end{equation}
and such that the Jacobian of $\Lambda$ (i.e.~$\det\,d_{y}\Lambda$)
is given by the monomial: 
\begin{equation}
\operatorname{Jac}_{y}(\Lambda)=v(y)\cdot y_{1}^{b_{1}}\cdots y_{m}^{b_{m}}.\label{eq:principalization of Jacobian}
\end{equation}
\end{enumerate}
\end{thm}

\begin{defn}
\label{def:log-principalizationl}Given an ideal $J=\langle f_{1},\dots,f_{r}\rangle$
of $F$-analytic functions on $U$, a map $\Lambda:Y\rightarrow U$
satisfying the conditions above is called a \emph{log-principalization}
of $J$. $\left\{ E_{i}\right\} _{i\in T}$ is the collection of \emph{divisors}
attached to $\Lambda$.
\end{defn}

Using Theorem \ref{thm:analytic resolution of sing}, the following
formula for $\lct_{F}(f;x)$ (recall Definition \ref{def:log canonical threshold })
can be given.
\begin{lem}[{see e.g.~\cite[Lemma 2.3]{GHS}}]
\label{lem:formula for lct using resolution}Let $f:F^{n}\to F$
be an $F$-analytic map. Let $\Lambda:Y\rightarrow F^{n}$ be a log-principalization
of $f$, with data $\left\{ E_{i}\right\} _{i\in T}$ and $\left\{ (a_{i},b_{i})\right\} _{i\in T}$
as in Theorem \ref{thm:analytic resolution of sing}. Then:
\begin{equation}
\lct_{F}(f;x)=\min_{i:\,x\in\Lambda(E_{i})}\frac{b_{i}+1}{a_{i}}.\label{eq:formula for lct using resolution data}
\end{equation}
\end{lem}

It will be useful to introduce the following relative variant of the
$F$-log-canonical threshold. A similar notion was introduced in \cite[p.2]{DM22}.
\begin{defn}
\label{def:relative log canonical threshold}Let $f:F^{n}\to F$ and
$g:F^{n}\rightarrow F$ be two analytic maps. Then the\emph{ $F$-log-canonical
threshold of $f$ relative to }$g$, at $x\in F^{n}$ is:
\[
\lct_{F}(f;g,x):=\sup\left\{ s>0:\exists\text{ a ball }x\in B\subseteq F^{n}\text{ s.t.}\int_{B}\frac{\left|g(x)\right|_{F}}{\left|f(x)\right|_{F}^{s}}dx<\infty\right\} .
\]
\end{defn}

\begin{rem}
\label{rem:Lct of polynomials on smooth varieties and behavior with field extensions}Let
$K$ be a field of characteristic zero, and $F\supseteq K$ a local
field. Given a smooth $K$-variety $X$, and polynomials $f,g\in K[X]$,
one can extend the notion of $\lct_{F}(f;g,x_{0})$, for $x_{0}\in X(F)$
to this setting, by analyzing integrability of $\frac{\left|g(x)\right|_{F}}{\left|f(x)\right|_{F}^{s}}\left|\omega\right|_{F}$
near $x_{0}$, for some invertible top form $\omega$. We then have
$\lct_{F}(f;g,x_{0})\geq\lct_{F'}(f;g,x_{0})$ for every field extension
$F'\supseteq F$. If $K=\Q$, then we further have $\lct_{F}(f;g)\geq\lct_{\C}(f;g)$.
For further details about this phenomenon we refer to \cite[Corollary 2]{Saia}
and \cite[Section 1.5]{Bud12}.
\end{rem}

\subsection{\label{subsec:Hyperplanes-arrangements--basic}Hyperplane arrangements
and their resolutions of singularities}

We first recall basic notions from the theory of hyperplane arrangements,
which will be used in $\mathsection$\ref{sec:Some-integrals-associated to coxeter arrangements}.
We follow \cite{Sta07}.
\begin{defn}
\label{def:hyperplane arrangement}~
\begin{enumerate}
\item A \emph{finite hyperplane arrangement} $\mathcal{A}$ is a finite
set of affine hyperplanes $W_{1},\dots,W_{m}$ (defined by $L_{W_{i}}(x)=0$)
in some $n$-dimensional vector space $V$ over a field $K$. 
\item Given a hyperplane arrangement $\mathcal{A}$, its \emph{defining
polynomial} is 
\[
Q_{\mathcal{A}}(x)=\prod_{W\in\mathcal{A}}L_{W}(x).
\]
\item The \emph{dimension of $\mathcal{A}$}, denoted $\dim\mathcal{A}$,
is $\mathrm{dim}V=n$. The \emph{rank of $\mathcal{A}$}, denoted
$\mathrm{rk}\mathcal{A}$, is the dimension of the subspace spanned
by the normal vectors to $W\in\mathcal{A}$. A hyperplane arrangement
$\mathcal{A}$ is called \emph{essential} if $\mathrm{rk}\mathcal{A}=\dim\mathcal{A}$.
\item An arrangement $\mathcal{A}$ is called \emph{central} if $\bigcap_{W\in\mathcal{A}}W\neq\slashed{O}$.
Then $\mathcal{A}$ is central if and only if it is a translate of
an arrangement of hyperplanes passing through $0\in V$. Moreover,
if $\mathcal{A}$ is central, then $\mathrm{rk}\mathcal{A}=\mathrm{codim}\bigcap_{W\in\mathcal{A}}W$. 
\item A central hyperplane arrangement $\mathcal{A}$ is called \emph{indecomposable}
if there is no linear change of coordinates on $V$ such that $Q_{\mathcal{A}}$
can be written as the product of two non-constant polynomials in disjoint
sets of variables. 
\end{enumerate}
\end{defn}

\begin{defn}
\label{def:more on hyperplane arrangements}Let $\mathcal{A}$ be
a finite hyperplane arrangement in $V$.
\begin{enumerate}
\item Let 
\[
L(\mathcal{A}):=\left\{ \bigcap_{W\in\mathcal{A}'}W\,:\mathcal{A}'\subseteq\mathcal{A}\text{ and }\bigcap_{W\in\mathcal{A}'}W\neq\slashed{O}\right\} .
\]
Define $W_{1}\leq W_{2}$ in $L(\mathcal{A})$ if $W_{1}\supseteq W_{2}$.
We call $L(\mathcal{A})$ the \emph{intersection poset of $\mathcal{A}$}. 
\item A \emph{lattice} is a poset $(S,\leq)$ for which any two elements
have a unique upper bound and a unique lower bound. If $\mathcal{A}$
is central, then $L(\mathcal{A})$ is a lattice, and is thus called
the \emph{intersection lattice of $\mathcal{A}$}. 
\item An element $W\in L(\mathcal{A})$ is called \emph{dense} if the hyperplane
arrangement $\mathcal{A}_{W}$ given by the image of $\bigcup_{W'\in\mathcal{A}:W'\supseteq W}W'$
in $V/W$ is indecomposable. Let $L_{\mathrm{dense}}(\mathcal{A})$
be the set of dense elements $W\in L(\mathcal{A})$. 
\item For each $W\in L(\mathcal{A})$ we write $r(W):=\mathrm{codim}W$
and $s(W):=\#\left\{ W'\in\mathcal{A}:W'\supseteq W\right\} $. 
\end{enumerate}
Hyperplane arrangements have a simple algorithm for log-resolution
of singularities, and the numerical data of the resolution can be
given as below. 
\end{defn}

\begin{thm}[{\cite[Theorem 8]{STV95} and \cite{DP95}. See also \cite[p.133]{BMT11}}]
\label{Thm:resolution of hyperplane arrangements}Let $F$ be a local
field of characteristic $0$, $V$ be an $F$-vector space, and let
$\mathcal{A}$ be an arrangement in $V$. Let $\Lambda:Y\rightarrow V$
be the map obtained by successive blowups, by taking for $d=0,1,\dots,\dim V-2$
the blowup along the (proper transform of) the union of the dense
elements $W\in L_{\mathrm{dense}}(\mathcal{A})$ of dimension $d$.
Then:
\begin{enumerate}
\item $\Lambda$ is a log-principalization of $Q_{\mathcal{A}}$. 
\item The set of divisors $\left\{ E_{W}\right\} _{W\in L_{\mathrm{dense}}(\mathcal{A})}$
of $\Lambda$ is parameterized by the dense elements $L_{\mathrm{dense}}(\mathcal{A})$. 
\item For each $W\in L_{\mathrm{dense}}(\mathcal{A})$, let $k_{W}$ be
the order of vanishing of $\operatorname{Jac}(\Lambda)$ along $E_{W}$,
let $a_{W}$ be the order of vanishing of $Q_{\mathcal{A}}$ along
$E_{W}$. Then:
\[
k_{W}=r(W)-1,\text{ \,\,\,\,and\,\,\,\, }a_{W}=s(W).
\]
Explicitly, for each $W'\in\mathcal{A}$, the order of vanishing of
$L_{W'}$ along $E_{W}$ is $1$ if $W'\supseteq W$ and $0$ otherwise. 
\end{enumerate}
\end{thm}

Theorem \ref{Thm:resolution of hyperplane arrangements} and Lemma
\ref{lem:formula for lct using resolution} imply the following: 
\begin{cor}[{cf.~\cite[Corollary 0.3]{Mus06}}]
\label{cor:lct of hyperplane arrangement}Let $\mathcal{A}$ be a
central hyperplane arrangement in $V$. Then:
\[
\mathrm{lct}_{F}(Q_{\mathcal{A}})=\underset{W\in L_{\mathrm{dense}}(\mathcal{A})}{\mathrm{min}}\frac{k_{W}+1}{a_{W}}=\underset{W\in L_{\mathrm{dense}}(\mathcal{A})}{\mathrm{min}}\frac{r(W)}{s(W)}.
\]
\end{cor}

\subsection{\label{subsec:Nilpotent-orbits-in gln}Nilpotent orbits in $\protect\gl_{n}$}

We follow the book of Collingwood\textendash McGovern \cite{CM93}. 
\begin{defn}
~\label{def:partitions}
\begin{enumerate}
\item A \emph{partition} $\lambda$ of $m\in\N$, denoted $\lambda\vdash m$,
is a non-increasing sequence $\lambda=(\lambda_{1},\dots,\lambda_{k})$
of non-negative integers that sum to $m$. Two partitions are equivalent
if they have the same non-zero parts. A partition $\lambda=(\lambda_{1},\dots,\lambda_{k})$,
with $\lambda_{k}>0$ is graphically encoded by a \emph{Young diagram},
which is a finite collection of boxes arranged in $k$ left-justified
rows, where the $j$-th row has $\lambda_{j}$ boxes. It is convenient
to write $\lambda=(1^{a_{1}}\cdots m^{a_{m}})$ for the partition
$m=\underset{a_{1}\text{ times}}{\underbrace{(1+\ldots+1)}}+\ldots+\underset{a_{m}\text{ times}}{\underbrace{(m+\ldots+m)}}$.
We denote by $\mathcal{P}(m)$ the set of all partitions of $m$. 
\item Given two partitions $\lambda=(\lambda_{1},\dots,\lambda_{k})$, $\mu=(\mu_{1},\dots,\mu_{k'})$
of $m$, where we may assume $k'=k$, we say that \emph{$\lambda$
dominates $\mu$}, and write $\lambda\succeq\mu$, if 
\begin{equation}
\sum_{j=1}^{u}\lambda_{j}\geq\sum_{j=1}^{u}\mu_{j}\text{ for all }u\leq k.\label{eq:dominance relation}
\end{equation}
\item Given a partition $\lambda=(\lambda_{1},\dots,\lambda_{k})$ of $m$,
we define the \emph{conjugate partition} $\lambda^{\mathrm{t}}=(\lambda_{1}^{\mathrm{t}},\dots,\lambda_{k'}^{\mathrm{t}})\vdash m$
where $\lambda_{j}^{\mathrm{t}}=\#\left\{ i:\lambda_{i}\geq j\right\} $.
The map $\lambda\mapsto\lambda^{\mathrm{t}}$ defines an involution
on $\mathcal{P}(m)$, which is order reversing, i.e., $\lambda\succeq\mu$
if and only if $\lambda^{\mathrm{t}}\preceq\mu^{\mathrm{t}}$. 
\end{enumerate}
\end{defn}

We write $\mathcal{N}\subseteq\gl_{n}$ for the set of nilpotent matrices
in $\gl_{n}$. Each $\GL_{n}(\C)$-orbit $\mathcal{O}$ in $\mathcal{N}$
is classified by its Jordan normal form $J_{\mathcal{O}}$. The block
sizes in $J_{\mathcal{O}}$ give rise to a partition $\lambda_{\mathcal{O}}\vdash n$. 
\begin{thm}[{\cite[Theorem 6.2.5]{CM93}, \cite{Ger61,Hes76}}]
\label{thm:orbit closure and dominance}The map $\mathcal{O}\mapsto\lambda_{\mathcal{O}}$
is a bijection $\mathcal{O}(\mathcal{N})\rightarrow\mathcal{P}(n)$.
Moreover, we have $\mathcal{O}'\subseteq\overline{\mathcal{O}}$ if
and only if $\lambda_{\mathcal{O}'}\preceq\lambda_{\mathcal{O}}$. 
\end{thm}

Let $\t\subseteq\gl_{n}$ be the standard Cartan subalgebra of diagonal
matrices. The root subspace decomposition of $\gl_{n}$ is 
\[
\gl_{n}=\t\oplus\bigoplus_{1\leq i\neq j\leq n}\g_{\alpha_{i,j}},
\]
where $\alpha_{i,j}:=\varepsilon_{i}-\varepsilon_{j}$, and $\varepsilon_{k}:\t\rightarrow\C$
is the projection to the $k$-th coordinate. Let $\Sigma(\gl_{n},\t)$
be the set of roots $\{\alpha_{i,j}\}_{i\neq j}$. Let $\mathfrak{b}\subseteq\gl_{n}$
be the Borel subalgebra of upper triangular matrices, let $\Sigma^{+}(\gl_{n},\t)\subseteq\Sigma(\gl_{n},\t)$
be the corresponding subset of \emph{positive roots}, so that $\mathfrak{b}=\t\oplus\bigoplus_{\alpha_{i,j}\in\Sigma^{+}(\gl_{n},\t)}\g_{\alpha_{i,j}}$,
and let $\Pi(\gl_{n},\t)=\left\{ \alpha_{i,i+1}\right\} _{i=1}^{n-1}$
be the set of simple roots associated with our choice of $\mathfrak{b}$.
We write $W\simeq S_{n}$ for the Weyl group of the root system. 

A parabolic subalgebra $\p\subseteq\gl_{n}$ is called \emph{standard},
if it contains $\mathfrak{b}$. A subset $R\subseteq\Pi(\gl_{n},\t)$
defines a standard parabolic subalgebra $\p_{R}$ as follows: let
$\t_{R}:=\bigcap_{\alpha\in R}\mathrm{Ker}\alpha\subseteq\t$ and
$\l_{R}:=\mathrm{Cent}_{\gl_{n}}(\t_{R})$. Then define $\p_{R}:=\l_{R}\oplus\n_{R}$,
where 
\[
\n_{R}:=\bigoplus_{\alpha\in\Sigma^{+}(\gl_{n},\t)\backslash\Sigma^{+}(\l_{R},\t)}\g_{\alpha}.
\]
Note that, up to conjugating with an element of $W$, $\l_{R}$ is
of the form $\text{\ensuremath{\l_{R}}}=\gl_{\nu_{1}}\oplus\dots\oplus\gl_{\nu_{M}}$
where $\nu=(\nu_{1},\dots,\nu_{M})\vdash n$. We write $\p_{\nu}:=\l_{\nu}\oplus\n_{\nu}$
for the standard parabolic with Levi $\l_{\nu}=\gl_{\nu_{1}}\oplus\dots\oplus\gl_{\nu_{M}}$,
and nilpotent radical $\n_{\nu}$. We can now state the following
result: 
\begin{prop}
\label{prop:parabolic subalgebra correspond to partitions}The map
$\nu\mapsto\p_{\nu}$ is a bijection between $\mathcal{P}(n)$ and
the set of conjugacy classes of parabolic subalgebras of $\gl_{n}$. 
\end{prop}

We will need the notion of Richardson orbits, which we recall below. 
\begin{defn}
\label{def:Richardson orbit}$\GG$ be a connected reductive algebraic
$F$-group, let $\ggm=\Lie(\GG)$. 
\begin{enumerate}
\item A nilpotent $\GG$-orbit $\mathcal{O}$ in $\ggm$ is called a \emph{Richardson
orbit}, if there exists an $F$-parabolic subgroup $\PP\subseteq\GG$,
with Levi decomposition $\PP=\LL\cdot\NN$, i.e.~$\LL$ is a Levi
subgroup, and $\NN$ is the unipotent radical of $\PP$, such that
$\mathcal{O}\cap\nnm$ is Zariski dense in $\nnm:=\Lie(\NN)$. In
this case, we say that $\PP=\LL\cdot\NN$ (or $\ppm=\llm\oplus\nnm$)
is a \emph{polarization} of $\mathcal{O}$. 
\item For any Richardson $\GG$-orbit $\mathcal{O}$ in $\ggm$ defined
over $F$, the set $\mathcal{O}(F)$ is a disjoint union of $\GG(F)$-orbits
$\mathcal{O}_{1},\dots,\mathcal{O}_{m}$, in $\ggm(F)$. We call $\mathcal{O}_{\mathrm{st}}:=\mathcal{O}(F)$
a \emph{stable Richardson orbit. If }$\PP=\LL\cdot\NN$ (or $\ppm=\llm\oplus\nnm$)
is a polarization of $\mathcal{O}$, then $P=L\cdot N$ (or $\p=\l\oplus\n$)
is a polarization of $\mathcal{O}_{\mathrm{st}}$, where $P,L,N,\p,\l,\n$
are the corresponding $F$-points of $\PP,\LL,\NN,\ppm,\llm,\nnm$. 
\end{enumerate}
\end{defn}

In $\GL_{n}$, it turns out that all nilpotent orbits are Richardson
orbits (see \cite[Theorem 7.2.3]{CM93}). Given a parabolic subalgebra
$\p=\l\oplus\n$, there is a unique orbit $\mathcal{O}\in\mathcal{O}(\mathcal{N})$
such that $\mathcal{O}\cap\n$ is open and Zariski dense in $\n$,
which is the Richardson orbit $\mathcal{O}_{\p}$ attached to $\p$.
Note that the same orbit $\mathcal{O}_{\p}$ is attached to each $\p'$
which is conjugate to $\p$. We now state the following remarkable
fact: 
\begin{thm}[{\cite[Theorem 7.2.3]{CM93}}]
\label{thm:Richardson orbits and parabolics}Let $\nu\in\mathcal{P}(n)$,
let $\p_{\nu}$ be the corresponding standard parabolic subalgebra,
and let $\mathcal{O}_{\p_{\nu}}$ be the Richardson orbit attached
to $\p_{\nu}$. Then we have $\lambda_{\mathcal{O}_{\p_{\nu}}}=\nu^{\mathrm{t}}$. 
\end{thm}

\begin{rem}
\label{rem:pathology for nilpotent orbit closure}Note that given
two standard parabolic subalgebras $\p_{\nu}$ and $\p_{\nu'}$, such
that $\p_{\nu}\subseteq\p_{\nu'}$, one has $\n_{\nu}\supseteq\n_{\nu'}$
and hence $\mathcal{O}_{\p_{\nu'}}\subseteq\overline{\mathcal{O}}_{\p_{\nu}}$.
The converse, however, is false. Take $\g=\gl_{4}$ and consider the
partitions $\nu=(2,2)$ and $\nu'=(3,1)$. Then $\p_{\nu}\nsubseteq\p_{\nu'}$
and $\p_{\nu'}\nsubseteq\p_{\nu}$. On the other hand, since $\nu'\succ\nu$,
we have $\nu'^{\mathrm{t}}\prec\nu^{\mathrm{t}}$, and hence by Theorems
\ref{thm:orbit closure and dominance} and \ref{thm:Richardson orbits and parabolics},
$\mathcal{O}_{\p_{\nu'}}\subseteq\overline{\mathcal{O}}_{\p_{\nu}}$. 
\end{rem}

\subsection{\label{subsec:Weyl-discriminants}Weyl discriminants}

Let $G$ be a reductive $F$-group of absolute rank $r$, with Lie
algebra $\g$. The Weyl discriminant $D_{\g}:\g\rightarrow F$ of
$\g$ is defined as
\begin{equation}
D_{\g}(X):=\mathrm{coeff}_{r}\left(\det\left(s-\mathrm{ad}_{X}|_{\g}\right)\right),\label{eq:Weyl discriminants-1}
\end{equation}
where $\mathrm{coeff}_{r}(\sum_{i}a_{i}s^{i})=a_{r}$ denotes the
coefficient of $s^{r}$ in the polynomial $\sum_{i}a_{i}s^{i}$. Note
that this is a polynomial condition in the variables of $\g$, so
that $D_{\g}:\g\rightarrow F$ is a polynomial. Similarly, the Weyl
discriminant of $G$ is $D_{G}(g):=\mathrm{coeff}_{r}\left(\det\left(s-\left(\mathrm{Ad}(g)-1\right)|_{\g}\right)\right)$,
which is again a polynomial. We refer to \cite[Section 7]{Kot05}
for further details. 

\section{\label{sec:Some-integrals-associated to coxeter arrangements}Some
integrals associated to finite Coxeter arrangements}

In this section we compute the log-canonical threshold of relative
Weyl discriminants, whose singularities have a strong relationship
with the singularities of Fourier transforms of orbital integrals. 

We assume that $\GG$ is a complex reductive group, and set $G=\GG(\C)$
and $\g=\mathrm{Lie}(G)$. We fix a maximal torus $T\subseteq G$,
with $\t:=\Lie(T)$. Let
\[
\g=\t\oplus\bigoplus_{\alpha\in\Sigma(\g,\t)}\g_{\alpha}
\]
be the root space decomposition, where $\Sigma=\Sigma(\g,\t)\subseteq\t^{*}$
is the set of roots of $\g$ relative to $\t$. Let $\Pi(\g,\t)$
be a basis of simple roots in $\Sigma(\g,\t)$, with $\Sigma^{+}(\g,\t)$
the corresponding set of positive roots. The \emph{semisimple rank}
of $\g$ is $\mathrm{ss.rk}(\g):=\mathrm{rk}([\g,\g])$. 

\subsection{\label{subsec:On-the-Coxeter}On the Coxeter number of simple Lie
algebras}
\begin{defn}[{\cite[Sections 3.16, 3.18, 3.20]{Hum90}}]
\label{def:The-Coxeter-number}Suppose that $[\g,\g]$ is simple,
so that $\Sigma(\g,\t)$ is irreducible. The \emph{Coxeter number}
of $\Sigma(\g,\t)$ is the order of the product of all simple reflections
in $W$. We denote it by $h_{\g}$, as it is independent of the choice
of a Cartan. Equivalently:
\begin{enumerate}
\item $h_{\g}$ is equal to $\underset{\alpha\in\Sigma(\g,\t)}{\max}\mathrm{ht}(\alpha)+1$,
where $\mathrm{ht}(\alpha)$ is the height of the root $\alpha$. 
\item We have $h_{\g}:=\frac{\left|\Sigma(\g,\t)\right|}{\mathrm{ss.rk}(\g)}$. 
\end{enumerate}
\end{defn}

The following lemma is used in $\mathsection$\ref{subsec:Evaluation-of-some}. 
\begin{lem}
\label{lem:inequalities on coxeter numbers}Let $\g$ be a complex
simple Lie algebra, and let $\t\subseteq\l\subseteq\g$ be a Levi
subalgebra. Then
\[
\frac{\left|\Sigma(\l,\t)\right|}{\mathrm{ss.rk}(\l)}\leq\frac{\left|\Sigma(\g,\t)\right|}{\mathrm{ss.rk}(\g)}=h_{\g}.
\]
In particular, if $[\l,\l]$ is simple, then $h_{\l}\leq h_{\g}$. 
\end{lem}

\begin{proof}
If $[\l,\l]$ is simple, then the lemma directly follows from Definition
\ref{def:The-Coxeter-number}(1), since the highest root of $\Sigma(\l,\t)$
is also a root of $\Sigma(\g,\t)$. 

Now suppose that $[\l,\l]$ is not simple. Then $[\l,\l]=\bigoplus_{i=1}^{k}\l_{i}$,
where each $\l_{i}$ is simple. Since $[\l_{i}+\t,\l_{i}+\t]=\l_{i}$
is simple, we have $\frac{\left|\Sigma(\l_{i}+\t,\t)\right|}{\mathrm{rk}(\l_{i})}\leq h_{\g}$
and therefore $\left|\Sigma(\l_{i}+\t,\t)\right|\leq h_{\g}\mathrm{rk}(\l_{i})$.
This implies: 
\[
\frac{\left|\Sigma(\l,\t)\right|}{\mathrm{ss.rk}(\l)}=\frac{\left|\Sigma([\l,\l],\t)\right|}{\mathrm{ss.rk}([\l,\l])}=\frac{\sum_{i=1}^{k}\left|\Sigma(\l_{i}+\t,\t)\right|}{\sum_{i=1}^{k}\mathrm{rk}(\l_{i})}\leq h_{\g}.\qedhere
\]
\end{proof}

\subsection{\label{subsec:Evaluation-of-some}Integrability of relative Weyl
discriminants over $\protect\C$}

Let $\p$ be a standard parabolic subalgebra of $\g$, with Levi $\l\supseteq\t$
and nilpotent radical $\n\subseteq\bigoplus_{\alpha\in\Sigma^{+}(\g,\t)}\g_{\alpha}$.
Recall the definitions of $D_{\g},D_{\l},\triangle_{\g}^{\t},\triangle_{\l}^{\t},\triangle_{\g,\l}^{\t}$
from (\ref{eq:Weyl discriminant on Cartan}) and $\mathsection$\ref{subsec:Weyl-discriminants}.
Since each of the roots $\alpha:\t\rightarrow\C$ is a linear map,
$\triangle_{\g}^{\t}$ is the defining polynomial of a central hyperplane
arrangement, which we denote by $\mathcal{A}_{\g,\t}$. Recall that
$L(\mathcal{A}_{\g,\t})$ is the set of all nonempty intersections
of hyperplanes in $\mathcal{A}_{\g,\t}$. We write $W_{\alpha}=\ker(\alpha)$
for the hyperplane corresponding to $\alpha\in\Sigma^{+}(\g,\t)$.
Denote by $\mathcal{A}_{\l,\t}\subseteq\mathcal{A}_{\g,\t}$ the hyperplanes
corresponding to $\alpha\in\Sigma^{+}(\l,\t)$. These kinds of arrangements
are called \emph{finite Coxeter arrangements}. In order to analyze
convergence of integrals involving $\triangle_{\g,\l}^{\t}(X)$ and
$\triangle_{\g}^{\t}(X)$, we need to apply Theorem \ref{Thm:resolution of hyperplane arrangements}
in the setting of Coxeter arrangements. 

Any subset $R\subseteq\Sigma^{+}(\g,\t)$ of positive roots defines
a subspace $\t_{R}:=\bigcap_{\alpha\in R}\mathrm{Ker}\alpha$, so
that the centralizer $\l_{R}:=\mathrm{Cent}_{\g}(\t_{R})$ is a Levi
subalgebra. Let $\Sigma(\l_{R},\t)$ be the induced root system, so
that
\[
\l_{R}=\t\oplus\bigoplus_{\alpha\in\Sigma(\l_{R},\t)}\g_{\alpha}=\t\oplus\bigoplus_{\alpha\in\Sigma(\g,\t):\alpha(\t_{R})=0}\g_{\alpha}.
\]
Moreover, $\Sigma(\l_{R},\t)\cap\Sigma^{+}(\g,\t)$ defines a system
of positive roots for $\Sigma(\l_{R},\t)$, which we denote by $\Sigma^{+}(\l_{R},\t)$.
Recall $s(\,)$ and $r(\,)$ in Definition \ref{def:more on hyperplane arrangements}(4),
and note that:
\begin{equation}
r(\t_{R}):=\mathrm{codim}_{\t}\t_{R}=\mathrm{ss.rk}(\l_{R}).\label{eq:formula for rank}
\end{equation}
and 
\begin{equation}
s(\t_{R})=\left|\left\{ \beta\in\Sigma^{+}(\g,\t):\mathrm{Ker}\beta\supseteq\t_{R}\right\} \right|=\left|\left\{ \beta\in\Sigma^{+}(\g,\t):\g_{\beta}\subseteq\l_{R}\right\} \right|=\left|\Sigma^{+}(\l_{R},\t)\right|.\label{eq:formula for s}
\end{equation}

\begin{lem}
\label{lem:bijection betwweb dense elements and standard Levis}The
map $\l\mapsto Z(\l)$, which sends $\l$ to its center, induces a
bijection between the set of Levi subalgebras $\t\subseteq\l\subseteq\g$
and $L(\mathcal{A}_{\g,\t})$, with inverse $\t'\mapsto\mathrm{Cent}_{\g}(\t')$.
This bijection restricts to a bijection between the Levi subalgebras
$\t\subseteq\l\subseteq\g$ such that $[\l,\l]$ is a simple Lie algebra,
and $L_{\mathrm{dense}}(\mathcal{A}_{\g,\t})$.
\end{lem}

\begin{proof}
Let $W\in L(\mathcal{A}_{\g,\t})$. Then $W\subseteq\t$ is given
as $W=\bigcap_{\alpha\in R}W_{\alpha}$ for some subset $R\subseteq\Sigma^{+}(\g,\t)$.
Since $W$ is a subtorus, we have $\mathrm{Cent}_{\g}(W)=\l\supseteq\t$
is a Levi subalgebra. On the other hand, given a Levi $\t\subseteq\l$,
its center $Z(\l)\subseteq\t$ is a subtorus, and it is given as $Z(\l):=\bigcap_{\alpha\in\Sigma^{+}(\l,\t)}W_{\alpha}$,
so that $Z(\l)\in L(\mathcal{A}_{\g,\t})$. Finally, for any subtorus
$\t'\subseteq\t$ we have $Z(\mathrm{Cent}_{\g}(\t'))=\t'$ and for
any Levi $\l\supseteq\t$, we have $\mathrm{Cent}_{\g}(Z(\l))=\l$.
This shows that $\l\mapsto Z(\l)$ is indeed a bijection with inverse
$\t'\mapsto\mathrm{Cent}_{\g}(\t')$.

For the second statement, suppose that $\t\subseteq\l\subseteq\g$
is a Levi, which corresponds to $Z(\l):=\mathfrak{t'}\in L(\mathcal{A}_{\g,\t})$.
Set $\t'':=\t/\mathfrak{t'}$ with projection $p_{\mathfrak{t'}}:\t\rightarrow\t''$.
If $\l=\t\oplus\bigoplus_{\alpha\in R}\g_{\alpha}$, then we set:
\[
\left(\mathcal{A}_{\g,\t}\right)_{\mathfrak{t'}}:=\left\{ p_{\mathfrak{t'}}(W_{\alpha}):\alpha\in R\right\} .
\]
Since $\t'$ is the center of $\l$, $\left(\mathcal{A}_{\g,\t}\right)_{\mathfrak{t'}}\simeq\mathcal{A}_{\mathfrak{[\l,\l]},\t''}$.
Now if $[\l,\l]=\l_{1}\oplus\l_{2}$ with $\l_{i}$ semisimple Lie
algebras, then $\mathcal{A}_{\mathfrak{[\l,\l]},\t''}\simeq\mathcal{A}_{\l_{1},\t}\cup\mathcal{A}_{\l_{2},\t}$
so $\mathcal{A}_{\mathfrak{[\l,\l]},\t}$ is not indecomposable.

For the other direction, suppose that $[\l,\l]$ is simple. Then $\mathcal{A}_{\mathfrak{[\l,\l]},\t''}$
is the arrangement associated to some complex simple Lie algebra (also
called an irreducible \emph{Weyl arrangement}), and such arrangements
are indecomposable. Indeed, by \cite[Proposition 4 and Theorem 5]{STV95}
(see also \cite[Theorem 2]{Cra67}), the indecomposability of a complex
irreducible Weyl arrangement $\mathcal{W}$ is equivalent to the condition
$(t-1)^{2}\nmid\chi_{\mathcal{W}}(t)$, where $\chi_{\mathcal{W}}(t)$
is the characteristic polynomial of $\mathcal{W}$, and this can be
seen in \cite[p.125]{BS98}. 
\end{proof}
We now analyze the log-canonical threshold of $\triangle_{\g}^{\t}$: 
\begin{prop}
\label{prop:lct for simple}Let $G$ be a complex connected simple
group, $\g=\mathrm{Lie}(G)$. Then: 
\[
\lct_{\C}(\triangle_{\g}^{\t})=\frac{2}{h_{\g}}.
\]
\end{prop}

\begin{proof}
By Corollary \ref{cor:lct of hyperplane arrangement}, and Eq. (\ref{eq:formula for rank}),
(\ref{eq:formula for s}) and Lemma \ref{lem:bijection betwweb dense elements and standard Levis},
we deduce
\begin{equation}
\lct_{\C}(\triangle_{\g}^{\t})=\underset{\substack{\slashed{O}\neq R\subseteq\Sigma^{+}(\g,\t)\\
\t_{R}\in L_{\mathrm{dense}}(\mathcal{A}_{\g,\t})
}
}{\min}\frac{r(\t_{R})}{s(\t_{R})}=\underset{\substack{\t\varsubsetneq\l\subseteq\g\\{}
[\l,\l]\text{ is simple}
}
}{\min}\frac{\mathrm{ss.rk}(\l)}{\left|\Sigma^{+}(\l,\t)\right|}=\underset{\substack{\t\varsubsetneq\l\subseteq\g\\{}
[\l,\l]\text{ is simple}
}
}{\min}\frac{2}{h_{\l}}.\label{eq:formula for lct}
\end{equation}
The proposition now follows by Lemma \ref{lem:inequalities on coxeter numbers}.
\end{proof}
\begin{cor}
\label{cor:lct of Weil discriminant}Let $G$ be a complex connected
reductive group, let $\g=\mathrm{Lie}(G)$, and let $[\g,\g]=\bigoplus_{i=1}^{M}\g_{i}$
be a decomposition of $[\g,\g]$ into its simple ideals. Then 
\[
\lct_{\C}(\triangle_{\g}^{\t})=\underset{1\leq i\leq M}{\min}\frac{2}{h_{\mathfrak{g_{i}}}}.
\]
\end{cor}

\begin{proof}
Write $\g=[\g,\g]\oplus Z(\g)$, and $\t=\t_{0}\oplus Z(\g)$. Since
$\mathcal{A}_{\g,\t}$ and $\mathcal{A}_{\mathfrak{[\g,\g]},\t_{0}}$
have the same defining polynomial up to a change of variables coming
from $\g=[\g,\g]\oplus Z(\g)$, we have:
\[
\lct_{\C}(\triangle_{\g}^{\t})=\lct_{\C}(\triangle_{[\g,\g]}^{\t_{0}}).
\]
Note that $\triangle_{[\g,\g]}^{\t_{0}}=\prod_{i=1}^{M}\triangle_{\g_{i}}^{\t_{i}}$,
where each $\triangle_{\g_{i}}^{\t_{i}}$ has a different set of variables.
This implies:
\[
\lct_{\C}(\triangle_{\g}^{\t})=\underset{i}{\min}\lct_{\C}(\triangle_{\g_{i}}^{\t_{i}})=\underset{i}{\min}\frac{2}{h_{\mathfrak{g_{i}}}},
\]
where the last equality follows from Proposition \ref{prop:lct for simple}.
We are now done. 
\end{proof}
\begin{prop}
\label{prop:formula for log canonical threshold}Let $G$ be a complex
connected simple group, $\g=\mathrm{Lie}(G)$. Let $\p$ be a parabolic
subalgebra with Levi $\l$. Then for every $m\geq0$, we have 
\[
\lct_{\C}(\triangle_{\g,\l}^{\t};(\triangle_{\l}^{\t})^{m},0)=\underset{\substack{\t\varsubsetneq\l'\subseteq\g\\{}
[\l',\l']\text{ is simple}
}
}{\min}\frac{\mathrm{ss.rk}(\l')+m\left|\Sigma^{+}(\l,\t)\cap\Sigma^{+}(\l',\t)\right|}{\left|\Sigma^{+}(\l',\t)\backslash\Sigma^{+}(\l,\t)\right|}
\]
\end{prop}

\begin{proof}
Let $\Lambda:Y\rightarrow\t$ be a resolution of singularities for
$\triangle_{\g}^{\t}$, as constructed in Theorem \ref{Thm:resolution of hyperplane arrangements}.
Then, restricting to any ball $B'$ around $0\in\t$, and changing
coordinates using $\Lambda$, one has: 
\begin{equation}
\int_{B'}\frac{\left|\triangle_{\l}^{\t}(X)\right|_{\C}^{m}}{\left|\triangle_{\g,\l}^{\t}(X)\right|_{\C}^{s}}dX=\int_{B'}\frac{\left|\triangle_{\l}^{\t}(X)\right|_{\C}^{s+m}}{\left|\triangle_{\g}^{\t}(X)\right|_{\C}^{s}}dX=\int_{\Lambda^{-1}(B')}\frac{\left|\triangle_{\l}^{\t}\circ\Lambda(y)\right|_{\C}^{s+m}}{\left|\triangle_{\g}^{\t}\circ\Lambda(y)\right|_{\C}^{s}}\left|\operatorname{Jac}_{y}(\Lambda)\right|_{\C}dy.\label{eq:changing coordinates}
\end{equation}
The integral in (\ref{eq:changing coordinates}) is finite if and
only if it is finite on each ball in $\Lambda^{-1}(B')$. By Theorem
\ref{Thm:resolution of hyperplane arrangements}, on a small enough
ball $B''$ around each point in $\Lambda^{-1}(B')$, there exist
divisors $E_{W_{i_{j}}}$ for some $W_{i_{1}},\dots,W_{i_{N}}\in L_{\mathrm{dense}}(\mathcal{A}_{\g,\t})$,
and a tuple of non-negative integers $\left\{ (a_{i_{j}},b_{i_{j}},k_{i_{j}})\right\} _{1\leq j\leq N}$
such that after changing coordinates from $B''$ to an open neighborhood
$\widetilde{B''}$ of $0$ in $\C^{\mathrm{rk}(G)}$, one has: 
\[
\triangle_{\g}^{\t}\circ\Lambda(y)=y_{1}^{a_{i_{1}}}\dots y_{N}^{a_{i_{N}}}u_{1}(y)
\]
\[
\triangle_{\l}^{\t}\circ\Lambda(y)=y_{1}^{b_{i_{1}}}\dots y_{N}^{b_{i_{N}}}u_{2}(y)
\]
\[
\operatorname{Jac}_{y}(\Lambda)=y_{1}^{k_{i_{1}}}\dots y_{N}^{k_{i_{N}}}v(y),
\]
with $u_{1},u_{2}$ and $v$ invertible near $0$. Since $\triangle_{\l}^{\t}|\triangle_{\g}^{\t}$,
we have $b_{i_{j}}\leq a_{i_{j}}$ for all $1\leq j\leq N$. Altogether,
the convergence of the integral on the right hand side of (\ref{eq:changing coordinates}),
restricted to $B''$, is equivalent to 
\begin{equation}
\int_{\widetilde{B''}}\frac{\left(\left|y_{1}\right|_{\C}^{b_{i_{1}}}\dots\left|y_{N}\right|_{\C}^{b_{i_{N}}}\right)^{m+s}}{\left|y_{1}\right|_{\C}^{sa_{i_{1}}}\dots\left|y_{N}\right|_{\C}^{sa_{i_{N}}}}\left|y_{1}\right|_{\C}^{k_{i_{1}}}\dots\left|y_{N}\right|_{\C}^{k_{i_{N}}}dy<\infty.\label{eq:reduction to a monomial integral}
\end{equation}
Recall from the footnote in Definition \ref{def:log canonical threshold }
that $\left|\,\cdot\,\right|_{\C}$ is normalized to be the square
of the standard absolute value on $\C$. Hence, (\ref{eq:reduction to a monomial integral})
converges if and only if for all $j=1,\dots,N$:
\[
mb_{i_{j}}+k_{i_{j}}+s(b_{i_{j}}-a_{i_{j}})>-1,
\]
or equivalently,
\[
s<\frac{1+k_{i_{j}}+mb_{i_{j}}}{a_{i_{j}}-b_{i_{j}}}.
\]
Recall from Lemma \ref{lem:bijection betwweb dense elements and standard Levis}
that each $W\in L_{\mathrm{dense}}(\mathcal{A}_{\g,\t})$ is of the
form $W=\t'$, where $\l'=\mathrm{Cent}_{\g}(\t')$ is a Levi such
that $[\l',\l']$ is simple. By (\ref{eq:formula for rank}), (\ref{eq:formula for s})
and by Theorem \ref{Thm:resolution of hyperplane arrangements}(3),
we have:
\[
k_{W}+1=\mathrm{ss.rk}(\l'),\,\,\,\,\,a_{W}=\left|\Sigma^{+}(\l',\t)\right|,\,\,\,\,\:\text{and }b_{W}=\left|\Sigma^{+}(\l,\t)\cap\Sigma^{+}(\l',\t)\right|.
\]
The result then follows since $\Lambda^{-1}(B')$ intersects all divisors
$\{E_{W}\}_{W\in L_{\mathrm{dense}}(\mathcal{A}_{\g,\t})}$. 
\end{proof}

\subsection{\label{subsec:Integrability in gln}Integrability of relative Weyl
discriminants in $\protect\gl_{n}$ }

In this section we consider the case of $\GG=\GL_{n}$, so that $\g=\gl_{n}$.
We follow the notation of $\mathsection$\ref{subsec:Nilpotent-orbits-in gln}.
We give a more detailed version of Proposition \ref{prop:formula for log canonical threshold},
and discuss the relations between $\mathrm{lct}_{\C}(\triangle_{\g,\l}^{\t};\triangle_{\l}^{\t},0)$,
for different Levi subalgebras. For the proof we need Lemmas \ref{lem:structure of Levis},
\ref{lem:elementary lemma} and \ref{lem:volume of partitions}, and
the following definition: given a matrix $A\in\gl_{n}$, and a subset
$I\subseteq[n]$ of size $k$, the\emph{ $I$-th minor of $A$} is
the matrix $A_{I}\in\gl_{n}$, such that:
\[
(A_{I})_{jk}:=\begin{cases}
A_{jk} & \text{if }j,k\in I,\\
0 & \text{otherwise}.
\end{cases}
\]
We write $\gl_{n,I}$ for the subalgebra of $\gl_{n}$ consisting
of all $A_{I}$, for $A\in\gl_{n}$. 
\begin{lem}
\label{lem:structure of Levis}Let $\t\varsubsetneq\l'\subseteq\gl_{n}$
be a Levi subalgebra of $\gl_{n}$. Suppose that $[\l',\l']\simeq\sl_{k}$,
for $k\geq2$.
\begin{enumerate}
\item There exists a permutation $\sigma\in S_{n}\simeq W$, such that $\sigma.([\l',\l'])\hookrightarrow\gl_{n}$
is the embedding $\sl_{k}\hookrightarrow\gl_{n}$, where $\sl_{k}$
is embedded in the lowest right $k\times k$-block. 
\item $\l'$ is the subalgebra spanned by $\t$ and $\gl_{n,I}$, for $I:=\{\sigma^{-1}(1),\dots,\sigma^{-1}(k)\}\subseteq[n]$.
\end{enumerate}
\end{lem}

\begin{proof}
1) Let $\beta_{1},\dots,\beta_{k-1}$ be a basis for the root system
$\Sigma(\l',\t)$. Then $\beta_{1},\dots,\beta_{k-1}$ can be extended
to a basis $\beta_{1},\dots,\beta_{k-1},\beta_{k},\dots,\beta_{n-1}$
of $\Sigma(\gl_{n},\t)$. Since the Weyl group acts transitively on
the set of bases \cite[Chapter 4.2.4, Theorem 7]{OV90}, there exists
a permutation matrix $\sigma$ such that $\sigma(\beta_{1},\dots,\beta_{k-1})$
is a subset of the standard basis $\alpha_{1,2},\alpha_{2,3},\dots,\alpha_{n-1,n}$.
Since $[\l',\l']\simeq\sl_{k}$, by composing with a cyclic permutation,
we may assume that $\sigma(\{\beta_{1},\dots,\beta_{k-1}\})=\{\alpha_{n-k+1,n-k+2},\dots,\alpha_{n-1,n}\}$.

2) Follows from Item (1), by analyzing $\sigma^{-1}(\gl_{k})$ for
$\gl_{k}\hookrightarrow\gl_{n}$, embedded as the lowest right $k\times k$-block. 
\end{proof}
\begin{lem}
\label{lem:elementary lemma}Let $a_{1},\dots,a_{N},b_{1},\dots,b_{N}\in\Z_{\geq1}$.
Suppose that $\frac{a_{1}}{b_{1}}\leq\dots\leq\frac{a_{N}}{b_{N}}$.
Then $\frac{a_{1}}{b_{1}}\leq\frac{a_{1}+a_{2}}{b_{1}+b_{2}}\leq\dots\leq\frac{\sum_{i=1}^{N}a_{i}}{\sum_{i=1}^{N}b_{i}}$.
The same result holds when $\leq$ is replaced with $<,\geq$ or $>$. 
\end{lem}

\begin{proof}
We prove the claim for $\leq$, by induction on $n$, where the other
cases are similar. Suppose that $b_{n+1}a_{i}\leq a_{n+1}b_{i}$ for
each $i<n+1$. Hence $b_{n+1}\sum_{i=1}^{n}a_{i}\leq a_{n+1}\sum_{i=1}^{n}b_{i}$.
The lemma now follows by the induction hypothesis:
\[
\left(\sum_{i=1}^{n}a_{i}\right)\left(\sum_{i=1}^{n+1}b_{i}\right)=b_{n+1}\sum_{i=1}^{n}a_{i}+\left(\sum_{i=1}^{n}a_{i}\right)\left(\sum_{i=1}^{n}b_{i}\right)\leq\left(\sum_{i=1}^{n+1}a_{i}\right)\left(\sum_{i=1}^{n}b_{i}\right).\qedhere
\]
\end{proof}
\begin{rem}
In the statements and proofs below, use the convention that $\binom{n}{k}=0$
if $k>n$. 
\end{rem}

\begin{lem}
\label{lem:volume of partitions}Let $\eta=(\eta_{1},...,\eta_{N}),\eta'=(\eta'_{1},...,\eta'_{N'})\in\mathcal{P}(n)$
be two partitions, with $\eta_{N},\eta'_{N'}\geq1$, such that $\eta'\preceq\eta$.
Then:
\begin{equation}
\sum_{l=1}^{N'}\binom{\eta'_{l}}{2}\leq\sum_{l=1}^{N}\binom{\eta_{l}}{2}.\label{eq:volume of partitions}
\end{equation}
\end{lem}

\begin{proof}
Since $\eta'\preceq\eta$, the partition $\eta$ can be obtained from
$\eta'$ by successive steps of moving one box in the Young diagram
of $\eta'$ from part $l$ to part $l'<l$ (see e.g.~\cite[Proposition 2.3]{Bry73}).
Since $\binom{\eta'_{l'}+1}{2}-\binom{\eta'_{l'}}{2}=\eta'_{l'}$
and $\binom{\eta'_{l}-1}{2}-\binom{\eta'_{l}}{2}=-(\eta'_{l}-1)$,
each such operation increases the sum $\sum_{l=1}^{N'}\binom{\eta'_{l}}{2}$
by a value of $\eta'_{l'}-\eta'_{l}+1>0$. 
\end{proof}
\begin{prop}
\label{Prop:lct of relative Weyl discriminant}Let $\g=\gl_{n}$.
Let $(n)\neq\nu=(\nu_{1},\dots,\nu_{N})\in\mathcal{P}(n)$ be a partition,
and let $\lambda=(\lambda_{1},\dots,\lambda_{M}):=\nu^{\mathrm{t}}$
be its conjugate partition.
\begin{enumerate}
\item Set $\l_{\nu}=\gl_{\nu_{1}}\oplus\dots\oplus\gl_{\nu_{N}}$ and $\p_{\nu}=\l_{\nu}\oplus\n_{\nu}$
as before, then 
\[
\lct_{\C}(\triangle_{\g,\l_{\nu}}^{\t})=\frac{\mathrm{ss.rk}(\g)}{\left|\Sigma^{+}(\g,\t)\backslash\Sigma^{+}(\l_{\nu},\t)\right|}=\frac{\mathrm{ss.rk}(\g)}{\dim\n_{\nu}}.
\]
\item For $1\le s\le\nu_{1}$, set $n_{s}:=\sum_{i=1}^{s}\lambda_{i}$.
Then for every integer $m\geq1$, 
\begin{equation}
\lct_{\C}(\triangle_{\g,\l_{\nu}}^{\t};(\triangle_{\l_{\nu}}^{\t})^{m},0)=\begin{cases}
\underset{1\leq s\leq\nu_{1}}{\min}\frac{\left(n_{s}-1\right)+\sum_{j=1}^{s}(j-1)\lambda_{j}}{\binom{n_{s}}{2}-\sum_{j=1}^{s}(j-1)\lambda_{j}} & \text{if }m=1,\\
\frac{2}{\lambda_{1}} & \text{if }m\geq2.
\end{cases}\label{eq:combinatorial formula for relative lct}
\end{equation}
\end{enumerate}
\end{prop}

\begin{proof}
Recall that by Proposition \ref{prop:formula for log canonical threshold}
we have, 
\begin{equation}
\lct_{\C}(\triangle_{\g,\l_{\nu}}^{\t};(\triangle_{\l_{\nu}}^{\t})^{m},0)=\underset{\substack{\t\varsubsetneq\l'\subseteq\g\\{}
[\l',\l']\text{ is simple}
}
}{\min}\frac{\mathrm{ss.rk}(\l')+m\left|\Sigma^{+}(\l_{\nu},\t)\cap\Sigma^{+}(\l',\t)\right|}{\left|\Sigma^{+}(\l',\t)\backslash\Sigma^{+}(\l_{\nu},\t)\right|}.\label{eq:non refined formula for relative lct}
\end{equation}
Note that for each $k\geq2$, and $\l'$ such that $[\l',\l']\simeq\sl_{k}$,
the right hand side of (\ref{eq:non refined formula for relative lct})
is minimal when $\left|\Sigma^{+}(\l_{\nu},\t)\cap\Sigma^{+}(\l',\t)\right|$
is minimal. Recall from Lemma \ref{lem:structure of Levis} that $\l'=\t+\gl_{n,I}$
for a subset $I=\{i_{1},...,i_{k}\}\subseteq[n]$ of size $k$. Denote
$J_{l,\nu}:=\{1+\sum_{i=1}^{l-1}\nu_{i},2+\sum_{i=1}^{l-1}\nu_{i}\dots,\sum_{i=1}^{l}\nu_{i}\}\subseteq[n]$.
Note that $\left|\Sigma^{+}(\l_{\nu},\t)\cap\Sigma^{+}(\l',\t)\right|$
is equal to $\sum_{l=1}^{N}\binom{\#I\cap J_{l,\nu}}{2}$, so it is
determined by the multiset $\left\{ \#I\cap J_{l,\nu}\right\} _{l=1}^{N}$.
Let us first replace $I$ with a more convenient subset $\widetilde{I}$
with the same multiset of intersections. Since $\nu_{1}\geq\ldots\geq\nu_{N}$,
there exists a subset $\widetilde{I}\subseteq[n]$, such that $\left\{ \#I\cap J_{l,\nu}\right\} _{l=1}^{N}=\left\{ \#\widetilde{I}\cap J_{l,\nu}\right\} _{l=1}^{N}$
and $\#\widetilde{I}\cap J_{1,\nu}\geq\ldots\geq\#\widetilde{I}\cap J_{N,\nu}$
is a partition $\eta\vdash k$. We then have:
\[
\left|\Sigma^{+}(\l_{\nu},\t)\cap\Sigma^{+}(\l',\t)\right|=\sum_{l=1}^{N}\binom{\#I\cap J_{l,\nu}}{2}=\sum_{l=1}^{N}\binom{\#\widetilde{I}\cap J_{l,\nu}}{2}=\sum_{l=1}^{N}\binom{\eta_{l}}{2}.
\]
Set $\mathcal{P}_{\nu}(k):=\left\{ \eta\vdash k:\eta_{l}\leq\nu_{l},\forall l\in[N]\right\} $,
and note that for each $k\geq2$
\begin{equation}
\underset{\substack{\t\varsubsetneq\l'\subseteq\g\\{}
[\l',\l']\simeq\sl_{k}
}
}{\min}\left|\Sigma^{+}(\l_{\nu},\t)\cap\Sigma^{+}(\l',\t)\right|=\underset{\eta\in\mathcal{P}_{\nu}(k)}{\min}\sum_{l=1}^{N}\binom{\eta_{l}}{2}.\label{eq:minimal expression}
\end{equation}
There exists a unique partition $\eta_{\min,\nu}$ in $\mathcal{P}_{\nu}(k)$,
such that $\eta_{\min,\nu}\preceq\eta$ for all $\eta\in\mathcal{P}_{\nu}(k)$.
The partition $\eta_{\min,\nu}$ is obtained by adding $k$ boxes,
one by one, to fill the columns in the Young diagram of $\nu$, from
top to bottom, left to right. The first column can fill $\lambda_{1}$
possible boxes, the second column has $\lambda_{2}$ possible boxes,
etc. Explicitly, the $u$-th box will be entered at the $\phi_{u}$-th
column, such that:
\[
\min(\sum_{i=1}^{\phi_{u}-1}\lambda_{i},k)+1\leq u\leq\min(\sum_{i=1}^{\phi_{u}}\lambda_{i},k).
\]
By Lemma \ref{lem:volume of partitions}, the minimum in (\ref{eq:minimal expression})
is obtained for this $\eta_{\min,\nu}$. Hence, by setting $\psi_{\nu,k}:=\sum_{u=1}^{k}(\phi_{u}-1)$,
by observing that $\sum_{l=1}^{N}\binom{(\eta_{\min,\nu})_{l}}{2}=\psi_{\nu,k}$,
and by (\ref{eq:minimal expression}), we now have:
\begin{equation}
\lct_{\C}(\triangle_{\g,\l_{\nu}}^{\t})=\underset{1\leq k\leq n}{\min}\underset{\substack{\t\varsubsetneq\l'\subseteq\g\\{}
[\l',\l']\simeq\sl_{k}
}
}{\min}\frac{\mathrm{ss.rk}(\l')}{\left|\Sigma^{+}(\l',\t)\backslash\Sigma^{+}(\l_{\nu},\t)\right|}=\underset{2\leq k\leq n}{\min}\frac{k-1}{\binom{k}{2}-\psi_{\nu,k}}.\label{eq:refined formula for lct}
\end{equation}
Note that $u+1-\phi_{u+1}$ is increasing in $u$. By Lemma \ref{lem:elementary lemma},
and since $\phi_{1}=1$, the following expression:
\[
\frac{k-1}{\binom{k}{2}-\psi_{\nu,k}}=\frac{\sum_{u=1}^{k-1}1}{\sum_{u=1}^{k-1}(u-(\phi_{u+1}-1))},
\]
is decreasing in $k$. Hence, the minimum in (\ref{eq:refined formula for lct})
is obtained when $k=n$, and Item (1) follows.

For Item (2), note that for $m\geq1$,
\begin{equation}
\lct_{\C}(\triangle_{\g,\l_{\nu}}^{\t};(\triangle_{\l_{\nu}}^{\t})^{m},0)=\underset{2\leq k\leq n}{\min}\frac{k-1+m\cdot\psi_{\nu,k}}{\binom{k}{2}-\psi_{\nu,k}}=\underset{2\leq k\leq n}{\min}\frac{\sum_{u=1}^{k-1}(m\phi_{u+1}-m+1)}{\sum_{u=1}^{k-1}(u+1-\phi_{u+1})}.\label{eq:refined formula for relative lct}
\end{equation}
Note that on each interval $n_{s-1}<u\leq n_{s}$, we have $\phi_{u}=s$,
and hence $\frac{m\phi_{u}-m+1}{(u-\phi_{u})}=\frac{ms-m+1}{u-s}$
is a decreasing function of $u$. Hence, by Lemma \ref{lem:elementary lemma},
the function $\frac{\sum_{u=1}^{k-1}(m\phi_{u+1}-m+1)}{\sum_{u=1}^{k-1}(u+1-\phi_{u+1})}$
is piecewise monotone in $k\in\{2,\ldots,n\}$, and therefore the
minimum in (\ref{eq:refined formula for relative lct}) must be obtained
when $k\in\left\{ \lambda_{1},\lambda_{1}+\lambda_{2},\dots,\sum_{j=1}^{\nu_{1}}\lambda_{j}\right\} $.
Since $\psi_{\nu,n_{s}}=\sum_{j=1}^{s}(j-1)\lambda_{j}$, we deduce
that:
\begin{equation}
\lct_{\C}(\triangle_{\g,\l_{\nu}}^{\t};(\triangle_{\l_{\nu}}^{\t})^{m},0)=\underset{1\leq s\leq\nu_{1}}{\min}\frac{\left(n_{s}-1\right)+m\sum_{j=1}^{s}(j-1)\lambda_{j}}{\binom{n_{s}}{2}-\sum_{j=1}^{s}(j-1)\lambda_{j}},\label{eq:non simplified formula for rlct}
\end{equation}
for every $m\geq1$. In particular, this implies the proposition in
the case $m=1$. 

To deal with the case that $m\geq2$, it is left to show that the
right hand side of (\ref{eq:non simplified formula for rlct}) is
equal to $\frac{2}{\lambda_{1}}$. Write $a_{s}^{(m)}:=\left(n_{s}-1\right)+m\sum_{j=1}^{s}(j-1)\lambda_{j}$
and $b_{s}:=\binom{n_{s}}{2}-\sum_{j=1}^{s}(j-1)\lambda_{j}$ so that
Eq.~(\ref{eq:non simplified formula for rlct}) becomes $\lct_{\C}(\triangle_{\g,\l_{\nu}}^{\t};(\triangle_{\l_{\nu}}^{\t})^{m},0)=\underset{1\leq s\leq\nu_{1}}{\min}\frac{a_{s}^{(m)}}{b_{s}}$.
Note that $\frac{a_{1}^{(m)}}{b_{1}}=\frac{\lambda_{1}-1}{\binom{\lambda_{1}}{2}}=\frac{2}{\lambda_{1}}$
and hence $\lct_{\C}(\triangle_{\g,\l_{\nu}}^{\t};(\triangle_{\l_{\nu}}^{\t})^{m},0)\leq\frac{2}{\lambda_{1}}$
for every $m\in\N$. To show the inequality in the other direction,
it is enough to show that 
\[
\frac{a_{u}^{(m)}-a_{u-1}^{(m)}}{b_{u}-b_{u-1}}\geq\frac{2}{\lambda_{1}}\,\,\,\text{ for every }u\in\left\{ 1,\dots,\nu_{1}\right\} ,\tag{\ensuremath{\star}}
\]
where $a_{0}^{(m)}=b_{0}=0$, since then by Lemma \ref{lem:elementary lemma}
we have $\frac{a_{s}^{(m)}}{b_{s}}=\frac{\sum_{u=1}^{s}\left(a_{u}^{(m)}-a_{u-1}^{(m)}\right)}{\sum_{u=1}^{s}\left(b_{u}-b_{u-1}\right)}\geq\frac{2}{\lambda_{1}}$.
Note that 
\begin{align}
 & \frac{a_{u}^{(m)}-a_{u-1}^{(m)}}{b_{u}-b_{u-1}}=\frac{\lambda_{u}+m(u-1)\lambda_{u}}{\binom{n_{u-1}+\lambda_{u}}{2}-\binom{n_{u-1}}{2}-(u-1)\lambda_{u}}\nonumber \\
= & \frac{2\left(m(u-1)+1\right)\lambda_{u}}{\left(n_{u-1}+\lambda_{u}\right)\left(n_{u-1}+\lambda_{u}-1\right)-n_{u-1}(n_{u-1}-1)-2(u-1)\lambda_{u}}\nonumber \\
= & \frac{2\left(m(u-1)+1\right)\lambda_{u}}{2\lambda_{u}n_{u-1}+\lambda_{u}^{2}-(2u-1)\lambda_{u}}=\frac{2\left(m(u-1)+1\right)}{2n_{u-1}+\lambda_{u}-(2u-1)}.\label{eq:simplification}
\end{align}
Applying the inequalities $\lambda_{u}\leq\lambda_{1}$, $n_{u-1}\leq(u-1)\lambda_{1}$
and $m\geq2$ to Eq.~(\ref{eq:simplification}) yields: 
\[
\frac{a_{u}^{(m)}-a_{u-1}^{(m)}}{b_{u}-b_{u-1}}\geq\frac{2\left(2u-1\right)}{(2u-1)\lambda_{1}-(2u-1)}=\frac{2}{\lambda_{1}-1}>\frac{2}{\lambda_{1}},
\]
which implies $(\star)$. This concludes the proof. 
\end{proof}
\begin{cor}
\label{cor:geometric formula for relative lct}In the setting of Proposition
\ref{Prop:lct of relative Weyl discriminant}, for each $1\le s\le\nu_{1}$,
set $[\lambda]_{s}:=(\lambda_{1},\ldots,\lambda_{s})$, let $\l_{s}\subseteq\gl_{n_{s}}$
be the Levi subalgebra corresponding to the conjugate partition $([\lambda]_{s})^{t}\in\mathcal{P}(n_{s})$
and let $\p_{s}=\l_{s}\oplus\n_{s}$ be the Levi decomposition of
the corresponding parabolic subalgebra. Let $\n_{\l_{s}}$ be a maximal
nilpotent subalgebra of $\l_{s}$. Then:
\[
\lct_{\C}(\triangle_{\gl_{n},\l_{\nu}}^{\t};\triangle_{\l_{\nu}}^{\t},0)=\underset{1\leq s\leq\nu_{1}}{\min}\frac{\mathrm{ss.rk}(\gl_{n_{s}})+\dim\n_{\l_{s}}}{\dim\n_{s}}.
\]
\end{cor}

\begin{proof}
Fix $1\le s\le\nu_{1}$ and set $\nu^{s}=([\lambda]_{s})^{t}$. Applying
the identity $\sum_{j=1}^{s}(j-1)\eta_{j}=\sum_{j=1}^{\eta_{1}}\binom{(\eta^{t})_{j}}{2}$
(see \cite[Page 3, Equation 1.6]{Mac95}) with $\eta=[\lambda]_{s}$,
to Eq.~(\ref{eq:combinatorial formula for relative lct}), we get,
\begin{equation}
\frac{\left(\sum_{j=1}^{s}j\lambda_{j}\right)-1}{\binom{\sum_{j=1}^{s}\lambda_{j}}{2}-\sum_{j=1}^{s}(j-1)\lambda_{j}}=\frac{-1+n_{s}+\sum_{j=1}^{s}(j-1)\lambda_{j}}{\binom{n_{s}}{2}-\sum_{j=1}^{s}(j-1)\lambda_{j}}=\frac{\mathrm{ss.rk}(\gl_{n_{s}})+\sum_{j=1}^{\lambda_{1}}\binom{(\nu^{s})_{j}}{2}}{\binom{n_{s}}{2}-\sum_{j=1}^{\lambda_{1}}\binom{(\nu^{s})_{j}}{2}}.\label{eq:auxiliary formula}
\end{equation}
With $\l_{s},\p_{s},\n_{s}$ as in the statement of the corollary,
one can observe that $\dim\n_{\l_{s}}=\sum_{j=1}^{\lambda_{1}}\binom{(\nu^{s})_{j}}{2}$.
We therefore have: 
\begin{equation}
\dim\n_{s}=\binom{n_{s}}{2}-\sum_{j=1}^{\lambda_{1}}\binom{(\nu^{s})_{j}}{2}.\label{eq:identity}
\end{equation}
Applying (\ref{eq:identity}) to (\ref{eq:auxiliary formula}) finishes
the proof. 
\end{proof}
We can now prove the following corollary: 
\begin{cor}
\label{cor:dominant Levi means large rlct}Let $\nu,\nu'\in\mathcal{P}(n)$.
Suppose that $\nu\succeq\nu'$. Then:
\[
\lct_{\C}(\triangle_{\g,\l_{\nu}}^{\t};\triangle_{\l_{\nu}}^{\t},0)\geq\lct_{\C}(\triangle_{\g,\l_{\nu'}}^{\t};\triangle_{\l_{\nu'}}^{\t},0).
\]
\end{cor}

\begin{proof}
In the notation of Proposition \ref{Prop:lct of relative Weyl discriminant},
if $\nu\succeq\nu'$ then also $\eta_{\min,\nu}\succeq\eta_{\min,\nu'}$.
By Lemma \ref{lem:volume of partitions},
\[
\psi_{\nu,k}=\sum_{l=1}^{N}\binom{(\eta_{\min,\nu})_{l}}{2}\geq\sum_{l=1}^{N}\binom{(\eta_{\min,\nu'})_{l}}{2}=\psi_{\nu',k},\text{ for every }2\leq k\leq n.
\]
We therefore get:
\[
\lct_{\C}(\triangle_{\g,\l_{\nu'}}^{\t};\triangle_{\l_{\nu'}}^{\t},0)=\underset{2\leq k\leq n}{\min}\frac{k-1+\psi_{\nu',k}}{\binom{k}{2}-\psi_{\nu',k}}\leq\underset{2\leq k\leq n}{\min}\frac{k-1+\psi_{\nu,k}}{\binom{k}{2}-\psi_{\nu,k}}=\lct_{\C}(\triangle_{\g,\l_{\nu}}^{\t};\triangle_{\l_{\nu}}^{\t},0).\qedhere
\]
\end{proof}

\section{\label{sec:Geometric construction of Fourier transform}Fourier transform
of stable Richardson orbital integrals through geometry}

Let $\GG$ be a connected reductive algebraic $F$-group, let $\ggm:=\Lie(\GG)$.
Let $\mathcal{O}\in\mathcal{O}(\mathcal{N}_{\ggm})$ be a nilpotent
orbit, with polarization $\PP=\LL\cdot\NN$ where $\Lie(\PP)=\ppm=\llm\oplus\nnm$.
We write $G,\g,\p,\l,\n$ for the $F$-points of $\GG,\ggm,\ppm,\llm,\nnm$. 

In this section we prove Theorem \ref{thm A: epsilon of Fourier of orbital integral is lct(D(G/M))}. 
\begin{thm}[Theorem \ref{thm A: epsilon of Fourier of orbital integral is lct(D(G/M))}]
\label{thm:Theorem A in section 4}Let $\mathcal{O}$ be a Richardson
orbit in $\ggm$ with a polarization $\PP=\LL\cdot\NN$. Let $\mathcal{O}_{\mathrm{st}}:=\mathcal{O}(F)$
be the corresponding stable Richardson orbit. Then 
\[
\epsilon_{\star}(\widehat{\xi}_{\mathcal{O}_{\mathrm{st}}})=2\lct_{F}(D_{\g,\l},0).
\]
\end{thm}

To prove the theorem, we need to describe $\widehat{\xi}_{\mathcal{O}_{\mathrm{st}}}$.
Theorem \ref{thm:Theorem A in section 4} has an alternative proof
using harmonic analysis, using methods presented in \cite[Section 13]{Kot05},
see Remark \ref{rem:A-field-dependent construction}. However, we
present a geometric construction, which produces $\widehat{\xi}_{\mathcal{O}(F)}=\widehat{\xi}_{\mathcal{O}_{\mathrm{st}}}$
after specializing to any local field $F$. 

\subsection{\label{subsec:geometric construction}A geometric description of
$\widehat{\xi}_{\mathcal{O}_{\mathrm{st}}}$ }

For a smooth, finite type $F$-scheme $X$, let $\mathcal{O}_{X,x}$
be the local ring of $X$ at $x$, and let $\m_{x}$ be its maximal
ideal. The cotangent space of $X$ at $x\in X$ is $T_{x}^{*}X:=\m_{x}/\m_{x}^{2}$,
which is a vector space over the residue field $k_{x}:=\mathcal{O}_{X,x}/\m_{x}$.
Similarly, the tangent space is $T_{x}X:=\left(\m_{x}/\m_{x}^{2}\right)^{*}$.
For each $1\leq m\leq\dim X$, let $\bigwedge^{m}T_{x}^{*}X$ be the
space of $m$-forms of $X$ at $x$. Let $TX$ (resp.~$T^{*}X$)
be the tangent (resp.~cotangent) sheaf of $X$ and let $\Omega^{m}X:=\bigwedge^{m}T^{*}X$
be the sheaf of $m$-differential forms on $X$. We write $\Omega^{\mathrm{top}}X:=\Omega^{\dim X}X$
for the sheaf of top differential forms. The ring of sections of a
sheaf $\mathcal{G}$ over an open subset $U\subseteq X$ is denoted
$\Gamma(U,\mathcal{G})$. An element $\omega\in\Gamma(X,\Omega^{m}X)$
is called a \emph{regular $m$-form} on $X$. An element $\omega\in\Gamma(U,\Omega^{m}X)$,
for an open set $U\subseteq X$, is called a \emph{rational $m$-form}
on $X$, and two rational forms are equivalent if they agree on a
smaller open set.

Note that for $U,X$ as above, $X(F)$ is a smooth $F$-analytic manifold,
and each $\omega\in\Gamma(U,\Omega^{\mathrm{top}}X)$ gives rise to
a rational differential form $\omega_{F}$ on $X(F)$. To such $\omega_{F}$,
we can attach a natural measure $\mu_{\omega_{F}}:=\left|\omega_{F}\right|_{F}$
on $X(F)$ (see e.g.~\cite[Definition 3.1]{AA16}).

For each parabolic subgroup $\PP\leq\GG$ denote by $\underline{\mathcal{P}}:=\{g\PP g^{-1}:g\in\GG\}\simeq\GG/\PP$
the corresponding flag variety. For a parabolic subgroup $\QQ$ we
denote by $\qqm$ its Lie algebra and by $\nnm_{\qqm}$ the radical
of $\frak{\qqm}$. We define a bundle $\widetilde{\mathcal{O}}$ over
$\underline{\mathcal{P}}$ by
\[
\widetilde{\mathcal{O}}:=\GG\times_{\PP}\nnm\simeq\left\{ (\QQ,X)\in\underline{\mathcal{P}}\times\ggm:X\in\nnm_{\qqm}\right\} \subset\underline{\mathcal{P}}\times\ggm.
\]
The variety $\widetilde{\mathcal{O}}$ is smooth, and it can be interpreted
as the total space of the cotangent bundle $T^{*}(\underline{\mathcal{P}})$
of $\underline{\mathcal{P}}$. Moreover, there is a $\GG$-action
on $\widetilde{\mathcal{O}}$ by $g.(\QQ,X):=(g\QQ g^{-1},gXg^{-1})$.
Finally, when $\PP$ is a polarization of a Richardson orbit $\mathcal{O}$,
projecting from $\widetilde{\mathcal{O}}\subset\underline{\mathcal{P}}\times\ggm$
to $\g$ gives rise to the \emph{generalized Springer resolution}
$\pi_{\mathcal{O}}:\widetilde{\mathcal{O}}\to\overline{\mathcal{O}}$. 
\begin{prop}
The map $\pi_{\mathcal{O}}$ is a $\GG$-equivariant, projective and
surjective morphism, which is \'etale over $\mathcal{O}$. 
\end{prop}

\begin{proof}
Since $\mathcal{O}$ is a Richardson orbit, by e.g.~\cite[Section 4.9, Eq. 1 and Lemma in 8.8]{Jan04},
we have $\dim\overline{\mathcal{O}}=2\dim(\GG/\PP)=\dim T^{*}(\GG/\PP)$,
and $\GG.\nnm=\overline{\mathcal{O}}$, so that $\pi_{\mathcal{O}}$
is well defined. The map $\pi_{\mathcal{O}}$ is projective e.g.~by
\cite[Page 3, Construction]{Hes79}, and it is clearly $G$-invariant.
Since $\pi_{\mathcal{O}}$ is dominant, it is smooth over a point
in $\mathcal{O}$, and hence it is smooth (hence \'etale) over $\mathcal{O}$. 
\end{proof}
Let $\NN^{-}$ be the unipotent radical of the opposite parabolic
to $\PP$, and let $\psi:\NN^{-}\times\nnm\rightarrow\widetilde{\mathcal{O}}$
be the map defined by $\psi(x,Y):=(x\PP x^{-1},\Ad(x)Y)$. Let $\omega_{\NN^{-}}$
be an $\NN^{-}$-invariant top form on $\NN^{-}$, and let $\omega_{\nnm}$
be an $\nnm$-invariant top form on $\nnm$. Both forms are unique
up to a constant. Moreover, $\omega_{\nnm}$ is also $(\Ad(\PP),\delta_{\PP})$-equivariant,
where $\delta_{\PP}$ is the modular character of $\PP$.
\begin{lem}
\label{lem:invariant top form on an alteration}Keeping the notation
as above, the following hold:
\begin{enumerate}
\item There exist $\GG$-invariant regular top forms $\omega_{\widetilde{\mathcal{O}}}$
on $\widetilde{\mathcal{O}}$ and $\omega_{\mathcal{O}}$ on $\mathcal{O}$,
such that $\omega_{\widetilde{\mathcal{O}}}=(\pi_{\mathcal{O}})^{*}\omega_{\mathcal{O}}$. 
\item The map $\psi$ is an isomorphism onto its image $(\NN^{-}.\PP)\times_{\PP}\nnm$,
and $\psi^{*}\omega_{\widetilde{\mathcal{O}}}=\omega_{\NN^{-}}\wedge\omega_{\nnm}$. 
\end{enumerate}
\end{lem}

\begin{proof}
(1) Since $\widetilde{\mathcal{O}}$ is identified with the cotangent
bundle $T^{*}(\underline{\mathcal{P}})$ of a smooth variety $\underline{\mathcal{P}}$,
it has a canonical non-degenerate regular $2$-form $\eta_{\mathcal{\widetilde{O}}}$,
which is $\GG$-invariant, as $\underline{\mathcal{P}}$ is. In addition,
there exists a $\GG$-invariant non-degenerate regular $2$-form $\eta_{\mathcal{O}}$
on $\mathcal{O}$, called the \emph{Kirillov\textendash Kostant form}
(see e.g.~\cite[II.3]{Aud04} or \cite[Section 17.3]{Kot05}). Since
the pullback $(\pi_{\mathcal{O}})^{*}\eta_{\mathcal{O}}$ is $\GG$-invariant,
we have $\eta_{\mathcal{\widetilde{O}}}=(\pi_{\mathcal{O}})^{*}\eta_{\mathcal{O}}$,
up to a constant. Taking the $\dim\GG/2$-exterior power $\eta_{\mathcal{\widetilde{O}}}\wedge\dots\wedge\eta_{\mathcal{\widetilde{O}}}$
of $\eta_{\mathcal{\widetilde{O}}}$ yields a $\GG$-invariant regular
top form $\omega_{\widetilde{\mathcal{O}}}$ on $\widetilde{\mathcal{O}}$,
which is a pullback $\omega_{\widetilde{\mathcal{O}}}=(\pi_{\mathcal{O}})^{*}\omega_{\mathcal{O}}$
of $\omega_{\mathcal{O}}:=\eta_{\mathcal{\mathcal{O}}}\wedge\dots\wedge\eta_{\mathcal{\mathcal{O}}}$.

For Item (2), note that $\psi:\NN^{-}\times\nnm\rightarrow\widetilde{\mathcal{O}}$
is $\NN^{-}$-equivariant from the left, and since $\omega_{\widetilde{\mathcal{O}}}$
is $\GG$-invariant, $\psi^{*}\omega_{\widetilde{\mathcal{O}}}$ must
be of the form $\omega_{\NN^{-}}\wedge\widetilde{\omega}$ for some
top form $\widetilde{\omega}$ on $\nnm$. By the $\NN^{-}$-invariance,
to show that $\psi^{*}\omega_{\widetilde{\mathcal{O}}}=\omega_{\NN^{-}}\wedge\omega_{\nnm}$
it is enough to show that $\left(\psi^{*}\omega_{\widetilde{\mathcal{O}}}\right)|_{\{e\}\times\nnm}=\omega_{\NN^{-}}\wedge\widetilde{\omega}|_{\{e\}\times\nnm}=\omega_{\NN^{-}}\wedge\omega_{\nnm}|_{\{e\}\times\nnm}$.
Note that $\omega_{\NN^{-}}|_{\{e\}}$ is $(\PP,\delta_{\PP}^{-1})$-invariant.
Since $\left(\psi^{*}\omega_{\widetilde{\mathcal{O}}}\right)|_{\{e\}\times\nnm}$
is $\PP$-invariant, this implies that $\widetilde{\omega}$ is $(\PP,\delta_{\PP})$-invariant,
thus equals to $\omega_{\nnm}$, up to a constant. 
\end{proof}
\begin{defn}
Following Lemma \ref{lem:invariant top form on an alteration}, we
set:
\begin{enumerate}
\item $\mu_{\widetilde{\mathcal{O}}(F)}:=\left|\omega_{\widetilde{\mathcal{O}}}\right|_{F}$
which is a $G$-invariant measure on $\widetilde{\mathcal{O}}(F)$. 
\item $\xi_{\widetilde{\mathcal{O}}(F)}:=i_{*}\mu_{\widetilde{\mathcal{O}}(F)}$
where $i$ is the closed embedding $i:\widetilde{\mathcal{O}}\hookrightarrow\underline{\mathcal{P}}\times\ggm.$
\end{enumerate}
\end{defn}

\begin{lem}
\label{lem: push of measure on bundle is an orbital integral}The
measure $(\pi_{\mathcal{O}})_{*}(\mu_{\widetilde{\mathcal{O}}(F)})$
is equal to $\mu_{\mathcal{O}_{\mathrm{st}}}$, up to a constant. 
\end{lem}

\begin{proof}
Note that $\left|\omega_{\mathcal{O}}\right|_{F}$ is a $\GG(F)$-invariant
measure on $\mathcal{O}(F)$, but in fact it is stably invariant,
since $\omega_{\mathcal{O}}$ is $\GG(\overline{F})$-invariant. Hence
$\left|\omega_{\mathcal{O}}\right|_{F}=\mu_{\mathcal{O}_{\mathrm{st}}}$.
By Lemma \ref{lem:invariant top form on an alteration}(1), if $M$
denotes the (constant) size of the fibers of $\pi_{\mathcal{O}}:\pi_{\mathcal{O}}^{-1}(\mathcal{O}(F))\rightarrow\mathcal{O}(F)$,
then $(\pi_{\mathcal{O}})_{*}(\mu_{\widetilde{\mathcal{O}}(F)})=M\cdot\mu_{\mathcal{O}_{\mathrm{st}}}$.
\end{proof}
Recall we defined in $\mathsection$\ref{subsec:Fourier-transform-of}
the Fourier transform $\mathcal{F}:\mathcal{S}(\g)\rightarrow\mathcal{S}(\g)$,
also denoted $f\mapsto\widehat{f}$, where we identified $\g$ with
$\g^{*}$ using the Killing form. Since $\widetilde{\mathcal{O}}\subseteq\underline{\mathcal{P}}\times\ggm$,
we may consider the fiber-wise Fourier transform of functions and
distributions on $\widetilde{\mathcal{O}}$ as follows. 
\begin{defn}
\label{def:fiberwise Fourier transform}Let $M$ be a smooth $F$-analytic
manifold, let $f\in\mathcal{S}(M\times\g)$ and $\xi\in\mathcal{S}^{*}(M\times\g)$.
For each $(m,Y)\in M\times\g$, set $f_{m}:=f|_{\{m\}\times\g}:\g\to\C$,
and define 
\[
\mathcal{F}_{M,\g}(f)(m,Y)=\widehat{f_{m}}(Y)\text{ \,\,\,\,\,\text{ and \,\,\,\,\,\,}}\langle\mathcal{F}_{M,\g}(\xi),f\rangle:=\langle\xi,\mathcal{F}_{M,\g}(f)\rangle.
\]
\end{defn}

Let $\Pi:\underline{\mathcal{P}}\times\ggm\rightarrow\ggm$ be the
projection, and let $\mathcal{P}=\underline{\mathcal{P}}(F)\simeq G/P$
(see e.g. \cite[Theorem 4.13a]{BT65}). The next lemma follows from
the above definition.
\begin{lem}
\label{lem: push by proper projection commutes with Fourier}For each
$f\in S(\g)$, we have $\mathcal{F}_{\mathcal{P},\g}(\Pi^{*}f)=\Pi^{*}(\mathcal{F}(f))$.
In particular, for each $\xi\in S^{*}(\mathcal{P}\times\g)$, we have
$\Pi_{*}(\mathcal{F}_{\mathcal{P},\g}(\xi))=\mathcal{F}(\Pi_{*}\xi)$. 
\end{lem}

Recall that $\nnm^{\perp}=\ppm$, where $\perp$ denotes the orthogonal
complement with respect to the Killing form. Let $\widetilde{\mathcal{O}}^{\perp}:=\GG\times_{\PP}\ppm\subseteq\underline{\mathcal{P}}\times\ggm$,
let $i^{\perp}:\widetilde{\mathcal{O}}^{\perp}\to\underline{\mathcal{P}}\times\ggm$
be the corresponding closed embedding and $\pi^{\perp}:\widetilde{\mathcal{O}}^{\perp}\to\ggm$
be the restriction to $\widetilde{\mathcal{O}}^{\perp}$ of the projection
$\Pi:\underline{\mathcal{P}}\times\ggm\rightarrow\ggm$ . Since $\nnm^{\perp}=\ppm$,
the bundle $\widetilde{\mathcal{O}}^{\perp}$ can be identified with
the conormal bundle of $\underline{\mathcal{P}}$ inside $\underline{\mathcal{P}}\times\ggm$,
and furthermore $\dim\widetilde{\mathcal{O}}^{\perp}=\dim\GG-\dim\PP+\dim\ppm=\dim\ggm$. 
\begin{lem}
There exists a $\GG$-invariant top form $\omega_{\widetilde{\mathcal{O}}^{\perp}}$
on $\widetilde{\mathcal{O}}^{\perp}$, such that if we set $\mu_{\widetilde{\mathcal{O}}^{\perp}(F)}:=\left|\omega_{\widetilde{\mathcal{O}}^{\perp}}\right|_{F}$
and $\xi_{\widetilde{\mathcal{O}}^{\perp}(F)}:=(i^{\perp})_{*}\mu_{\widetilde{\mathcal{O}}^{\perp}(F)}$,
then $\xi_{\widetilde{\mathcal{O}}^{\perp}(F)}=\mathcal{F}_{\mathcal{P},\g}(\xi_{\widetilde{\mathcal{O}}(F)})$. 
\end{lem}

\begin{proof}
Let $\Psi:\NN^{-}\times\ggm\rightarrow\underline{\mathcal{P}}\times\ggm$
be the map $\Psi(x,Y):=(x\PP x^{-1},\Ad(x).Y)$, and let $j_{\nnm}:\NN^{-}\times\nnm\hookrightarrow\NN^{-}\times\ggm$
and $j_{\ppm}:\NN^{-}\times\ppm\hookrightarrow\NN^{-}\times\ggm$
be the standard inclusions. Let $\psi^{\perp}:\NN^{-}\times\ppm\rightarrow\widetilde{\mathcal{O}}^{\perp}$
be the map defined by $\Psi\circ j_{\ppm}$. Let $\omega_{\ppm}$
be a translation invariant top form on $\ppm$, and note that $\omega_{\ppm}$
is $(\PP,\delta_{\PP})$-equivariant. Set $\widetilde{\omega}^{\perp}:=\omega_{\NN^{-}}\wedge\omega_{\ppm}$,
which is a top form on $\NN^{-}\times\ppm$. Similarly to Lemma \ref{lem:invariant top form on an alteration},
the form $\omega^{\perp}:=((\psi^{\perp})^{-1})^{*}\widetilde{\omega}^{\perp}$
extends to a $\GG$-invariant top form $\omega_{\widetilde{\mathcal{O}}^{\perp}}$
on $\widetilde{\mathcal{O}}^{\perp}$.

Since $\mathcal{F}\circ(\Ad(g)^{*}f)=\Ad(g)^{*}\mathcal{F}(f)$ for
$f\in\mathcal{S}(\g)$ and $g\in G$, we have:
\begin{equation}
\mathcal{F}_{N^{-},\g}\circ\Psi^{*}=\Psi^{*}\circ\mathcal{F}_{\mathcal{P},\g}.\label{eq:Fourier commutes}
\end{equation}
Moreover, since $\n^{\perp}=\p$, the following holds:
\[
\mathcal{F}_{N^{-},\g}\left((j_{\nnm})_{*}\left|\omega_{\NN^{-}}\wedge\omega_{\nnm}\right|_{F}\right)=(j_{\ppm})_{*}\left|\omega_{\NN^{-}}\wedge\omega_{\ppm}\right|_{F}.
\]
We therefore deduce that:
\begin{align*}
\mathcal{F}_{\mathcal{P},\g}(\xi_{\widetilde{\mathcal{O}}(F)}) & =\mathcal{F}_{\mathcal{P},\g}\left((\Psi\circ j_{\nnm})_{*}\left|\omega_{\NN^{-}}\wedge\omega_{\nnm}\right|_{F}\right)=\Psi_{*}\mathcal{F}_{N^{-},\g}\left((j_{\nnm})_{*}\left|\omega_{\NN^{-}}\wedge\omega_{\nnm}\right|_{F}\right)\\
 & =\Psi_{*}(j_{\ppm})_{*}\left|\omega_{\NN^{-}}\wedge\omega_{\ppm}\right|_{F}=(i^{\perp})_{*}\left|\omega_{\widetilde{\mathcal{O}}^{\perp}}\right|_{F}=\xi_{\widetilde{\mathcal{O}}^{\perp}(F)}.\qedhere
\end{align*}
\end{proof}
\begin{cor}
\label{cor:Fourier transform of orbital integrals as pushforward measures}We
have 
\[
\widehat{\xi}_{\mathcal{O}_{\mathrm{st}}}=\pi_{*}^{\perp}(\mu_{\widetilde{\mathcal{O}}^{\perp}(F)}).
\]
\end{cor}

\begin{proof}
Indeed, by Lemma \ref{lem: push by proper projection commutes with Fourier},
\[
\pi_{*}^{\perp}(\mu_{\widetilde{\mathcal{O}}^{\perp}(F)})=\Pi_{*}\xi_{\widetilde{\mathcal{O}}^{\perp}(F)}=\Pi_{*}\mathcal{F}_{\mathcal{P},\g}(\xi_{\widetilde{\mathcal{O}}(F)})=\mathcal{F}(\Pi_{*}\xi_{\widetilde{\mathcal{O}}(F)})=\widehat{\xi}_{\mathcal{O}_{\mathrm{st}}}.\qedhere
\]
\end{proof}
Note that the map $\pi^{\perp}:\widetilde{\mathcal{O}}^{\perp}\rightarrow\ggm$
is dominant. By e.g.~\cite[Corollary 3.6]{AA16}, it follows that
$\pi_{*}^{\perp}(\mu_{\widetilde{\mathcal{O}}^{\perp}})$ has a density
which is given by an absolutely integrable function. This agrees with
Harish-Chandra's result (see the discussion in $\mathsection$\ref{subsec:Fourier-transform-of}). 
\begin{prop}
\label{prop:explicit description of the Fourier transform of obital integral }In
the notation of Theorem \ref{thm:Theorem A in section 4}, the distribution
$\widehat{\xi}_{\mathcal{O}_{\mathrm{st}}}$ is represented on the
set of regular semisimple elements $\g_{\rss}$ by the function
\begin{equation}
f_{\mathcal{O}_{\mathrm{st}}}(Y)=\frac{1}{\left|D_{\g}(Y)\right|_{F}^{\frac{1}{2}}}\sum_{i=1}^{M_{Y}}\left|D_{\mathfrak{\l}}(Y_{i})\right|_{F}^{1/2},\label{eq:concrete formula}
\end{equation}
where $Y_{1},\dots,Y_{M_{Y}}$ are representatives of $L$-conjugacy
classes of elements in $\l$, which are $G$-conjugate to $Y\in\g$.
\end{prop}

\begin{proof}
Let $\Phi:=\pi^{\perp}\circ\psi^{\perp}\circ(\exp,\mathrm{Id}):\n^{-}\times\p\simeq\g\rightarrow\g$,
be the map $\Phi(X,Z)=\mathrm{Ad}(\exp(X)).Z$, where $\exp:\n^{-}\rightarrow N^{-}$
is the exponential map. It follows from Corollary \ref{cor:Fourier transform of orbital integrals as pushforward measures}
that
\[
\Phi_{*}\mu_{\g}=\pi_{*}^{\perp}\circ\psi_{*}^{\perp}\left|\omega_{\NN^{-}}\wedge\omega_{\ppm}\right|_{F}=\pi_{*}^{\perp}(\mu_{\widetilde{\mathcal{O}}^{\perp}(F)})=\widehat{\xi}_{\mathcal{O}_{\mathrm{st}}}.
\]
Let us first characterize the fibers of $\pi^{\perp}$ over $\g_{\rss}$.
If $Y\notin G.\p$, then $(\pi^{\perp})^{-1}(Y)=\varnothing$. Otherwise,
$Y\in\g_{\rss}$ is conjugated to $\widetilde{Y}\in\p\cap\g_{\rss}$.
Since $\widetilde{Y}$ is contained in some Levi factor of $\p$,
and any two such Levis are conjugated by $N$ (\cite[Proposition 16.1.1(2)]{Spr98}),
$Y$ is conjugated to $\l\cap\g_{\rss}$. Since $\pi_{\mathcal{O}}$
is $G$-invariant, we may assume $Y\in\l\cap\g_{\rss}$. Since $Y$
is regular, $\t:=\mathrm{Cent}_{\g}(Y)$ is a Cartan contained in
$\l$, and $(\pi^{\perp})^{-1}(Y)=\left\{ (Q,Y)\in\mathcal{P}\times\g:\q\ni Y\right\} $.
Note that each such $\q$ must contain $\t$ as well, so $(\pi^{\perp})^{-1}(Y)$
is identified with the $W(G,T)$-orbit of $P$ in $\mathcal{P}$,
where $T$ is a maximal torus with $\Lie(T)=\t$. Since $W(L,T)$
is the stabilizer of this action, $(\pi^{\perp})^{-1}(Y)=\left\{ (wPw^{-1},Y)\right\} _{w\in W_{G/L}}$,
where $W_{G/L}$ is a set of representatives for $W(G,T)/W(L,T)$.
Moreover, $W_{G/L}.Y=\left\{ Y_{1},\dots,Y_{M_{Y}}\right\} $ is a
set of representatives of $L$-conjugacy classes of elements in $\l$,
which are $G$-conjugate to $Y\in\g$.

Note that the set
\[
S:=\left\{ Y\in\g_{\rss}:(\pi^{\perp})^{-1}(Y)\subseteq N^{-}P/P\times\{Y\}\right\} 
\]
is of full measure in $\g_{\rss}$ since its complement is a finite
union of subvarieties of codimension $\geq1$. Moreover, for each
$Y\in S$ we have $\Phi^{-1}(Y)=\{(X_{i},\widetilde{Y_{i}})\}_{i=1}^{M_{Y}}$,
where $\widetilde{Y}_{i}\in\p$ is $P$-conjugated to $Y_{i}$. 

Let $L_{\widetilde{X}}:\n^{-}\rightarrow\n^{-}$ be the map $L_{\widetilde{X}}(X)=\log\left(\exp(\widetilde{X})\exp(X)\right)$.
By the Baker\textendash Campbell\textendash Hausdorff formula, and
since $\ad_{\widetilde{X}}|_{\n^{-}}$ is nilpotent, it follows that
$\mathrm{Jac}(L_{\widetilde{X}})|_{0}=1$. Since $\Phi=\mathrm{Ad}(\exp(-\widetilde{X}))\circ\Phi\circ(L_{\widetilde{X}},\mathrm{Id})$,
we have by the chain rule, for every $(\widetilde{X},\widetilde{Z})\in\n^{-}\times\p$:
\[
\mathrm{Jac}(\Phi)|_{(0,\widetilde{Z})}=\mathrm{Jac}(\mathrm{Ad}(\exp(-\widetilde{X})))|_{\mathrm{Ad}(\exp(\widetilde{X}).\widetilde{Z}}\cdot\mathrm{Jac}(\Phi)|_{(\widetilde{X},\widetilde{Z})}\cdot\mathrm{Jac}((L_{\widetilde{X}},\mathrm{Id}))|_{(0,\widetilde{Z})}=\mathrm{Jac}(\Phi)|_{(\widetilde{X},\widetilde{Z})},
\]
By the Gelfand\textendash Leray formula for a density of a pushforward
measure (see e.g.~\cite[Corollary 3.6]{AA16}), for each $Y\in S$,
one has: 
\begin{equation}
f_{\mathcal{O}_{\mathrm{st}}}(Y)=\sum_{i=1}^{M_{Y}}\left|\mathrm{Jac}(\Phi)|_{(X_{i},\widetilde{Y}_{i})}\right|_{F}^{-1}=\sum_{i=1}^{M_{Y}}\left|\mathrm{Jac}(\Phi)|_{(0,\widetilde{Y}_{i})}\right|_{F}^{-1}.\label{eq:Gelfand-Leray}
\end{equation}
Note that
\[
\Phi(X,Z)=\Ad(\exp(X))Z=\exp(\ad(X))(Z)=Z+[X,Z]+\frac{1}{2}[X,[X,Z]]\dots
\]
For each $W,Z'\in\p$, $X'\in\n^{-}$, and a formal variable $t$,
we have:
\[
\Phi(tX',W+tZ')=W+tZ'+t[X',W]+\text{higher terms in }t,
\]
and therefore $D_{\Phi}|_{(0,W)}(X',Z')=Z'+[X',W].$ Hence
\begin{equation}
\left|\mathrm{Jac}(\Phi)|_{(0,\widetilde{Y}_{i})}\right|_{F}=\left|\det(\ad\widetilde{Y}_{i}|_{\n^{-}})\right|_{F}=\left|\det(\ad Y_{i}|_{\n^{-}})\right|_{F}=\left|\frac{D_{\g}(Y_{i})}{D_{\mathfrak{\l}}(Y_{i})}\right|_{F}^{\frac{1}{2}}=\left|\frac{D_{\g}(Y)}{D_{\mathfrak{\l}}(Y_{i})}\right|_{F}^{\frac{1}{2}}.\label{eq:formula for Jacobian}
\end{equation}
Combining (\ref{eq:Gelfand-Leray}) with (\ref{eq:formula for Jacobian}),
the proposition follows.
\end{proof}
\begin{proof}[Proof of Theorem \ref{thm:Theorem A in section 4}]
By (\ref{eq:Gelfand-Leray}) and (\ref{eq:formula for Jacobian}),
we have for each $(X,Z)\in\n^{-}\times\p$:
\begin{equation}
\left|\mathrm{Jac}(\Phi)|_{(X,Z)}\right|_{F}=\left|D_{\g,\l}(\overline{Z})\right|_{F}^{\frac{1}{2}},\label{eq:Jacobian is relative discriminant}
\end{equation}
where $\overline{Z}$ is the projection of $Z\in\p$ to $\l$. Since
$\widehat{\xi}_{\mathcal{O}_{\mathrm{st}}}=\pi_{*}^{\perp}(\mu_{\widetilde{\mathcal{O}}^{\perp}(F)})$
is a pushforward of a smooth measure by a proper analytic map $\pi^{\perp}:\widetilde{\mathcal{O}}^{\perp}(F)\to\g$,
by \cite[Theorem 1.1 or Proposition 4.1]{GHS} we have $\epsilon_{\star}(\widehat{\xi}_{\mathcal{O}_{\mathrm{st}}})=\lct_{F}(\mathrm{Jac}(\pi^{\perp}))$.
Recall that $\Phi=\pi^{\perp}\circ\psi^{\perp}\circ(\exp,\mathrm{Id})$.
Since $\psi^{\perp}\circ(\exp,\mathrm{Id}):\n^{-}\times\p\rightarrow\widetilde{\mathcal{O}}^{\perp}(F)$
is an analytic isomorphism to its image, by the $G$-action on $\widetilde{\mathcal{O}}^{\perp}(F)$,
and by (\ref{eq:Jacobian is relative discriminant}) we deduce,
\[
\epsilon_{\star}(\widehat{\xi}_{\mathcal{O}_{\mathrm{st}}})=\lct_{F}(\mathrm{Jac}(\pi^{\perp}))=\lct_{F}(\mathrm{Jac}(\Phi))=2\lct_{F}(D_{\g,\l}).\qedhere
\]
\end{proof}
\begin{rem}
\label{rem:A-field-dependent construction}An analytic construction
of $\widehat{\xi}_{\mathcal{O}_{\mathrm{st}}}$ goes as follows. Let
$\delta_{\{0\},\l}$ be the delta distribution on $\l$, i.e.~$\delta_{\{0\},\l}(f):=f(0)$,
for $f\in\mathcal{S}(\l)$. We identify $\xi_{\mathcal{O}_{\mathrm{st}}}$
with $\mathrm{Ind}_{P}^{G}(\delta_{\{0\},\l})\in\mathcal{S}^{*}(\g)$,
where $\mathrm{Ind}_{P}^{G}:\mathcal{S}^{*}(\l)\rightarrow\mathcal{S}^{*}(\g)$
denotes parabolic induction of distributions (see e.g.~\cite[Section 13.4]{Kot05}).
Since Fourier transform commutes with taking parabolic induction \cite[Lemma 13.4]{Kot05},
we have $\widehat{\xi}_{\mathcal{O}_{\mathrm{st}}}=\widehat{\mathrm{Ind}_{P}^{G}(\delta_{\{0\},\l})}=\mathrm{Ind}_{P}^{G}(\mu_{\l})$,
where $\mu_{\l}$ is the Haar measure on $\l$. Writing $\mathrm{Ind}_{P}^{G}(\mu_{\l})$
using the formula for $\mathrm{Ind}_{P}^{G}$, we deduce that $\widehat{\xi}_{\mathcal{O}_{\mathrm{st}}}$,
as a distribution on $\g$, is given by integration with respect to
the function $f_{\mathcal{O}_{\mathrm{st}}}$ as in (\ref{eq:concrete formula})
(see e.g.~\cite[Lemma 13.2]{Kot05}).

It is worth noting another description of $\widehat{\xi}_{\mathcal{O}_{\mathrm{st}}}$
in a small neighborhood of $0$.  In \cite{How74}, in the case $\GG=\GL_{n}$,
Howe identified $\widehat{\xi}_{\mathcal{O}_{\mathrm{st}}}\circ\log(g)$,
for $g$ in a sufficiently small neighborhood of $e\in\GL_{n}(F)$,
with the character of $\mathrm{Ind}_{\underline{P}(F)}^{\GL_{n}(F)}1$,
where $\underline{P}$ is a polarization of $\mathcal{O}$. This was
generalized by Moeglin and Waldspurger \cite{MW87}. 
\end{rem}

\begin{rem}
\label{rem:Rango Rao}If $\mathcal{O}$ in Theorem \ref{thm:Theorem A in section 4}
is an even orbit, then one can take the Richardson parabolic to be
the Jacobson\textendash Morozov parabolic associated to $\mathcal{O}$
(see \cite[p.505-506]{Rao72}). Then, the generalized Springer resolution
$\pi_{\mathcal{O}}:\GG\times_{\PP}\nnm\to\overline{\mathcal{O}}$
is in fact a symplectic resolution of singularities (see e.g. \cite[Proposition 1.5]{Fu03}). 

On the other hand, if one takes a non-even orbit $\mathcal{O}$, and
uses its Jacobson\textendash Morozov parabolic, then we can still
repeat the construction above, but two significant problems arise: 
\begin{itemize}
\item We still have $\mu_{\mathcal{O}_{\mathrm{st}}}:=(\pi_{\mathcal{O}})_{*}(\mu_{\widetilde{\mathcal{O}}(F)})$,
but now the measure $\mu_{\widetilde{\mathcal{O}}(F)}$ has a more
complicated description than in the Richardson case. This is computed
in \cite{Rao72}, and the extra complication involves the function
$\varphi(X)$ in the notation of \cite[Theorem 1]{Rao72}. In particular,
the partial Fourier transform $\mathcal{F}_{\mathcal{P},\g}(i_{*}\mu_{\widetilde{\mathcal{O}}(F)})$
is not necessarily a positive measure anymore. 
\item The map $\pi^{\perp}:\widetilde{\mathcal{O}}^{\perp}\to\ggm$ is not
a map between varieties of equal dimensions ($\dim\widetilde{\mathcal{O}}^{\perp}>\dim\ggm$),
and it is significantly harder to determine the integrability of pushforward
measures under such maps. Hence, even if $\mathcal{F}_{\mathcal{P},\g}(i_{*}\mu_{\widetilde{\mathcal{O}}(F)})$
was a positive measure, it would be hard to determine the integrability
of $\widehat{\xi}_{\mathcal{O}_{\mathrm{st}}}=\pi_{*}^{\perp}(\mathcal{F}_{\mathcal{P},\g}(i_{*}\mu_{\widetilde{\mathcal{O}}(F)})$). 
\end{itemize}
\end{rem}

\section{\label{sec:Integrability-of-global characters}Integrability of $\widehat{\xi}_{\mathcal{O}}$
and of Harish-Chandra characters in $\protect\GL_{n}$}

In this section we prove Theorems \ref{thmB:lower bound on epsilon of reps},
\ref{thmC: epsilon of Fourier of orbital integral for GL_n} and \ref{thm D:epsilon of a representation}.

\subsection{\label{subsec:An-explicit-formula for epsilon(O)}A formula and a
lower bound for $\epsilon_{\star}(\widehat{\xi}_{\mathcal{O}})$ }

We would like to use Proposition \ref{prop:formula for log canonical threshold},
in order to compute $\lct_{F}(D_{\g,\l};0)$. We use the following
versions of the Weyl integration formula for the Lie algebra and the
group:
\begin{prop}[{see e.g.~\cite[Eq. (7.7.1)]{Kot05}, \cite[Lemma 4]{HC70} and \cite[Proposition 5.27]{Kna01}}]
\label{prop:Weyl integration formula}Let $G=\GG(F)$ for a connected
reductive group $\GG$, with $\g=\Lie(G)$. For each maximal $F$-torus
$T$, let $W(G,T):=N_{G}(T)/T$ be the corresponding Weyl group, and
let $d\overline{g}$ be a $G$-invariant measure on $G/T$. Then for
a suitable normalization of the Haar measures $dX$ on $\t$, $dY$
on $\g$, $dt$ on $T$ and of the measure $d\overline{g}$, the following
hold for each $f\in\mathcal{S}(\g)$ and $h\in\mathcal{S}(G)$:
\begin{align*}
\int_{\g}f(Y)dY & =\sum_{T}\frac{1}{\#W(G,T)}\int_{\t}\left|D_{\g}(X)\right|_{F}\left(\int_{G/T}f(\Ad(g).X)d\overline{g}\right)dX,\\
\int_{G}h(g)\mu_{G} & =\sum_{T}\frac{1}{\#W(G,T)}\int_{T}\left|D_{G}(t)\right|_{F}\left(\int_{G/T}h(gtg^{-1})d\overline{g}\right)dt,
\end{align*}
where $T$ runs over all representatives of $G$-conjugacy classes
of maximal $F$-tori in $G$.
\end{prop}

\begin{rem}
\label{rem:Chevalley model}Any connected $F$-split reductive group
$\GG$ is a base change of a $\Q$-split reductive group $\GG_{\mathrm{Ch}}$
(its \emph{Chevalley model}, see \cite[XXV.1.1]{ABDGGRS65}). If $\ggm_{\mathrm{Ch}}:=\Lie(\GG_{\mathrm{Ch}})$,
and $\Q\subseteq K$ is a field, we denote by $\underline{\g}_{K}$
the base change to $K$ of $\ggm_{\mathrm{Ch}}$, and write $\g_{K}:=\underline{\g}_{K}(K)$.
In particular, $\g_{F}\simeq\g$. 
\end{rem}

\begin{lem}
\label{lem:hyperplane arrangments in the F-split case}Let $G=\GG(F)$,
for a connected $F$-split reductive group $\GG$, with $\g=\Lie(G)$.
Let $T$ be an $F$-split torus, with $\t:=\Lie(T)$. Let $\p$ be
a parabolic subalgebra with Levi $\l$. Then for every $m\in\N$,
\begin{equation}
\lct_{F}\left(\triangle_{\g,\l}^{\t};(\triangle_{\l}^{\t})^{m},0\right)=\lct_{\C}\left(\triangle_{\g_{\C},\l_{\C}}^{\t_{\C}};(\triangle_{\l_{\C}}^{\t_{\C}})^{m},0\right).\label{eq:Lemma 5.3}
\end{equation}
\end{lem}

\begin{proof}
We would like to estimate the integral $\int_{B}\frac{\left|\triangle_{\l}^{\t}(X)\right|_{F}^{m+s}}{\left|\triangle_{\g}^{\t}(X)\right|_{F}^{s}}dX$,
for a small ball $B$ around $0$ in $\t$. We may assume $\GG=(\GG_{\mathrm{Ch}})_{F}$
and that $\underline{T}$ and $\LL$ are the base change to $F$ of
$\Q$-split torus and Levi in $\GG_{\mathrm{Ch}}$. In particular,
we have $\Sigma(\g_{\C},\t_{\C})\simeq\Sigma(\g,\t)$, and all roots
are defined over $\Q$. Moreover, $\triangle_{\l}^{\t}$ and $\triangle_{\g}^{\t}$
are the defining polynomials of two hyperplane arrangements with coefficients
in $\Q$, and the dimension of each intersection of hyperplanes $W=\bigcap_{\alpha\in R}W_{\alpha}$
in $\t(\C)$ for $R\subseteq\Sigma^{+}(\g,\t)$ is the same as in
$\t(F)$. This gives a bijection $L_{\mathrm{dense}}(\mathcal{A}_{\g_{\C},\t_{\C}})\simeq L_{\mathrm{dense}}(\mathcal{A}_{\g,\t})$. 

The map $\Lambda$ (resp.~$\Lambda_{\C}$) from Theorem \ref{Thm:resolution of hyperplane arrangements}
is a log-principalization for $\triangle_{\l}^{\t},\triangle_{\g}^{\t}$
(resp.~$\triangle_{\l_{\C}}^{\t_{\C}},\triangle_{\g_{\C}}^{\t_{\C}}$).
Repeating the proof of Proposition \ref{prop:formula for log canonical threshold},
the number of divisors $\left\{ E_{W}\right\} _{W\in L_{\mathrm{dense}}(\mathcal{A}_{\g,\t})}$,
and the numerical data $\left\{ (a_{W},b_{W},k_{W})\right\} _{W\in L_{\mathrm{dense}}(\mathcal{A}_{\g,\t})}$
are the same for $\Lambda$ and $\Lambda_{\C}$. We therefore get
the same formula as in Proposition \ref{prop:formula for log canonical threshold},
which implies the lemma. 
\end{proof}
\begin{prop}
\label{prop:relating lct to relative lct in split groups}In the setting
of of Lemma \ref{lem:hyperplane arrangments in the F-split case},
we have:
\begin{align}
2\lct_{F}(D_{\g,\l};0) & =\lct_{\C}\left(\triangle_{\g_{\C},\l_{\C}}^{\t_{\C}};\triangle_{\l_{\C}}^{\t_{\C}},0\right),\text{ and, }\label{eq:Prop 5.4}\\
2\lct_{F}(D_{\g};0) & =\lct_{\C}\left(\triangle_{\g_{\C}}^{\t_{\C}};\triangle_{\g_{\C}}^{\t_{\C}},0\right)=1+\lct_{\C}(\triangle_{\g_{\C}}^{\t_{\C}};0)\label{eq:Prop 5.4 lower}
\end{align}
\end{prop}

\begin{proof}
We prove (\ref{eq:Prop 5.4}). The proof (\ref{eq:Prop 5.4 lower})
is almost identical. By applying the Weyl integration formula (Proposition
\ref{prop:Weyl integration formula}) for $G=L$ and $f=1_{B}\cdot\left|D_{\g,\l}\right|_{F}^{-s/2}$,
for a small ball $B$ around $0$ in $\l$, we get: 
\[
\int_{B}\left|D_{\g,\l}(Y)\right|_{F}^{-s/2}dY=\sum_{T\subseteq L}\frac{1}{\#W(L,T)}\int_{\t}\left|D_{\g,\l}(X)\right|_{F}^{-s/2}\left|D_{\l}(X)\right|_{F}\left(\int_{L/T}1_{B}(\Ad(l).X)d\overline{l}\right)dX.
\]
Since each of the summands is positive, we deduce that:
\begin{equation}
2\lct_{F}(D_{\g,\l};0)=\underset{T\subseteq L}{\min}\sup\left\{ s>0:\exists B\subseteq\l\text{ s.t.}\int_{\t}\frac{\left|D_{\l}\right|_{F}^{\left(1+s\right)/2}}{\left|D_{\g}\right|_{F}^{s/2}}\left(\left|D_{\l}\right|_{F}^{1/2}\int_{L/T}1_{B}(\Ad(l).X)d\overline{l}\right)dX<\infty\right\} ,\label{eq:auxilary}
\end{equation}
where $T\subseteq L$ runs over all $L$-conjugacy classes of maximal
$F$-tori in $L$. Now let $T_{0}\subseteq L$ be a maximal $F$-split
torus. By \cite[Lemma 13.3]{Kot05}, we have 
\begin{equation}
\int_{L/T_{0}}1_{B}(\Ad(l).X)d\overline{l}=\left|D_{\l}\right|_{F}^{-1/2}\widetilde{\psi_{B}}(X),\label{eq:descend for split torus}
\end{equation}
where $\widetilde{\psi_{B}}(X)>0$ in a neighborhood of $0$ in $\t_{0}$.
Plugging (\ref{eq:descend for split torus}) into (\ref{eq:auxilary}),
by (\ref{eq:reduction to hyperplane arrangement}) and by Lemma \ref{lem:hyperplane arrangments in the F-split case},
we deduce that:
\begin{equation}
2\lct_{F}(D_{\g,\l};0)\leq\lct_{F}(\triangle_{\g,\l}^{\t_{0}};\triangle_{\l}^{\t_{0}},0)=\lct_{\C}(\triangle_{\g_{\C},\l_{\C}}^{\t_{0,\C}};\triangle_{\l_{\C}}^{\t_{0,\C}},0).\label{eq:5.3}
\end{equation}
Applying the Weyl integration formula to $\LL(\C)$, and again by
(\ref{eq:descend for split torus}), we get:
\begin{equation}
\lct_{\C}(\triangle_{\g_{\C},\l_{\C}}^{\t_{0,\C}};\triangle_{\l_{\C}}^{\t_{0,\C}},0)=2\lct_{\C}(D_{\g,\l};0).\label{eq:5.4}
\end{equation}
By (\ref{eq:5.3}), (\ref{eq:5.4}) and since $\lct_{F}\geq\lct_{\C}$
(Remark \ref{rem:Lct of polynomials on smooth varieties and behavior with field extensions}),
we get the opposite inequality to (\ref{eq:5.3}), which implies the
proposition. 
\end{proof}
\begin{proof}[Proof of Theorem \ref{thmB:lower bound on epsilon of reps}]
We first prove Item (1). Let $\xi$ be a distribution as in Item
(1). Then by the discussion before Theorem \ref{thmB:lower bound on epsilon of reps},
$\xi$ is given by a locally integrable function $f_{\xi}\in L_{\mathrm{Loc}}^{1}(\g)$
and $\left|D_{\g}\right|_{F}^{\frac{1}{2}}\cdot f_{\xi}$ is locally
bounded. Let $F'$ be a field extension of $F$ so that $\underline{G}_{F'}$
is $F'$-split. Since $\lct_{F}\geq\lct_{F'}$ (see Remark \ref{rem:Lct of polynomials on smooth varieties and behavior with field extensions}),
by Proposition \ref{prop:relating lct to relative lct in split groups}
and by Corollary \ref{cor:lct of Weil discriminant} , we have 
\[
2\lct_{F}(D_{\g};0)\geq2\lct_{F'}(D_{\g_{F'}};0)=1+\lct_{\C}(\triangle_{\g_{\C}}^{\t_{\C}},0)=1+\underset{1\leq i\leq M}{\min}\frac{2}{h_{\mathfrak{g_{i}}}}.
\]
For every compact subset $B\subseteq\g$, and every $s<\underset{1\leq i\leq M}{\min}\frac{2}{h_{\mathfrak{g_{i}}}}$,
one has:
\begin{equation}
\int_{B}\left|f_{\xi}(X)\right|^{1+s}dX=\int_{B}\left(\left|D_{\g}(X)\right|_{F}^{\frac{1}{2}}\cdot\left|f_{\xi}(X)\right|\right)^{1+s}\left|D_{\g}(X)\right|_{F}^{-\frac{1+s}{2}}dX\leq M^{1+s}\int_{B}\left|D_{\g}(X)\right|_{F}^{-\frac{1+s}{2}}dX<\infty,\label{eq:proof of Theorem B}
\end{equation}
where $M$ is an upper bound on the $L^{\infty}$-norm of $\left|D_{\g}\right|_{F}^{\frac{1}{2}}\cdot f_{\xi}|_{B}$.
This finishes the proof of Item (1). 

We now turn to the proof of Item (2). Let $\pi\in\mathrm{Rep}(G)$.
Let $\pi_{1},\dots,\pi_{M'}$ be the irreducible subquotients of $\pi$.
Clearly $\epsilon_{\star}(\pi)\geq\underset{1\leq i\leq M'}{\min}\epsilon_{\star}(\pi_{i})$,
so we may assume $\pi\in\Irr(G)$. By \cite[Theorem 16.3]{HC99} and
\cite[Theorem 10.35]{Kna01}, the function $\left|D_{G}\right|_{F}^{\frac{1}{2}}\left|\theta_{\pi}\right|$
is locally bounded for every $\pi\in\Irr(G)$. This means that for
every ball $B$ in $G$, and every $s>0$, there exists $C>0$ such
that: 
\begin{equation}
\int_{B}\left|\theta_{\pi}\right|^{1+s}\mu_{G}\leq C\int_{B}\left|D_{G}\right|_{F}^{-\frac{1+s}{2}}\mu_{G}.\label{eq:reduction to integrability of Weyl discriminant}
\end{equation}
Note that $\mu_{G}$ is given by $\left|\omega_{\underline{G}}\right|_{F}$
for an invertible top-form $\omega_{\underline{G}}$ on $\underline{G}$.
Let $\underline{H}$ be the Chevalley group such that $\underline{H}_{\overline{F}}\simeq\underline{G}_{\overline{F}}$
(as in Remark \ref{rem:Chevalley model}) so that $\underline{H}_{F'}\simeq\underline{G}_{F'}$
for some finite extension $F'\supseteq F$. By Remark \ref{rem:Lct of polynomials on smooth varieties and behavior with field extensions},
we have $\lct_{F}(D_{G})\geq\lct_{F'}(D_{G})=\lct_{F'}(D_{\underline{H}(F')})\geq\lct_{\C}(D_{\underline{H}(\C)})$.
It is enough to show:
\[
\int_{B'}\left|D_{\underline{H}(\C)}\right|_{\C}^{-\frac{1+s}{2}}\mu_{\underline{H}(\C)}<\infty\text{ \,\,\,\,for each }s<\underset{1\leq i\leq M}{\min}\frac{2}{h_{\mathfrak{g_{i}}}}\text{ and each ball }B'\subseteq\underline{H}(\C).\tag{\ensuremath{\star}}
\]
By Proposition \ref{prop:Weyl integration formula} and by the Archimedean
version of (\ref{eq:descend for split torus}) (\cite[Lemma 10.17]{Kna01}),
we deduce,
\[
(\star)\text{ holds}\,\Longleftrightarrow\int_{B''}\left|D_{\underline{H}(\C)}\right|_{\C}^{-\frac{s}{2}}\mu_{\underline{T}(\C)}<\infty\text{ \,\,\,\,for each }s<\underset{1\leq i\leq M}{\min}\frac{2}{h_{\mathfrak{g_{i}}}}\text{ and each ball }B''\subseteq\underline{T}(\C),
\]
where $\underline{T}$ is a maximal torus in $\underline{H}$. Changing
coordinates via the exponential map $\exp:\t(\C)\rightarrow\underline{T}(\C)$,
which is a local diffeomorphism, reduces our integral to $\int_{B'''}\left|D_{\underline{H}(\C)}\circ\exp\right|_{\C}^{-\frac{s}{2}}\mu_{\t(\C)}$.
We get,
\begin{align}
 & (-1)^{\dim\underline{T}}\cdot D_{\underline{H}(\C)}(\exp(X))=\det(1-\Ad(\exp(X))|_{\mathfrak{h}_{\C}/\t_{\C}})=\det(1-\exp(\ad(X))|_{\mathfrak{h}_{\C}/\t_{\C}})\nonumber \\
= & \prod_{\alpha\in\Sigma(\mathfrak{h}_{\C},\t_{\C})}\left(1-\exp(\alpha(X))\right)=\left(\prod_{\alpha\in\Sigma(\mathfrak{h}_{\C},\t_{\C})}\alpha(X)\right)\prod_{\alpha\in\Sigma(\mathfrak{h}_{\C},\t_{\C})}\frac{1-\exp(\alpha(X))}{\alpha(X)}.\label{eq:comparing group Weyl discriminant to Lie algebra discriminant}
\end{align}
Note that $\left|\frac{1-\exp(\alpha(X))}{\alpha(X)}\right|_{\C}$
is invertible in a small neighborhood of $0\in\t(\C)$. Hence, 
\[
2\lct_{\C}(D_{\underline{H}(\C)}|_{\underline{T}(\C)};e)=\lct_{\C}(\triangle_{\mathfrak{h}_{\C}}^{\t_{\C}};0)=\lct_{\C}(\triangle_{\mathfrak{h}_{\C}}^{\t_{\C}}),
\]
Now for every $t'\in\underline{T}(\C)$, let $S_{\C}=\mathrm{Cent}_{\underline{H}(\C)}(t')$
and let $\mathfrak{s}_{\C}:=\Lie(S_{\C}).$ Note that $\mathfrak{s}_{\C}\supseteq\t_{\C}$
is a reductive subalgebra of $\mathfrak{h}_{\C}$, and that for each
$\alpha\in\Sigma(\mathfrak{h}_{\C},\t_{\C})$, if $\xi_{\alpha}:\underline{T}(\C)\rightarrow\C^{\times}$
is the corresponding character of $\underline{T}(\C)$, then $\xi_{\alpha}(t')=1$
if and only if $\alpha\in\Sigma(\mathfrak{s}_{\C},\t_{\C})$. Now,
by (\ref{eq:comparing group Weyl discriminant to Lie algebra discriminant}),
for $t$ in a small enough ball around $e\in\underline{T}(\C)$, the
function $\left|\frac{D_{\underline{H}(\C)}(tt')}{D_{S_{\C}}(tt')}\right|_{\C}$
is invertible. Moreover, we have
\[
\left|D_{S_{\C}}(tt')\right|_{\C}=\left|\prod_{\alpha\in\Sigma(\mathfrak{s}_{\C},\t_{\C})}\left(1-\xi_{\alpha}(tt')\right)\right|_{\C}=\left|\prod_{\alpha\in\Sigma(\mathfrak{s}_{\C},\t_{\C})}\left(1-\xi_{\alpha}(t)\right)\right|_{\C}=\left|D_{S_{\C}}(t)\right|_{\C},
\]
and therefore we get
\[
2\lct_{\C}(D_{\underline{H}(\C)}|_{\underline{T}(\C)};t')=2\lct_{\C}(D_{S_{\C}}|_{\underline{T}(\C)};t')=2\lct_{\C}(D_{S_{\C}}|_{\underline{T}(\C)};e)=\lct_{\C}(\triangle_{\mathfrak{s}_{\C}}^{\t_{\C}}).
\]
Since $\triangle_{\mathfrak{s}_{\C}}^{\t_{\C}}$ divides $\triangle_{\mathfrak{h}_{\C}}^{\t_{\C}}$,
we have $\lct_{\C}(\triangle_{\mathfrak{s}_{\C}}^{\t_{\C}})\geq\lct_{\C}(\triangle_{\mathfrak{h}_{\C}}^{\t_{\C}})$,
and we deduce:
\[
2\lct_{\C}(D_{\underline{H}(\C)}|_{\underline{T}(\C)})=\lct_{\C}(\triangle_{\mathfrak{h}_{\C}}^{\t_{\C}}).
\]
Since $(\g_{1})_{\C},\dots,(\g_{M})_{\C}$ are the simple constituents
of $[\mathfrak{h}_{\C},\mathfrak{h}_{\C}]$, Corollary \ref{cor:lct of Weil discriminant}
implies that $(\star)$ holds, as required. 
\end{proof}
We can now conclude Theorem \ref{thmC: epsilon of Fourier of orbital integral for GL_n}. 
\begin{proof}[Proof of Theorem \ref{thmC: epsilon of Fourier of orbital integral for GL_n}]
Recall that in $\gl_{n}(F)$ all nilpotent orbits are stable and
Richardson. Hence, by Theorem \ref{thm A: epsilon of Fourier of orbital integral is lct(D(G/M))},
for each $\mathcal{O}\in\mathcal{O}(\mathcal{N}_{\gl_{n}(F)})$, with
polarizing parabolic $\p=\l\oplus\n$, we have $\epsilon_{\star}(\widehat{\xi}_{\mathcal{O}})=2\lct_{F}(D_{\g,\l};0)$.
By Proposition \ref{prop:relating lct to relative lct in split groups},
we have $\epsilon_{\star}(\widehat{\xi}_{\mathcal{O}})=\lct_{\C}(\triangle_{\g_{\C},\l_{\C}}^{\t_{\C}};\triangle_{\l_{\C}}^{\t_{\C}},0)$.
The theorem now follows from Proposition \ref{Prop:lct of relative Weyl discriminant},
Theorem \ref{thm:orbit closure and dominance} and Corollary \ref{cor:dominant Levi means large rlct}.
\end{proof}

\subsection{\label{subsec:A-formula-for epsilon(pi)}A formula for $\epsilon_{\star}(\pi;e)$ }

We now turn to describe $\epsilon_{\star}(\pi;e)$ for $\pi\in\Irr(\GL_{n}(F))$.
\begin{lem}
\label{lem:interaction between different orbits}Let $\mathcal{O}_{1},\dots,\mathcal{O}_{N}\in\mathcal{O}(\mathcal{N}_{\gl_{n}(F)})$
be nilpotent orbits. Suppose that $\epsilon_{\star}(\widehat{\xi}_{\mathcal{O}_{1}})\leq\epsilon_{\star}(\widehat{\xi}_{\mathcal{O}_{i}})$
for all $i$. Then there exists $\delta>0$ such that for every $a\in\C^{N-1}$,
with $\left|a\right|<\delta$, one has:
\[
\epsilon_{\star}(\widehat{\xi}_{\mathcal{O}_{1}}+\sum_{i=2}^{N}a_{i}\widehat{\xi}_{\mathcal{O}_{i}})=\epsilon_{\star}(\widehat{\xi}_{\mathcal{O}_{1}}).
\]
\end{lem}

\begin{proof}
Each $\epsilon_{\star}(\widehat{\xi}_{\mathcal{O}_{i}})$ is represented
by a function $f_{\mathcal{O}_{i}}$. We obviously have the inequality
\[
\epsilon_{\star}(\widehat{\xi}_{\mathcal{O}_{1}}+\sum_{i=2}^{N}a_{i}\widehat{\xi}_{\mathcal{O}_{i}})\geq\underset{i}{\min}\epsilon_{\star}(\widehat{\xi}_{\mathcal{O}_{i}})=\epsilon_{\star}(\widehat{\xi}_{\mathcal{O}_{1}}).
\]
By Propositions \ref{prop:Weyl integration formula} and \ref{prop:explicit description of the Fourier transform of obital integral }:
\begin{align}
\epsilon_{\star}(\widehat{\xi}_{\mathcal{O}_{1}}) & =\underset{T\subseteq G}{\min}\sup\left\{ s>0:\exists B\subseteq\g\text{ s.t.}\int_{\t}\left|f_{\mathcal{O}_{1}}(X)\right|^{1+s}\left|D_{\g}(X)\right|_{F}\left(\int_{G/T}1_{B}(\Ad(g).X)d\overline{g}\right)dX<\infty\right\} \nonumber \\
 & =\underset{T\subseteq G}{\min}\sup\left\{ s>0:\exists B\subseteq\g\text{ s.t.}\int_{\t}\frac{\left|D_{\l}(X)\right|_{F}^{(1+s)/2}}{\left|D_{\g}(X)\right|_{F}^{s/2}}\left(\left|D_{\g}(X)\right|_{F}^{\frac{1}{2}}\int_{G/T}1_{B}(\Ad(g).X)d\overline{g}\right)dX<\infty\right\} .\label{eq:formula for leading orbit}
\end{align}
By (\ref{eq:descend for split torus}), we see that the minimum in
(\ref{eq:formula for leading orbit}) is obtained for $T=T_{0}$ a
maximal $F$-split torus. Therefore, it is enough to show that:
\begin{align}
 & \sup\left\{ s>0:\exists B\subseteq\t_{0}\text{ s.t.}\int_{B}\left|f_{\mathcal{O}_{1}}(X)+\sum_{i=2}^{N}a_{i}f_{\mathcal{O}_{i}}(X)\right|^{1+s}\left|\triangle_{\g}^{\t_{0}}(X)\right|_{F}dX<\infty\right\} \label{eq:linear combination of orbital integrals}\\
= & \sup\left\{ s>0:\exists B\subseteq\t_{0}\text{ s.t.}\int_{B}\left|f_{\mathcal{O}_{1}}(X)\right|^{1+s}\left|\triangle_{\g}^{\t_{0}}(X)\right|_{F}dX<\infty\right\} .\nonumber 
\end{align}
Note that by Proposition \ref{prop:explicit description of the Fourier transform of obital integral },
since the building blocks of each $f_{\mathcal{O}_{i}}(X)$ are the
Weyl discriminants $\triangle_{\l}^{\t_{0}}$ of suitable Levis $\l$
in $\g$, to analyze the integrals in (\ref{eq:linear combination of orbital integrals}),
we may apply the log-principalization $\Lambda:Y\rightarrow\t_{0}$
of Theorem \ref{Thm:resolution of hyperplane arrangements}. Then,
as in the proof of Proposition \ref{prop:formula for log canonical threshold},
the integrability of both expressions in (\ref{eq:linear combination of orbital integrals})
is determined by the order of vanishing of $f_{\mathcal{O}_{1}}\circ\Lambda,\dots,f_{\mathcal{O}_{N}}\circ\Lambda,\triangle_{\g}^{\t_{0}}\circ\Lambda$
as well as of $\operatorname{Jac}(\Lambda)$ along the divisors $\{E_{W}\}_{W\in L_{\mathrm{dense}}(\mathcal{A}_{\g,\t_{0}})}$
of $\Lambda$. Moreover, there is a single divisor $E_{W_{0}}$ which
controls the integrability of $f_{\mathcal{O}_{1}}\circ\Lambda$;
if $c_{W_{0}}$ (resp.~$d_{W_{0}},k_{W_{0}}$) is the order of vanishing
of $\frac{1}{f_{\mathcal{O}_{1}}\circ\Lambda}$ (resp.~$\triangle_{\g}^{\t_{0}}\circ\Lambda,\operatorname{Jac}(\Lambda)$)
along $E_{W_{0}}$, then
\begin{align}
 & \sup\left\{ s>0:\exists B\subseteq\t_{0}\text{ s.t.}\int_{B}\left|f_{\mathcal{O}_{1}}(X)\right|^{1+s}\left|\triangle_{\g}^{\t_{0}}(X)\right|_{F}dX<\infty\right\} \nonumber \\
= & \sup\left\{ s>0:d_{W_{0}}+k_{W_{0}}-c_{W_{0}}(1+s)>-1\right\} =\frac{d_{W_{0}}+k_{W_{0}}+1}{c_{W_{0}}}-1.\label{eq:epsilon is determined by a single divisor}
\end{align}
Now let $c_{2,W_{0}},\dots,c_{N,W_{0}}$ be the order of vanishing
of $\frac{1}{f_{\mathcal{O}_{i}}\circ\Lambda}$ along $E_{W_{0}}$.
By our assumption that $\epsilon_{\star}(\widehat{\xi}_{\mathcal{O}_{1}})\leq\epsilon_{\star}(\widehat{\xi}_{\mathcal{O}_{i}})$
we must have 
\[
\frac{d_{W_{0}}+k_{W_{0}}+1}{c_{i,W_{0}}}-1\geq\frac{d_{W_{0}}+k_{W_{0}}+1}{c_{W_{0}}}-1,\text{ for all }2\leq i\leq N,
\]
or equivalently, $c_{i,W_{0}}\leq c_{W_{0}}$. But this means that
when restricting to any small neighborhood in $Y$ which intersects
$E_{W_{0}}$, and taking $a\in\C^{N-1}$, with $\left|a\right|\ll1$,
the order of the pole of $f_{\mathcal{O}_{1}}\circ\Lambda$ along
$E_{W_{0}}$ is the same as of $f_{\mathcal{O}_{1}}\circ\Lambda+\sum_{i=2}^{N}a_{i}f_{\mathcal{O}_{i}}\circ\Lambda$.
This verifies Equality (\ref{eq:linear combination of orbital integrals}),
as required.
\end{proof}
\begin{prop}[{\cite{BB82,MW87}, see also \cite[p.227]{GGS17}}]
\label{prop:local character for GLn}Let $F$ be a local field of
characteristic $0$, and let $\pi\in\Irr(\GL_{n}(F))$. Then, in the
setting of Theorem \ref{thm:local character expansion}, all orbits
$\mathcal{O}'\in\mathcal{O}(\mathcal{N}_{\gl_{n}(F)})$ with $c_{\mathcal{O}'}(\pi)\neq0$
are in the closure a single nilpotent orbit $\mathcal{O}_{\max}$.
Similarly, $a_{\mathcal{O}'}(\pi)=0$ for all $\mathcal{O}'\neq\mathcal{O}_{\max}$. 
\end{prop}

\begin{proof}[Proof of Theorem \ref{thm D:epsilon of a representation}]
Recall that $F$ is non-Archimedean. For each $\mathcal{O}\in\mathcal{O}(\mathcal{N}_{\gl_{n}(F)})$,
let $f_{\mathcal{O}}\in L_{\mathrm{loc}}^{1}(\gl_{n}(F))$ be a function
representing $\widehat{\xi}_{\mathcal{O}}$. Then by Proposition \ref{prop:local character for GLn}
and Theorem \ref{thm:local character expansion}, for a small enough
neighborhood $0\in B\subseteq\g$, we have 
\[
\theta_{\pi}(\exp(X))=c_{\mathcal{O}_{\max}}(\pi)f_{\mathcal{O}_{\max}}(X)+\sum_{\mathcal{O}\neq\mathcal{O}_{\max}}c_{\mathcal{O}}(\pi)\cdot f_{\mathcal{O}}(X).
\]
Note that for every $c\in F$ with $0<\left|c\right|_{F}<1$, the
function $\theta_{\pi}(\exp(X))$ is in $L^{1+\epsilon}(c^{2}B)$
if and only if $\theta_{\pi}(\exp(c^{2}X))$ is in $L^{1+\epsilon}(B)$.
On the other hand, by the homogeneity of $\widehat{\xi}_{\mathcal{O}}$
(\cite[Lemma 17.2]{Kot05}),
\[
\left|c\right|_{F}^{\dim\mathcal{O}_{\max}}\theta_{\pi}(\exp(c^{2}X))=c_{\mathcal{O}_{\max}}(\pi)f_{\mathcal{O}_{\max}}(X)+\sum_{\mathcal{O}\neq\mathcal{O}_{\max}}\left|c\right|_{F}^{\dim\mathcal{O}_{\max}-\dim\mathcal{O}}c_{\mathcal{O}}(\pi)\cdot f_{\mathcal{O}}(X).
\]
Finally, taking $c$ with small enough $\left|c\right|_{F}$, and
using Lemma \ref{lem:interaction between different orbits}, the proposition
follows.
\end{proof}

\section{Upper bound on multiplicities of $K$-types in admissible representations}

In this section we prove Theorems \ref{thmE:upper bounds on multiplicities of K-types}
and \ref{thm F:bound on multilicites in compact homogeneous} and
Corollaries \ref{cor:upper bound for multiplicities in principal series}
and \ref{cor:Fourier coefficients of power word measure}.

We first recall some basic facts on the non-commutative Fourier transform
of absolutely integrable functions on compact groups. 

\subsection{\label{subsec:Hausdorff=002013Young-theorem-for}Hausdorff\textendash Young
theorem for compact groups}

In this subsection we follow \cite[Section 2.3]{App14}. If $H$ is
a compact (second countable) group, set $\widehat{H}:=\Irr(H)$ and
define $\mathcal{M}(\widehat{H}):=\bigcup_{\tau\in\widehat{H}}\mathrm{End}_{\C}(V_{\tau})$,
where $V_{\tau}$ is the representation space of $\tau$. We say that
a map $T:\widehat{H}\rightarrow\mathcal{M}(\widehat{H})$ is \textit{compatible}
if $T(\tau)\in\mathrm{End}_{\C}(V_{\tau})$ for any $\tau\in\widehat{H}$.
We denote the space of compatible mappings by $\mathcal{L}(\widehat{H})$.
Based on the Peter\textendash Weyl Theorem, one can define the \textit{non-commutative
Fourier transform} $\mathcal{F}:L^{1}(H)\rightarrow\mathcal{L}(\hat{H})$,
by
\begin{equation}
\mathcal{F}(f)(\tau)=\int_{H}f(h)\cdot\tau(h^{-1})\mu_{H},\label{eq:Fourier coefficients}
\end{equation}
where $\mu_{H}$ is the normalized Haar measure. For $1\leq p<\infty$,
we define $\mathcal{H}_{p}(\hat{H})$ as the linear span of all $T\in\mathcal{L}(\hat{G})$
for which
\begin{equation}
\left\Vert T\right\Vert _{p}:=\left(\sum_{\pi\in\hat{H}}\dim(\tau)\cdot\left\Vert T(\tau)\right\Vert _{\mathrm{Sch,p}}^{p}\right)^{\frac{1}{p}}<\infty,\label{eq:Lp on Fourier}
\end{equation}
where $\left\Vert T(\tau)\right\Vert _{\mathrm{Sch,p}}:=\left(\tr(\left(T(\tau)T(\tau)^{*}\right)^{p/2})\right)^{\frac{1}{p}}$
is the Schatten $p$-norm. 

The classical Hausdorff\textendash Young theorem in Fourier analysis
says that if $1\leq p\leq2$, and $f\in L^{p}(F^{n})$, then $\mathcal{F}(f)\in L^{q}(F^{n})$
for $\frac{1}{p}+\frac{1}{q}=1$. This theorem has the following group-theoretic
analogue:
\begin{lem}[{\cite[Theorem 2.32]{App14} and \cite[Section 2.14]{Edw72}}]
\label{Lemma: Hausdorff--Young}Let $H$ be a compact group, let
$1\leq p\leq2$ and let $f\in L^{p}(H)$. Then $\mathcal{F}(f)\in\mathcal{H}_{q}(\hat{H})$,
where $\frac{1}{p}+\frac{1}{q}=1$ and $\left\Vert \mathcal{F}(f)\right\Vert _{q}\leq\left\Vert f\right\Vert _{p}$. 
\end{lem}

Suppose now that $f$ is a conjugate invariant function. Then each
of the Fourier coefficients in (\ref{eq:Fourier coefficients}) is
a scalar matrix $\mathcal{F}(f)(\tau):=a_{\tau}\cdot I_{V_{\tau}}$,
where $a_{\tau}=\frac{1}{\dim\tau}\int_{H}f(h)\Theta_{\tau}(h^{-1})\mu_{H}$.
In this case, we have an equality of distributions 
\[
f=\sum_{\tau\in\Irr(H)}\dim\tau\cdot a_{\tau}\Theta_{\tau},
\]
and the condition $\mathcal{F}(f)\in\mathcal{H}_{q}(\hat{H})$ is
equivalent to:
\begin{equation}
\left\Vert \mathcal{F}(f)\right\Vert _{q}:=\left(\sum_{\tau\in\Irr(H)}\left(\dim\tau\right)^{2}\left|a_{\tau}\right|^{q}\right)^{\frac{1}{q}}<\infty.\label{eq:L^q norm of Fourier transform of a conjugate invariant function}
\end{equation}

\subsection{Upper bounds on multiplicities of $K$-types}

We assume $F$ is non-Archimedean. Let $\GG$ be a reductive algebraic
$F$-group, $G=\GG(F)$. 
\begin{lem}
\label{lem:restiction to maximal compact}Let $\pi\in\mathrm{Rep}(G)$,
let $K$ be an open compact subgroup of $G$, and let $\pi|_{K}=\sum_{\tau\in\Irr(K)}m_{\tau,\pi}\tau$
be the decomposition of $\pi$ into irreducible representations of
$K$, with $m_{\tau,\pi}:=\dim\mathrm{Hom}_{K}\left(\tau,\pi|_{K}\right)$.
Then we have the following equality of distributions: 
\[
\Theta_{\pi}|_{K}=\sum_{\tau\in\Irr(K)}m_{\tau,\pi}\Theta_{\tau}.
\]
\end{lem}

\begin{proof}
Let $f\in\mathcal{S}(K)$. Then there exists an open compact subgroup
$K'\subseteq K$ such that $f(K'kK')=f(k)$ for all $k\in K$. By
the admissibility of $\pi$ we have $\dim\pi^{K'}<\infty$, and hence
$\Theta_{\pi}(f):=\tr\left(\pi(f)|_{\pi^{K'}}\right)$. Let us denote
by $\Irr_{K'}(K)$ the finite set of irreducible representations of
$K$ which have $K'$-fixed vectors. Then:
\[
\Theta_{\pi}(f)=\tr\left(\pi(f)|_{\pi^{K'}}\right)=\sum_{\tau\in\Irr_{K'}(K)}m_{\tau,\pi}\tr\left(\tau(f)\right)=\sum_{\tau\in\Irr_{K'}(K)}m_{\tau,\pi}\Theta_{\tau}(f)=\sum_{\tau\in\Irr(K)}m_{\tau,\pi}\Theta_{\tau}(f).\qedhere
\]
\end{proof}
We can now apply Lemmas \ref{Lemma: Hausdorff--Young} and \ref{lem:restiction to maximal compact}
to the character of $\pi\in\mathrm{Rep}(G)$ and obtain: 
\begin{cor}[Theorem \ref{thmE:upper bounds on multiplicities of K-types}]
\label{cor:exponential bounds on multiplicities of K-types}Let $\pi\in\mathrm{Rep}(G)$
and let $K\leq G$ be an open compact subgroup. Then for every $\epsilon<\epsilon_{\star}(\pi)$,
there exists $C=C(\epsilon,K,\pi)>0$ such that for every $\tau\in\Irr(K)$:
\[
\dim\mathrm{Hom}_{K}\left(\tau,\pi|_{K}\right)<C\left(\dim\tau\right)^{\frac{1-\epsilon}{1+\epsilon}}.
\]
\end{cor}

\begin{proof}
By definition of $\epsilon_{\star}(\pi)$, if $\theta_{\pi}$ represents
$\Theta_{\pi}$, then $\theta_{\pi}\in L^{1+\epsilon}(K)$ for every
$\epsilon<\epsilon_{\star}(\pi)$ and hence $\mathcal{F}(\theta_{\pi})$
$\in\mathcal{H}_{\frac{1+\epsilon}{\epsilon}}(\hat{K}),$ so that:
\begin{equation}
\sum_{\tau\in\mathrm{Irr}(K)}\left(\dim\tau\right)^{2}\cdot(\frac{m_{\tau,\pi}}{\dim\tau})^{1+\frac{1}{\epsilon}}=\sum_{\tau\in\mathrm{Irr}(K)}(\dim\tau)^{1-\frac{1}{\epsilon}}(m_{\tau,\pi})^{1+\frac{1}{\epsilon}}<\widetilde{C},\label{eq:translation of Fourier to bounds on multipliity}
\end{equation}
for some $\widetilde{C}>0$. Hence, there exists $M\in\N$, such that
for every $\tau\in\mathrm{Irr}(K)$ with $\dim\tau>M$, one has $m_{\tau,\pi}<(\dim\tau)^{\frac{1-\epsilon}{1+\epsilon}}$.
Taking $C=M$, implies the corollary. 
\end{proof}
We now specialize to the special case of parabolically induced representations
in $\GL_{n}(F)$. Let $\PP\leq\GL_{n}$ be a standard parabolic, and
let
\[
\pi=\mathrm{Ind}_{\PP(F)}^{\GL_{n}(F)}1=\left\{ f:\GL_{n}(F)\rightarrow\C:f(gp)=\delta_{\PP(F)}^{-1/2}(p)f(g)\text{ for all }p\in\PP(F)\right\} ,
\]
be the (normalized) induction of the trivial character from $\PP(F)$,
where $\delta_{\PP(F)}$ is the modular character of $\PP(F)$.
\begin{lem}
\label{lemm:global character of parabolic induction}We have the following
equality of distributions:
\[
\Theta_{\pi}|_{\GL_{n}(\mathcal{O}_{F})}=\Theta_{\mathrm{Ind}_{\PP(\mathcal{O}_{F})}^{\GL_{n}(\mathcal{O}_{F})}1}=\sum_{\tau\in\Irr(\GL_{n}(\mathcal{O}_{F}))}m_{\tau,\pi}\Theta_{\tau}.
\]
\end{lem}

\begin{proof}
Note that $\PP(F)\GL_{n}(\mathcal{O}_{F})=\GL_{n}(F)$ so $\GL_{n}(\mathcal{O}_{F})$
acts transitively on the flag variety $\GL_{n}(F)/\PP(F)$. The stabilizer
is $\PP(F)\cap\GL_{n}(\mathcal{O}_{F})=\PP(\mathcal{O}_{F})$. Hence
$\GL_{n}(F)/\PP(F)\simeq\GL_{n}(\mathcal{O}_{F})/\PP(\mathcal{O}_{F})$
as $\GL_{n}(\mathcal{O}_{F})$-spaces. Since $\delta_{\PP(F)}|_{\PP(\mathcal{O}_{F})}=1$,
the map $f\mapsto f|_{\GL_{n}(\mathcal{O}_{F})}$ induces an isomorphism
$\pi|_{\GL_{n}(\mathcal{O}_{F})}\simeq\mathrm{Ind}_{\PP(\mathcal{O}_{F})}^{\GL_{n}(\mathcal{O}_{F})}1$.
The lemma now follows from Lemma \ref{lem:restiction to maximal compact}. 
\end{proof}
Corollary \ref{cor:upper bound for multiplicities in principal series}
is proved at the end of $\mathsection$\ref{subsec:Bounds-on-multiplicities in compact homogeneous}.
Here is a nice consequence of it. 
\begin{example}
\label{exam:analyzing maximal parabolics}Consider the case when $\PP_{\lambda}$
is a maximal parabolic, that is, when $\lambda=(n-m,m)$ for $m\leq n/2$.
Then the conjugate partition is $\nu=1^{n-2m}2^{m}$. We now show
that 
\begin{equation}
\epsilon_{\star}(\mathrm{Ind}_{\PP_{\lambda}(F)}^{\GL_{n}(F)}1)=1.\label{eq:epsilon for maximal parabolics}
\end{equation}
Indeed, for each $k',m\geq1$ we have:
\begin{align}
\binom{2m+k'}{2} & =\frac{(2m+k')(2m+k'-1)}{2}\leq2m^{2}+\frac{k'^{2}}{2}+2mk'-m-\frac{k'}{2}\leq2m^{2}+k'^{2}+2mk'-1\nonumber \\
 & =2\left(\binom{m+1}{2}+\binom{m}{2}\right)+\binom{k'+1}{2}+\binom{k'}{2}+2mk'-1.\label{eq:inequality}
\end{align}
Hence, by Theorem \ref{thmC: epsilon of Fourier of orbital integral for GL_n},
Corollary \ref{cor:upper bound for multiplicities in principal series}
and by (\ref{eq:inequality}), $\epsilon_{\star}(\mathrm{Ind}_{\PP_{\lambda}(F)}^{\GL_{n}(F)}1)$
is equal to:
\begin{align*}
 & \min\left\{ \underset{1\leq k\leq m}{\min}\frac{\left(\sum_{j=1}^{k}2j\right)-1}{\binom{2k}{2}-\sum_{j=1}^{k}2(j-1)},\underset{1\leq k'\leq n-2m}{\min}\frac{m^{2}+m-1+\sum_{j=1}^{k'}(m+j)}{\binom{2m+k'}{2}-2\binom{m}{2}-\sum_{j=1}^{k'}(j+m-1)}\right\} \\
= & \min\left\{ \underset{1\leq k\leq m}{\min}\left(1+\frac{k-1}{k^{2}}\right),\underset{1\leq k'\leq n-2m}{\min}\frac{2\binom{m+1}{2}+\binom{k'+1}{2}+mk'-1}{\binom{2m+k'}{2}-2\binom{m}{2}-\binom{k'}{2}-mk'}\right\} \\
= & \min\left\{ 1,\underset{1\leq k'\leq n-2m}{\min}\frac{2\binom{m+1}{2}+\binom{k'+1}{2}+mk'-1}{\binom{2m+k'}{2}-2\binom{m}{2}-\binom{k'}{2}-mk'}\right\} =1.
\end{align*}
In \cite{KK19}, Kionke and Klopsch introduced the following zeta
function, for each compact $p$-adic group $L$, and a closed subgroup
$H<L$,
\[
\zeta_{\mathrm{Ind}_{H}^{L}1}(s):=\sum_{\tau\in\Irr(L)}m_{\tau,\mathrm{Ind}_{H}^{L}1}(\dim\tau)^{-s}=\sum_{\tau\in\Irr(L)}\dim\tau^{H}\cdot(\dim\tau)^{-s}.
\]
By Lemma (\ref{Lemma: Hausdorff--Young}) and Eq.~(\ref{eq:epsilon for maximal parabolics}),
we deduce, for every $m\leq n/2$ and every $\lambda=(n-m,m)$, that
$\mathcal{F}(\Theta_{\mathrm{Ind}_{\PP_{\lambda}(\mathcal{O}_{F})}^{\GL_{n}(\mathcal{O}_{F})}1})\in\mathcal{H}_{2+s}(\widehat{\GL_{n}(\mathcal{O}_{F})})$
for all $s>0$. In particular, by (\ref{eq:translation of Fourier to bounds on multipliity}):
\[
\zeta_{\mathrm{Ind}_{\PP_{\lambda}(\mathcal{O}_{F})}^{\GL_{n}(\mathcal{O}_{F})}1}(s)=\sum_{\tau\in\Irr\left(\GL_{n}(\mathcal{O}_{F})\right)}m_{\tau,\mathrm{Ind}_{\PP_{\lambda}(\mathcal{O}_{F})}^{\GL_{n}(\mathcal{O}_{F})}1}\left(\dim\tau\right)^{-s}\leq\sum_{\tau\in\Irr\left(\GL_{n}(\mathcal{O}_{F})\right)}\left(\dim\tau\right)^{-s}m_{\tau,\mathrm{Ind}_{\PP_{\lambda}(\mathcal{O}_{F})}^{\GL_{n}(\mathcal{O}_{F})}1}^{2+s}<\infty,
\]
which implies that the abscissa of convergence of $\zeta_{\mathrm{Ind}_{\PP(\mathcal{O}_{F})}^{\GL_{n}(\mathcal{O}_{F})}1}$
is equal to $0$. This agrees with the results of \cite[Section 7.3]{KK19}. 
\end{example}

\subsection{\label{subsec:Bounds-on-multiplicities in compact homogeneous}Bounds
on multiplicities in compact homogeneous spaces}

In this subsection we prove Theorem \ref{thm F:bound on multilicites in compact homogeneous}
and Corollary \ref{cor:upper bound for multiplicities in principal series}.
We recall the notation from $\mathsection$\ref{subsec:Branching-multiplicities-in compact homogeneous};
let $\underline{G}$ be a complex, connected reductive algebraic group,
let $\underline{T}$ be a maximal torus. Let $\underline{L}$ be a
Levi subgroup of $\underline{G}$, containing $\underline{T}$. Let
$G_{\C}=\underline{G}(\C)$, $T_{\C}=\underline{T}(\C)$ and $L_{\C}=\underline{L}(\C)$.
Let $K\leq G_{\C}$ be a maximal compact subgroup in $G_{\C}$, containing
the maximal compact torus $T$ of $T_{\C}$. Recall that $L:=L_{\C}\cap K$
is called a \emph{Levi subgroup} of $K$. We write $\mu_{K},\mu_{L}$
and $\mu_{T}$ for the unique Haar probability measures on $K,L$
and $T$, respectively. 

Let $\tau\in\Irr(K)$. Note that $\tau\simeq\tau^{w}$ for every $w\in W(K,T)$,
and hence $\dim\tau^{L}=\dim\tau^{wLw^{-1}}$ for every $w\in W(K,T)$.
By applying the Weyl integration formula (\cite[Theorem 4.45]{Kna01})
twice, moving from $L$ to $T$, and then from $T$ to $K$, we get:
\begin{align}
\dim\tau^{L} & =\frac{\left|W(L,T)\right|}{\left|W(K,T)\right|}\sum_{w\in W(K,T)/W(L,T)}\dim\tau^{wLw^{-1}}=\frac{\left|W(L,T)\right|}{\left|W(K,T)\right|}\sum_{w\in W(K,T)/W(L,T)}\int_{L}\Theta_{\tau}(wlw^{-1})\mu_{L}\nonumber \\
 & =\frac{1}{\left|W(K,T)\right|}\sum_{w\in W(K,T)/W(L,T)}\int_{T}\Theta_{\tau}(wtw^{-1})\left|D_{L}(t)\right|\mu_{T}\nonumber \\
 & =\frac{1}{\left|W(K,T)\right|}\int_{T}\Theta_{\tau}(t)\sum_{w\in W(K,T)/W(L,T)}\left|D_{L}^{w}(t)\right|\mu_{T}\nonumber \\
 & =\frac{1}{\left|W(K,T)\right|}\int_{T}\Theta_{\tau}(t)\frac{\sum_{w\in W(K,T)/W(L,T)}\left|D_{L}^{w}(t)\right|}{\left|D_{K}(t)\right|}\left|D_{K}(t)\right|\mu_{T}=\int_{K}\Theta_{\tau}(k)\phi_{K,L}(k)\mu_{K},\label{eq:refined formula for multiplicity}
\end{align}
where $\phi_{K,L}$ is a conjugate invariant function on $K$, such
that $\phi_{K,L}|_{T}=\frac{\sum_{w\in W(K,T)/W(L,T)}\left|D_{L}^{w}(t)\right|}{\left|D_{K}(t)\right|}$.
By (\ref{eq:refined formula for multiplicity}), the Fourier coefficient
of $\phi_{K,L}$ at $\tau\in\Irr(K)$ is $\dim\tau^{L}$, so by Frobenius
reciprocity, we have the following equality, as distributions:
\begin{equation}
\phi_{K,L}=\sum_{\tau\in\Irr(K)}\dim\tau^{L}\Theta_{\tau}=\sum_{\tau\in\Irr(K)}m_{\tau,L^{2}(K/L)}\Theta_{\tau}=\Theta_{L^{2}(K/L)}.\label{eq:description of character of induced representation}
\end{equation}

\begin{proof}[Proof of Theorem \ref{thm F:bound on multilicites in compact homogeneous}]
It is enough to consider the case that $K$ is semisimple. Item (2)
follows from Item (1) by Lemma \ref{Lemma: Hausdorff--Young}, as
in the proof of Corollary \ref{cor:exponential bounds on multiplicities of K-types}.
It is left to show Item (1). Note that by Weyl's integration formula:
\[
\int_{K}\left|\phi_{K,L}\right|^{1+\epsilon}\mu_{K}=\frac{1}{\left|W(K,T)\right|}\int_{T}\frac{\left(\sum_{w\in W(K,T)/W(L/T)}\left|D_{L}^{w}\right|\right)^{1+\epsilon}}{\left|D_{K}\right|^{\epsilon}}\mu_{T}
\]
which is bounded if and only if $\int_{T}\frac{\left|D_{L}\right|^{1+\epsilon}}{\left|D_{K}\right|^{\epsilon}}\mu_{T}<\infty$.
Let us first compute $\epsilon_{\star}(\Theta_{L^{2}(K/L)};e)$. Since
$\exp:\t\rightarrow T$ is a diffeomorphism when restricted to a small
neighborhood $B\subseteq\t$ of $0$, we get 
\begin{equation}
\epsilon_{\star}(\Theta_{L^{2}(K/L)};e)=\underset{s>0}{\sup}\left\{ \int_{B}\frac{\left|D_{L}\circ\exp(X)\right|^{1+s}}{\left|D_{K}\circ\exp(X)\right|^{s}}\mu_{\t}<\infty\right\} .\label{eq:formula for epsilon at e}
\end{equation}
 Note that by (\ref{eq:comparing group Weyl discriminant to Lie algebra discriminant}),
both $\frac{\left|D_{K}\circ\exp(X)\right|}{\left|\triangle_{\mathfrak{g}_{\C}}^{\t_{\C}}\right|^{2}}$
and $\frac{\left|D_{L}\circ\exp(X)\right|}{\left|\triangle_{\l_{\C}}^{\t_{\C}}\right|^{2}}$
are invertible in a neighborhood of $0$. Hence, by (\ref{eq:formula for epsilon at e})
and Proposition \ref{prop:formula for log canonical threshold},
\begin{align}
\epsilon_{\star}(\Theta_{L^{2}(K/L)};e) & =\frac{1}{2}\lct\left(\triangle_{\g_{\C},\l_{\C}}^{\t_{\C}};\left(\triangle_{\l_{\C}}^{\t_{\C}}\right)^{2},0\right)=\frac{1}{2}\underset{\substack{\t_{\C}\varsubsetneq\l'\subseteq\g_{\C}\\{}
[\l',\l']\text{ is simple}
}
}{\min}\frac{\mathrm{ss.rk}(\l')+2\left|\Sigma^{+}(\l_{\C},\t)\cap\Sigma^{+}(\l',\t)\right|}{\left|\Sigma^{+}(\l',\t)\backslash\Sigma^{+}(\l_{\C},\t)\right|}\nonumber \\
 & =\frac{1}{2}\underset{\substack{\t_{\C}\varsubsetneq\l'\subseteq\g_{\C}\\
\l'\text{ is a Levi of }\g_{\C}
}
}{\min}\frac{\mathrm{ss.rk}(\l')+2\left|\Sigma^{+}(\l_{\C},\t)\cap\Sigma^{+}(\l',\t)\right|}{\left|\Sigma^{+}(\l',\t)\backslash\Sigma^{+}(\l_{\C},\t)\right|},\label{eq:refined formula for epsilon at e}
\end{align}
where the last inequality follows by applying Lemma \ref{lem:elementary lemma},
with $a_{i}=\mathrm{ss.rk}(\l_{i})+2\left|\Sigma^{+}(\l_{\C},\t)\cap\Sigma^{+}(\l_{i},\t)\right|$
and $b_{i}=\left|\Sigma^{+}(\l_{i},\t)\backslash\Sigma^{+}(\l_{\C},\t)\right|$,
for $i\in[N]$, where $\t\subseteq\l_{1},...,\l_{N}$ are Levi subalgebras
of $\g$, such that $[\l',\l']=\bigoplus_{i=1}^{N}[\l_{i},\l_{i}]$. 

Now, if $e\neq t'\in T\subseteq K$, denote by $S:=\mathrm{Cent}_{K}(t')$,
$S_{\C}:=\mathrm{Cent}_{G_{\C}}(t')$, and similarly $\mathfrak{s}:=\Lie(S)$
and $\mathfrak{s}_{\C}:=\Lie(S_{\C})$. Note that $\mathfrak{s}_{\C}\supseteq\t_{\C}$
is a maximal rank reductive subalgebra of $\g_{\C}$. Set $\Sigma^{+}(\mathfrak{s}_{\C},\t_{\C}):=\Sigma(\mathfrak{s}_{\C},\t_{\C})\cap\Sigma^{+}(\g_{\C},\t_{\C})$
and $\Sigma^{+}(\l_{\C}\cap\mathfrak{s}_{\C},\t_{\C}):=\Sigma(\l_{\C}\cap\mathfrak{s}_{\C},\t_{\C})\cap\Sigma^{+}(\l_{\C},\t_{\C})$.
Note that for each $\alpha\in\Sigma^{+}(\g_{\C},\t_{\C})$, if $\xi_{\alpha}:T\rightarrow\C^{\times}$
is the corresponding character of $T$, then $\xi_{\alpha}(t')=1$
if and only if $\alpha\in\Sigma^{+}(\mathfrak{s}_{\C},\t_{\C})$.
Thus, by (\ref{eq:comparing group Weyl discriminant to Lie algebra discriminant}),
in a small neighborhood of $t'\in T$, the function $\frac{\left|D_{K}(t)\right|}{\left|D_{S}(t)\right|}$
is invertible and similarly $\frac{\left|D_{L}(t)\right|}{\left|D_{S\cap L}(t)\right|}$
is invertible. Finally, observe that $\l_{\C}\cap\mathfrak{s}_{\C}$
is a Levi subalgebra of $\mathfrak{s}_{\C}$ and $L\cap S$ is a Levi
subgroup of $S$. In particular, by (\ref{eq:refined formula for epsilon at e})
we have, 
\[
\epsilon_{\star}(\Theta_{L^{2}(K/L)};t')=\epsilon_{\star}(\Theta_{L^{2}(S/L\cap S)};e)=\frac{1}{2}\underset{\substack{\t_{\C}\varsubsetneq\l''\subseteq\mathfrak{s}_{\C}\\
\l''\text{ is a Levi of }\mathfrak{s}_{\C}
}
}{\min}\frac{\mathrm{ss.rk}(\l'')+2\left|\Sigma^{+}(\l_{\C}\cap\mathfrak{s}_{\C},\t)\cap\Sigma^{+}(\l'',\t)\right|}{\left|\Sigma^{+}(\l'',\t)\backslash\Sigma^{+}(\l_{\C}\cap\mathfrak{s}_{\C},\t)\right|}.
\]
To finish the proof of Item (1), it is left to observe that the set
of all pseudo-Levi subalgebras $\t_{\C}\varsubsetneq\l'''\subseteq\g_{\C}$
coincides with the union of the set of Levi subalgebras $\t_{\C}\varsubsetneq\l''\subseteq\mathfrak{s}_{\C}$
of the pseudo-Levi subalgebras $\mathfrak{s}_{\C}\varsupsetneq\t_{\C}$
of $\g_{\C}$. Indeed, by Borel\textendash de Siebenthal theory \cite{BdS49},
the pseudo-Levi subalgebras $\mathfrak{s}_{\C}\supsetneq\t_{\C}$
of $\g_{\C}$ correspond to removing a subset of vertices in the extended
Dynkin diagram, and then Levi subalgebras $\l''\varsupsetneq\t_{\C}$
of $\mathfrak{s}_{\C}$ correspond to further removing vertices from
the resulting diagram, so by e.g.~\cite[Proposition 2]{Som98}, $\l''$
is a pseudo-Levi subalgebra of $\g_{\C}$. Conversely, by already
taking $\l''=\mathfrak{s}_{\C}$, one obtains all possible pseudo-Levi
subalgebras. This proves the theorem. 
\end{proof}
\begin{proof}[Proof of Corollary \ref{cor:upper bound for multiplicities in principal series}]
By Remark \ref{rem:A-field-dependent construction}, by Theorem \ref{thmC: epsilon of Fourier of orbital integral for GL_n},
and by Proposition \ref{prop:relating lct to relative lct in split groups},
we have
\[
\epsilon_{\star}(\mathrm{Ind}_{\PP(F)}^{\GL_{n}(F)}1;e)=\epsilon_{\star}(\widehat{\xi}_{\mathcal{O}_{\mathrm{st}}})=2\lct_{F}(D_{\gl_{n},\l_{\lambda}})=2\lct_{\C}(D_{\gl_{n},\l_{\lambda}}),
\]
where $\mathcal{O}_{\mathrm{st}}$ is the (stable) Richardson orbit
attached to $\PP(F)$. On the other hand, by \cite[Theorem 3]{vD72},
$\Theta_{\mathrm{Ind}_{\PP(F)}^{\GL_{n}(F)}1}$ is given as an average
of Weyl conjugates of $\left|\frac{D_{\GL_{n}(F)}}{D_{L_{\lambda}(F)}}\right|_{F}^{-\frac{1}{2}}$.
From the discussion in the proof of Theorem \ref{thmB:lower bound on epsilon of reps},
\[
\epsilon_{\star}(\mathrm{Ind}_{\PP(F)}^{\GL_{n}(F)}1)\geq\sup\left\{ s:\int_{B''}\frac{\left|D_{L_{\lambda}(\C)}\right|_{\C}^{\frac{1+s}{2}}}{\left|D_{\GL_{n}(\C)}\right|_{\C}^{\frac{s}{2}}}\mu_{\underline{T}(\C)}<\infty\text{ \,\,\,\,for each ball }B''\subseteq\underline{T}(\C)\right\} .
\]
By a similar argument as in the proof of Theorem \ref{thm F:bound on multilicites in compact homogeneous}
above, the right hand side is equal to $2\lct_{\C}(D_{\gl_{n},\l_{\lambda}})$.
Corollary \ref{cor:upper bound for multiplicities in principal series}
now follows from Corollary \ref{cor:exponential bounds on multiplicities of K-types}
and Lemma \ref{lemm:global character of parabolic induction}. 
\end{proof}

\subsection{Applications in random matrix theory- a proof of Theorem \ref{thmG:intergability of power measure}}

Let $w_{\ell}:\U_{n}\rightarrow\U_{n}$ be the $\ell$-th power map
$X\mapsto X^{\ell}$, for $\ell\in\N$. The distribution of the $\ell$-th
power of a random unitary matrix is given by the pushforward measure
$\tau_{\ell,n}:=(w_{\ell})_{*}\mu_{\U_{n}}$. Recall that: 
\[
\tau_{\ell,n}=\sum_{\rho\in\Irr(\U_{n})}\overline{a_{n,\ell,\rho}}\cdot\Theta_{\tau},\text{\,\,\, where \,\,\,\,}a_{n,\ell,\rho}:=\int_{\U_{n}}\Theta_{\rho}(X^{\ell})\mu_{\U_{n}}.
\]
 Let $H_{n,\ell}:=\U_{\left\lfloor n/\ell\right\rfloor +1}^{j}\times\U_{\left\lfloor n/\ell\right\rfloor }^{\ell-j}$,
where $j:=n\mod\ell$ and set $\widetilde{\ell}:=\min(\ell,n)$. The
following is a consequence of the works of Rains \cite{Rai97,Rai03}: 
\begin{prop}[{\cite[Proposition 7.3]{AGL}}]
\label{prop:Fourier coefficients of a power}For every $n\geq2$,
every $\ell\geq1$ and every $\rho\in\Irr(\U_{n})$, we have: 
\[
a_{n,\ell,\rho}=\int_{\U_{n}}\Theta_{\rho}(X^{\ell})\mu_{\U_{n}}=\dim\rho^{H_{n,\ell}}.
\]
\end{prop}

We now prove Theorem \ref{thmG:intergability of power measure}. 
\begin{proof}[Proof of Theorem \ref{thmG:intergability of power measure}]
By the discussion in $\mathsection$\ref{subsec:Applications-in-Random},
and in particular by (\ref{eq:pushforward is a character}), we have
$\tau_{\ell,n}=\Theta_{L^{2}(\U_{n}/H_{n,\ell})}$. Since the Levi
subgroup $H_{n,\ell}$ is of the form $H_{n,\ell}:=L_{\C}\cap\U_{n}$,
where $\l_{\C}:=\Lie(L_{\C})=\gl_{\left\lfloor n/\ell\right\rfloor +1}^{j}\times\gl_{\left\lfloor n/\ell\right\rfloor }^{\ell-j}$
is a Levi subalgebra of $\g_{\C}=\gl_{n}$, the theorem now follows
from Corollary \ref{cor:epsilon for Levis in SU_n}.
\end{proof}
Corollary \ref{cor:Fourier coefficients of power word measure} and
Young's convolution inequality imply the following\footnote{See discussion at the end of p.3 of \cite{GHS}}: 
\begin{cor}
\label{cor:Mixing for power word}For every $n\geq2$ and $\ell\geq2$: 
\begin{enumerate}
\item The measure $\tau_{\ell,n}^{*\widetilde{\ell}+1}$ is in $L^{q}(\U_{n})$,
for all $1\leq q<\infty$. 
\item The measure $\tau_{\ell,n}^{*\widetilde{\ell}+2}$ has bounded density. 
\end{enumerate}
\end{cor}

\begin{rem}
Note that in \cite[Proposition 7.5]{AGL}, we showed that whenever
$n\gg_{\ell}1$, the measure $\tau_{\ell,n}^{*\ell-1}$ does not have
bounded density. This shows that Corollary \ref{cor:Mixing for power word}
is very close to being (and potentially is) tight for large $n$. 
\end{rem}

\appendix

\section{\label{sec:Integrability-exponents-of}Integrability exponent of
Fourier transform of nilpotent orbital integrals in $\protect\gl_{10}$}

\global\long\def\thefigure{\arabic{figure}}%
In the figure below we include calculation of $\epsilon_{\star}(\widehat{\xi}_{\mathcal{O}})$
for each nilpotent orbit in $\gl_{10}$. Each box represents a nilpotent
orbit in $\gl_{10}$, and includes the partition of the orbit (in
black) and the value of $\epsilon_{\star}(\widehat{\xi}_{\mathcal{O}})$
(in blue). An arrow is drawn from Box $A$ to Box $B$ if $\mathcal{O}_{B}\subset\overline{\mathcal{O}}_{A}$.
One can note that this diagram indeed respects the relations given
in Theorem \ref{thmC: epsilon of Fourier of orbital integral for GL_n}(2):
smaller orbits have larger values of $\epsilon_{\star}(\widehat{\xi}_{\mathcal{O}})$.
In addition, as shown in Example \ref{exa:some examples of epsilon},
the value $\epsilon_{\star}(\widehat{\xi}_{\mathcal{O}})$ for all
orbits whose polarizing parabolic is maximal is $1$ (orbits whose
partition only has parts of size $1$ or $2$), and the values of
$\epsilon_{\star}(\widehat{\xi}_{\mathcal{O}})$ for the regular and
subregular orbits are as expected. 
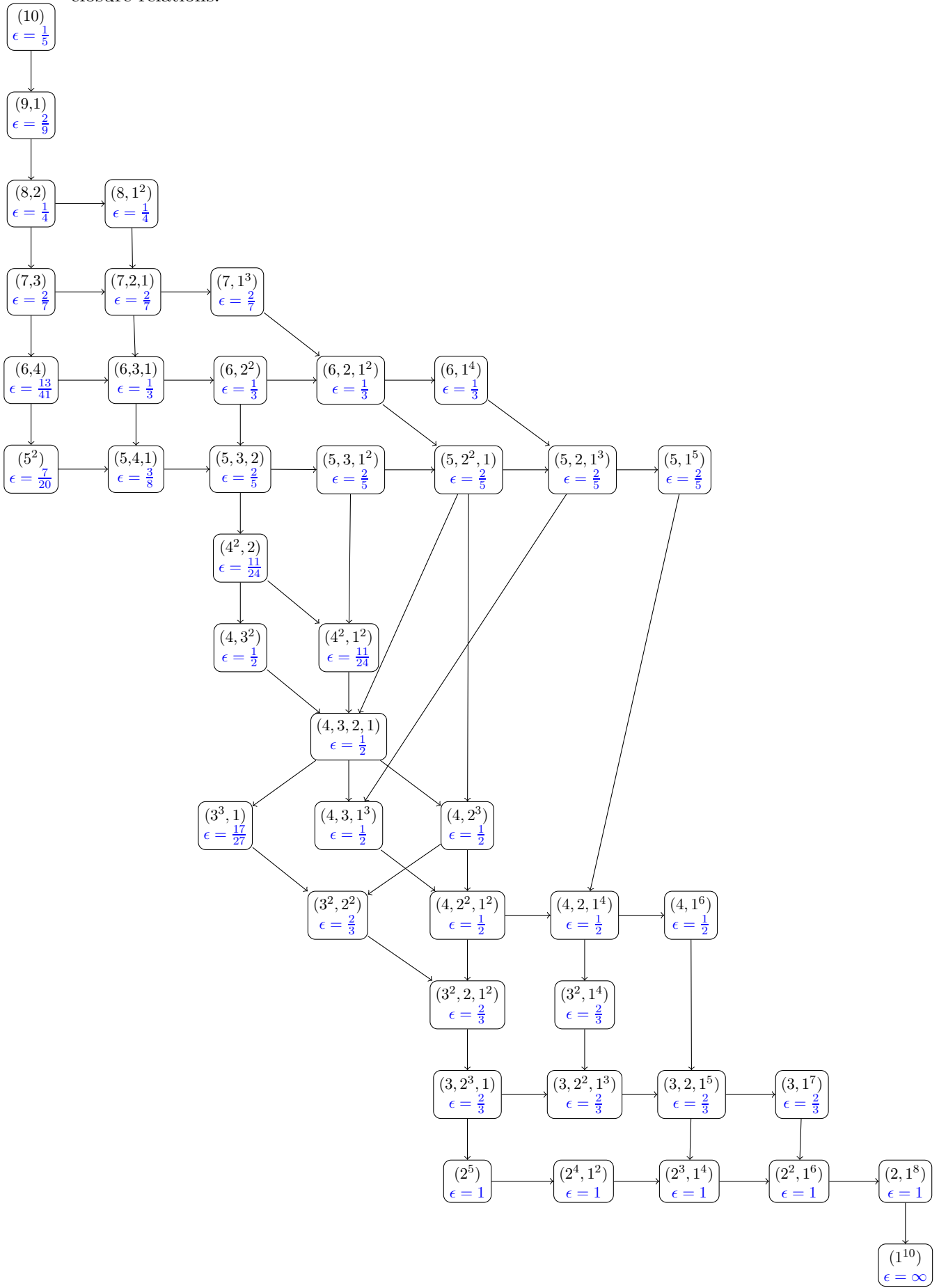
\begin{figure}
\label{fig: epsilon calculation for gl10} \centering
\begin{centering}
\caption{Calculation of $\epsilon_{\star}(\widehat{\xi}_{\mathcal{O}})$ for
the nilpotent orbits in $\protect\gl_{10}$ together with orbit closure
relations.}
\par\end{centering}
\begin{tikzpicture}[
    scale=0.78,
    transform shape,
    node distance=1cm and 1cm,
    main/.style = draw, rectangle, rounded corners, minimum height=2em, minimum width=3em, align=center
]

\node[main] (p10) {(10)\\ \color{blue}{$\epsilon=\frac{1}{5}$}};
\node[main] (p91) [below=of p10] {(9,1)\\ \color{blue}{$\epsilon=\frac{2}{9}$}};
\node[main] (p82) [below=of p91] {(8,2)\\ \color{blue}{$\epsilon=\frac{1}{4}$}};
\node[main] (p811) [right=1.2cm of p82] {$(8,1^2)$ \\ \color{blue}{$\epsilon=\frac{1}{4}$}};
\node[main] (p73) [below=of p82] {(7,3)\\ \color{blue}{$\epsilon=\frac{2}{7}$}};
\node[main] (p721) [right=1.2cm of p73] {(7,2,1)\\ \color{blue}{$\epsilon=\frac{2}{7}$}};
\node[main] (p7111) [right=1.2cm of p721] {$(7,1^3)$\\ \color{blue}{$\epsilon=\frac{2}{7}$}};
\node[main] (p64) [below=of p73] {(6,4)\\ \color{blue}{$\epsilon=\frac{13}{41}$}};
\node[main] (p631) [right=1.2cm of p64] {(6,3,1)\\ \color{blue}{$\epsilon=\frac{1}{3}$}};
\node[main] (p622) [right=1.2cm of p631] {($6,2^2$)\\ \color{blue}{$\epsilon=\frac{1}{3}$}};
\node[main] (p6211) [right=1.2cm of p622] {$(6,2,1^2)$\\ \color{blue}{$\epsilon=\frac{1}{3}$}};
\node[main] (p61111) [right=1.2cm of p6211] {$(6,1^4)$\\ \color{blue}{$\epsilon=\frac{1}{3}$}};
\node[main] (p55) [below=of p64] {($5^2$)\\ \color{blue}{$\epsilon=\frac{7}{20}$}};
\node[main] (p541) [right=1.2cm of p55] {(5,4,1)\\ \color{blue}{$\epsilon=\frac{3}{8}$}};
\node[main] (p532) [below=of p622] {($5,3,2$)\\ \color{blue}{$\epsilon=\frac{2}{5}$}};
\node[main] (p5311) [below=of p6211] {$(5,3,1^2)$\\ \color{blue}{$\epsilon=\frac{2}{5}$}};
\node[main] (p5221) [right=1.2cm of p5311] {$(5,2^2,1)$\\ \color{blue}{$\epsilon=\frac{2}{5}$}};
\node[main] (p52111) [right=1.1cm of p5221] {$(5,2,1^3)$\\ \color{blue}{$\epsilon=\frac{2}{5}$}};
\node[main] (p511111) [right=of p52111] {$(5,1^5)$\\ \color{blue}{$\epsilon=\frac{2}{5}$}};
\node[main] (p442) [below=of p532] {($4^2,2$)\\ \color{blue}{$\epsilon=\frac{11}{24}$}};
\node[main] (p433) [below=of p442] {($4,3^2$)\\ \color{blue}{$\epsilon=\frac{1}{2}$}};
\node[main] (p4411)[right=1.25cm of p433]  {($4^2,1^2$)\\ \color{blue}{$\epsilon=\frac{11}{24}$}};
\node[main] (p4321) [below=of p4411] {($4,3,2,1$)\\ \color{blue}{$\epsilon=\frac{1}{2}$}};
\node[main] (p43111) [below=of p4321] {($4,3,1^3$)\\ \color{blue}{$\epsilon=\frac{1}{2}$}};
\node[main] (p4222) [right=1.4cm of p43111] {($4,2^3$)\\ \color{blue}{$\epsilon=\frac{1}{2}$}};
\node[main] (p42211) [below=of p4222] {($4,2^2,1^2$)\\ \color{blue}{$\epsilon=\frac{1}{2}$}};
\node[main] (p421111) [right=1.1cm of p42211] {($4,2,1^4$)\\ \color{blue}{$\epsilon=\frac{1}{2}$}};
\node[main] (p4111111) [right=1.1cm of p421111] {($4,1^6$)\\ \color{blue}{$\epsilon=\frac{1}{2}$}};
\node[main] (p33211) [below=of p42211] {($3^2,2,1^2$)\\ \color{blue}{$\epsilon=\frac{2}{3}$}};
\node[main] (p3322) [left=1.5cm of p42211] {($3^2,2^2$)\\ \color{blue}{$\epsilon=\frac{2}{3}$}};
\node[main] (p3331) [left=1.5cm of p43111] {($3^3,1$)\\ \color{blue}{$\epsilon=\frac{17}{27}$}};
\node[main] (p32221) [below=of p33211] {($3,2^3,1$)\\ \color{blue}{$\epsilon=\frac{2}{3}$}};
\node[main] (p331111) [below=of p421111] {($3^2,1^4$)\\ \color{blue}{$\epsilon=\frac{2}{3}$}};
\node[main] (p322111) [right=1.1cm of p32221] {($3,2^2,1^3$)\\ \color{blue}{$\epsilon=\frac{2}{3}$}};
\node[main] (p3211111) [right=0.85cm of p322111] {($3,2,1^5$)\\ \color{blue}{$\epsilon=\frac{2}{3}$}};
\node[main] (p31111111) [right=1.2cm of p3211111] {($3,1^7$)\\ \color{blue}{$\epsilon=\frac{2}{3}$}};
\node[main] (p22222) [below=of p32221] {($2^5$)\\ \color{blue}{$\epsilon=1$}};
\node[main] (p222211) [right=1.5cm of p22222] {($2^4,1^2$)\\ \color{blue}{$\epsilon=1$}};
\node[main] (p2221111) [right=1.1cm of p222211] {($2^3,1^4$)\\ \color{blue}{$\epsilon=1$}};
\node[main] (p22111111) [right=1.2cm of p2221111] {($2^2,1^6$)\\ \color{blue}{$\epsilon=1$}};
\node[main] (p211111111) [right=1.2cm of p22111111] {($2,1^8$)\\ \color{blue}{$\epsilon=1$}};
\node[main] (p1111111111) [below=of p211111111] {($1^{10}$)\\ \color{blue}{$\epsilon=\infty$}};

\draw[->] (p10) -- (p91);
\draw[->] (p91) -- (p82);
\draw[->] (p82) -- (p811);
\draw[->] (p82) -- (p73);
\draw[->] (p811) -- (p721);
\draw[->] (p73) -- (p721);
\draw[->] (p721) -- (p7111);
\draw[->] (p73) -- (p64);
\draw[->] (p721) -- (p631);
\draw[->] (p7111) -- (p6211);
\draw[->] (p64) -- (p55);
\draw[->] (p64) -- (p631);
\draw[->] (p631) -- (p622);
\draw[->] (p622) -- (p6211);
\draw[->] (p6211) -- (p61111);
\draw[->] (p631) -- (p541);
\draw[->] (p622) -- (p532);
\draw[->] (p6211) -- (p5221);
\draw[->] (p61111) -- (p52111);
\draw[->] (p55) -- (p541);
\draw[->] (p541) -- (p532);
\draw[->] (p532) -- (p5311);
\draw[->] (p532) -- (p442);
\draw[->] (p5311) -- (p5221);
\draw[->] (p5311) -- (p4411);
\draw[->] (p5221) -- (p4321);
\draw[->] (p5221) -- (p52111);
\draw[->] (p5221) -- (p4222);
\draw[->] (p52111) -- (p511111);
\draw[->] (p52111) -- (p43111);
\draw[->] (p511111) -- (p421111);
\draw[->] (p442) -- (p4411);
\draw[->] (p442) -- (p433);
\draw[->] (p4411) -- (p4321);
\draw[->] (p433) -- (p4321);
\draw[->] (p4321) -- (p4222);
\draw[->] (p4321) -- (p3331);
\draw[->] (p4321) -- (p43111);
\draw[->] (p43111) -- (p42211);
\draw[->] (p4222) -- (p42211);
\draw[->] (p4222) -- (p3322);
\draw[->] (p42211) -- (p421111);
\draw[->] (p42211) -- (p33211);
\draw[->] (p421111) -- (p331111);
\draw[->] (p421111) -- (p4111111);
\draw[->] (p4111111) -- (p3211111);
\draw[->] (p3331) -- (p3322);
\draw[->] (p3322) -- (p33211);
\draw[->] (p33211) -- (p32221);
\draw[->] (p331111) -- (p322111);
\draw[->] (p32221) -- (p322111);
\draw[->] (p32221) -- (p22222);
\draw[->] (p3211111) -- (p2221111);
\draw[->] (p31111111) -- (p22111111);
\draw[->] (p322111) -- (p3211111);
\draw[->] (p3211111) -- (p31111111);
\draw[->] (p22222) -- (p222211);
\draw[->] (p222211) -- (p2221111);
\draw[->] (p2221111) -- (p22111111);
\draw[->] (p22111111) -- (p211111111);
\draw[->] (p211111111) -- (p1111111111);

\end{tikzpicture}
\end{figure}

\newpage{}

\bibliographystyle{alpha}
\bibliography{bibfile}

\end{document}